\documentclass[a4paper,reqno,11pt]{amsart}
\usepackage{amsmath,amsfonts,amssymb,amsthm}
\usepackage{bbm}
\usepackage[dvipsnames]{xcolor}
\usepackage[utf8]{inputenc}
\usepackage[T1]{fontenc}
\usepackage[margin = 1in]{geometry}
\usepackage{mathtools}
\usepackage{bbm}
\usepackage{enumitem}
\usepackage{textcomp}
\usepackage{graphicx}
\usepackage{float}
\usepackage{subcaption}
\usepackage{xcolor}
\usepackage{enumitem}
\usepackage[export]{adjustbox}
\usepackage{verbatim}
\usepackage{mathrsfs}
\usepackage{cancel}
\usepackage{wrapfig}

\usepackage{urwchancal}
\DeclareFontFamily{OT1}{pzc}{}
\DeclareFontShape{OT1}{pzc}{m}{it}%
              {<-> s * [1.15] pzcmi7t}{}
\DeclareMathAlphabet{\mathpzc}{OT1}{pzc}{m}{it}

\usepackage{tikz}

\usepackage[toc,page]{appendix}
\usepackage{cite}

\usepackage[hyphens]{url}
\usepackage[hidelinks, colorlinks, allcolors = black]{hyperref}
\usepackage{hyperref}

\makeatletter
\def\user@resume{resume}
\def\user@intermezzo{intermezzo}
\newcounter{previousequation}
\newcounter{lastsubequation}
\newcounter{savedparentequation}
\renewenvironment{subequations}[1][]{%
      \def\user@decides{#1}%
      \setcounter{previousequation}{\value{equation}}%
      \ifx\user@decides\user@resume
           \setcounter{equation}{\value{savedparentequation}}%
      \else
      \ifx\user@decides\user@intermezzo
           \refstepcounter{equation}%
      \else
           \setcounter{lastsubequation}{0}%
           \refstepcounter{equation}%
      \fi\fi
      \protected@edef\theHparentequation{%
          \@ifundefined {theHequation}\theequation \theHequation}%
      \protected@edef\theparentequation{\theequation}%
      \setcounter{parentequation}{\value{equation}}%
      \ifx\user@decides\user@resume
           \setcounter{equation}{\value{lastsubequation}}%
         \else
           \setcounter{equation}{0}%
      \fi
      \def\theequation  {\theparentequation  \alph{equation}}%
      \def\theHequation {\theHparentequation \alph{equation}}%
      \ignorespaces
}{%
  \ifx\user@decides\user@resume
       \setcounter{lastsubequation}{\value{equation}}%
       \setcounter{equation}{\value{previousequation}}%
  \else
  \ifx\user@decides\user@intermezzo
       \setcounter{equation}{\value{parentequation}}%
  \else
       \setcounter{lastsubequation}{\value{equation}}%
       \setcounter{savedparentequation}{\value{parentequation}}%
       \setcounter{equation}{\value{parentequation}}%
  \fi\fi
  \ignorespacesafterend
}
\makeatother

\makeatletter
\newcommand{\mylabel}[2]{#2\def\@currentlabel{#2}\label{#1}}
\makeatother

\makeatletter
\newcommand{\proofstep}[1]{%
  \par
  \addvspace{\medskipamount}
  \textit{#1\@addpunct{.}}\enspace\ignorespaces
}
\makeatother

\renewcommand{\d}{\ensuremath{\mathrm{d}}}
\newcommand{\C}{\ensuremath{\mathbb{C}}}
\newcommand{\K}{\ensuremath{\mathbb{K}}}

\newcommand{\N}{\ensuremath{\mathbb{N}}}

\newcommand{\R}{\ensuremath{\mathbb{R}}}

\newcommand{\Z}{\ensuremath{\mathbb{Z}}}
\def\R{\mathbb{R}}
\def\C{\mathbb{C}}
\def\N{\mathbb{N}}

\newcommand{\NN}{\ensuremath{\mathcal N}}

\newcommand{\SSS}{\ensuremath{\mathcal S}}
\newcommand{\TT}{\ensuremath{\mathcal T}}

\newtheorem{theorem}{Theorem}[section]

\newtheorem{proposition}[theorem]{Proposition}

\newtheorem{lemma}[theorem]{Lemma}

\numberwithin{equation}{section}
\newcommand{\Dualpair}[2]{\ensuremath{\left\langle \kern-0.5ex \left\langle #1,#2 \right\rangle \kern-0.5ex \right\rangle}}

\newcommand{\floor}[1]{\ensuremath{\lfloor #1 \rfloor}}
\newcommand{\ceil}[1]{\ensuremath{\lceil #1 \rceil}}
\DeclareMathOperator{\ind}{\ensuremath{\mathbbm{1}}}

\newcommand{\qvar}[2]{\ensuremath{\left\langle \kern-0.5ex \left\langle #1 \right\rangle \kern-0.5ex \right\rangle_{#2}}}

\newcommand{\eps}{\ensuremath{\varepsilon}}
\newcommand{\verti}[1]{\ensuremath{\lvert #1 \rvert}}
\newcommand{\vertii}[1]{\ensuremath{\lVert #1 \rVert}}
\newcommand{\vertiii}[1]{{\verti{\kern-0.25ex\verti{\kern-0.25ex\verti{ #1
    }\kern-0.25ex}\kern-0.25ex}}}

\newcommand{\MMM}{\mathcal{M}}

\newcommand{\alp}{\alpha}

\newcommand{\alpm}{\alpha -1}

\allowdisplaybreaks

\newcommand\blfootnote[1]{%
  \begingroup
  \renewcommand\thefootnote{}\footnote{#1}%
  \addtocounter{footnote}{-1}%
  \endgroup
}

\begin{document}
%
\title[Classical solutions to the thin-film equation with general mobility]{Classical solutions to the thin-film equation with general mobility in the perfect-wetting regime}
\thanks{The first author is indebted to \emph{Dominik John} for fruitful discussions on this problem approximately ten years ago in particular regarding suitable coordinate transformations. The authors thank \emph{Max Sauerbrey} for discussions at an early stage of this project. The first author is grateful to \emph{Michele Precuzzi} and the second author is grateful to \emph{Floris Roodenburg} for discussions. The authors thank \emph{Francisco Carvalho}, \emph{Emiel Lorist}, and \emph{Floris Roodenburg} for careful readings of the manuscript. The authors acknowledge the \emph{Lorentz Center} in Leiden for hosting the workshop \emph{Analysis and numerics of nonlinear PDEs: degeneracies \& free boundaries}, where fruitful discussions related to this project have taken place. Several suggestions of the anonymous reviewers have led to improvements of presentation and content of this revised version. This work is partially based on the second author's master's thesis in applied mathematics prepared under the advice of the first author at \emph{Delft University of Technology}. This publication is part of the project \emph{Codimension two free boundary problems} (with project number \emph{VI.Vidi.223.019} of the research program \emph{ENW Vidi}) which is financed by the \emph{Dutch Research Council} (\emph{NWO})}
\date{\today}
\author{Manuel V. Gnann}
\address
{Delft Institute of Applied Mathematics, Faculty of Electrical Engineering, Mathematics and Computer Science, Delft University of Technology, Mekelweg 4, 2628 CD Delft, Netherlands}
\email[Manuel~V.~Gnann]{M.V.Gnann@tudelft.nl}
\author{Anouk C. Wisse}
\email[Anouk~C.~Wisse]{A.C.Wisse@tudelft.nl}
\begin{abstract}
We prove well-posedness, partial regularity, and stability of the thin-film equation $h_t + (m(h) h_{zzz})_z = 0$ with general mobility $m(h) = h^n$ and mobility exponent $n\in (1,\tfrac{3}{2})\cup (\tfrac{3}{2},3)$ in the regime of perfect wetting (zero contact angle). After a suitable coordinate transformation to fix the free boundary (the contact line where liquid, air, and solid coalesce), the thin-film equation is rewritten as an abstract Cauchy problem and we obtain maximal $L^{p}_t$-regularity for the linearized evolution. Partial regularity close to the free boundary is obtained by studying the elliptic regularity of the spatial part of the linearization. This yields solutions that are non-smooth in the distance to the free boundary, in line with previous findings for source-type self-similar solutions. In a scaling-wise quasi-minimal norm for the initial data, we obtain a well-posedness and asymptotic stability result for perturbations of traveling waves. The novelty of this work lies in the usage of $L^{p}$-estimates in time, where $1 < p < \infty$, while the existing literature mostly deals with $p = 2$ at least for nonlinear mobilities. This turns out to be essential to obtain for the first time a well-posedness result in the perfect-wetting regime for all physical nonlinear slip conditions except for a strongly degenerate case at $n = \tfrac 3 2$ and the well-understood Greenspan-slip case $n = 1$. Furthermore, compared to [J. Differential Equations, 257(1):15-81, 2014] by Giacomelli, the first author of this paper, Kn\"upfer, and Otto, where a PDE approach yields $L^2_t$-estimates, well-posedness, and stability for $1.8384 \approx \tfrac{3}{17}(15-\sqrt{21}) < n < \tfrac{3}{11}(7+\sqrt{5}) \approx 2.5189$, our functional-analytic approach is significantly shorter while at the same time giving a  more general result.
\end{abstract}
\maketitle

\blfootnote{
\includegraphics[height=15mm]{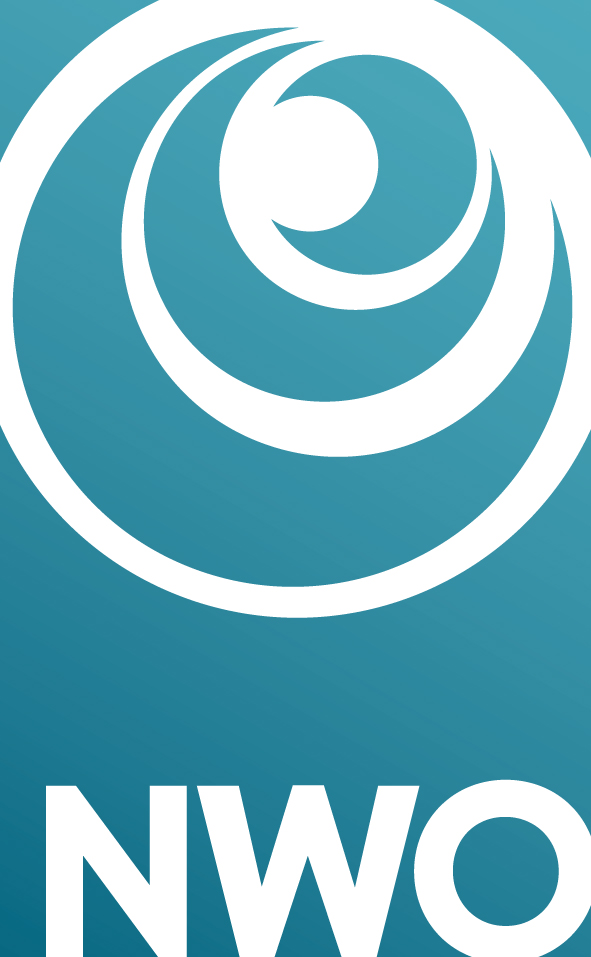}
}

\tableofcontents
%

\section{Introduction}
\subsection{The Thin-Film Equation formulated as a Classical Free-Boundary Problem}
We consider the following free-boundary problem to the thin-film equation
\begin{subequations}\label{TFE}
\begin{align}
    h_t+(h^nh_{zzz})_z &= 0 &&\text{for }t>0,\;z>Z(t),\label{TFEa}\\
    h = h_z &= 0 &&\text{for } t>0,\; z = Z(t),\label{TFEb}\\
    \lim_{z\downarrow Z(t)} h^{n-1}h_{zzz} &= Z_t(t) &&\text{for }t>0,\label{TFEc}
\end{align}
\end{subequations}
describing the time evolution of the height $h(t,z)$ of a viscous thin film on a one-dimensional flat substrate, as visualized in Figure~\ref{fig:thin_film}. Here, the independent variables $t$ and $z$ denote time and lateral position, respectively. The fluid covers the interval $(Z(t),\infty)$, where the free boundary $z = Z(t)$ is called the triple junction or contact line since this is the point where gas, liquid, and solid coalesce. The thin-film equation \eqref{TFEa} can be derived from the Navier-Stokes equations via a lubrication approximation \cite{Bonn2009,Gennes1985,GuentherProkert2008,Oron1997}, and the particular case $n = 1$ can be interpreted either as Greenspan's slip condition \cite{Greenspan1978} or as the lubrication approximation of Darcy's law in the Hele-Shaw cell \cite{GiacomelliOtto2003,KnuepferMasmoudi2013,KnuepferMasmoudi2015,MatiocProkert2012}.
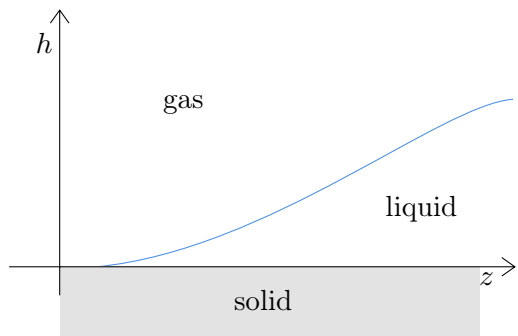
\begin{figure}[htb]
    \centering
\begin{tikzpicture}[x=0.7pt,y=0.7pt,yscale=-1,xscale=1]
\draw   [color={rgb, 255:red, 74; green, 144; blue, 226 }] (97.05,154.63) .. controls (189.5,145.33) and (280.5,67.33) .. (319.5,64.33) ;
\draw  [draw opacity=0][fill={rgb, 255:red, 227; green, 227; blue, 227 }  ,fill opacity=1 ] (77.05,154.63) -- (301.5,154.63) -- (301.5,194.63) -- (77.05,194.63) -- cycle ;
\draw  (50,154.63) -- (320.5,154.63)(77.05,16.33) -- (77.05,170) (313.5,149.63) -- (320.5,154.63) -- (313.5,159.63) (72.05,23.33) -- (77.05,16.33) -- (82.05,23.33)  ;

\draw (169,166) node [anchor=north west][inner sep=0.75pt]   [align=left] {solid};
\draw (250,115) node [anchor=north west][inner sep=0.75pt]   [align=left] {liquid};
\draw (131,61) node [anchor=north west][inner sep=0.75pt]   [align=left] {gas};
\draw (300,156.4) node [anchor=north west][inner sep=0.75pt]    {$z$};
\draw (63,27.4) node [anchor=north west][inner sep=0.75pt]    {$h$};
\end{tikzpicture}
\caption{A viscous thin film as described by \eqref{TFE}.}
\label{fig:thin_film}
\end{figure}

The exponent $n$ in \eqref{TFEa} is called the mobility exponent and takes values in the interval $[1,3]$. The value of this exponent is related to choice of slip, that is, the boundary condition at the liquid-solid interface in the underlying Navier-Stokes system. We will focus on the cases $1 < n < \tfrac{3}{2}$ and $\tfrac{3}{2} < n < 3$. This is because for $n = 3$ (no slip) or $n > 3$, the free boundary of the film cannot move \cite{Huh1971}, while for $n<0$ the propagation speed is infinite and if $0<n<1$ the height of the film can become negative \cite{BW}. The case of linear mobilities $n = 1$ is by now well understood (see details on references below). Our choice of exponents in particular includes the case of linear Navier slip \cite{Bonn2009,Gennes1985,Navier1823,Oron1997}. We exclude $n=\tfrac{3}{2}$ due to resonances leading to logarithmic corrections in this case \cite{Bernis1992}.

\medskip

The boundary condition $h=0$ at $z=Z(t)$ in \eqref{TFEb} determines the position of the contact line, while $h_z=0$ at $z=Z(t)$ entails that the contact angle $\theta$ between the liquid-solid and the gas-liquid interfaces at the triple junction is zero. This implies that the fluid will eventually cover the entire solid. The latter is evident from Young's law \cite{Young1805}, that is,
\begin{equation}\label{Young}
    \gamma_\mathrm{gs}=\gamma_\mathrm{ls}+\cos \theta\gamma_\mathrm{gl},
\end{equation}
where $\gamma_\mathrm{gs}, \gamma_\mathrm{ls}$, and $\gamma_\mathrm{gl}$ are the surface tensions between the gas-solid, liquid-solid, and gas-liquid interfaces, respectively (see Figure~\ref{fig:young's_equation}).

\begin{figure}[htp]
    \centering

\tikzset{every picture/.style={line width=0.75pt}}
\begin{tikzpicture}[x=0.75pt,y=0.75pt,yscale=-1,xscale=1]

\draw  [draw opacity=0][fill={rgb, 255:red, 208; green, 205; blue, 205 }  ,fill opacity=1 ] (69.5,111) -- (325.5,111) -- (325.5,143) -- (69.5,143) -- cycle ;
\draw [color={rgb, 255:red, 74; green, 144; blue, 226 }  ,draw opacity=1 ]   (176.5,111) .. controls (259.5,21) and (333.5,39.33) .. (323.5,36.33) ;
\draw  [draw opacity=0][dash pattern={on 4.5pt off 4.5pt}] (191.46,89.87) .. controls (195.74,90.82) and (199.63,93.69) .. (201.83,98.02) .. controls (203.72,101.75) and (203.99,105.85) .. (202.91,109.48) -- (188.86,104.6) -- cycle ;
\draw[dash pattern={on 2.5pt off 2.5pt},domain=316:360] plot ({175+30*cos(\x)}, {110+30*sin(\x)});
\draw    (69.5,111) -- (325.5,111) ;
\draw [color={rgb, 255:red, 255; green, 0; blue, 0 }  ,draw opacity=1 ][fill={rgb, 255:red, 255; green, 0; blue, 0 }  ,fill opacity=1 ][-stealth]   (176.5,111) -- (235,49.01) ;
\draw [color={rgb, 255:red, 255; green, 0; blue, 0 }  ,draw opacity=1 ][-stealth]    (176.5,111) -- (84.5,111.33) ;
\draw [color={rgb, 255:red, 255; green, 0; blue, 0 }  ,draw opacity=1 ][-stealth]    (176.5,111) -- (231.5,111.32) ;

\draw (241,121) node [anchor=north west][inner sep=0.75pt]   [align=left] {solid};
\draw (280,67) node [anchor=north west][inner sep=0.75pt]   [align=left] {liquid};
\draw (109,44) node [anchor=north west][inner sep=0.75pt]   [align=left] {gas};
\draw (191.5,96.4) node [anchor=north west][inner sep=0.75pt]  [font=\footnotesize]  {${\textstyle \theta }$};
\draw (86,90.4) node [anchor=north west][inner sep=0.75pt]  [font=\small]  {$\gamma _\mathrm{gs}$};
\draw (190,50.4) node [anchor=north west][inner sep=0.75pt]  [font=\small]  {$\gamma _\mathrm{gl}$};
\draw (215,91.4) node [anchor=north west][inner sep=0.75pt]  [font=\small]  {$\gamma _\mathrm{ls}$};

\end{tikzpicture}
\caption{Surface tensions acting at the triple junction.}
    \label{fig:young's_equation}
\end{figure}
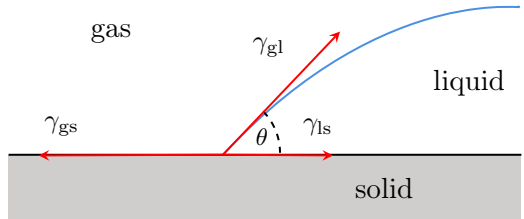

When $\gamma_\mathrm{gs}< \gamma_\mathrm{ls}+\gamma_\mathrm{gl}$, $\theta$ has to be strictly positive, an equilibrium can be obtained, and we are in the regime of partial wetting. If, on the other hand, $\gamma_\mathrm{gs} \ge \gamma_\mathrm{ls}+\gamma_\mathrm{gl}$, we need to have $\theta = 0$ and, at least for $\gamma_\mathrm{gs} > \gamma_\mathrm{ls}+\gamma_\mathrm{gl}$, the fluid film will not stop spreading. Hence, we are in the complete-wetting regime considered in this paper.

\medskip

In \cite{GGKO} it is shown that for $n=2$, under a smallness condition for the initial value for a suitably transformed version of \eqref{TFE}, the resulting problem has a unique classical solution. Furthermore,
in \cite[Remark~3.4]{GGKO} it is explained how the result generalizes immediately to the range
\begin{equation}\label{mobility-range-1}
1.8384 \approx \tfrac{3}{17}(15-\sqrt{21}) < n < \tfrac{3}{11}(7+\sqrt{5}) \approx 2.5189.
\end{equation}
In this paper, we generalize this result to cover the more natural and except for isolated points full range
\begin{equation}\label{mobility-range-2}
1 < n < \tfrac 3 2 \qquad \text{and} \qquad \tfrac 3 2 < n < 3
\end{equation}
(in particular up to Darcy dynamics, $n = 1$, and the no-slip case $n = 3$), by employing a functional-analytic approach relying on semigroup theory. Compared to the methods in \cite{GGKO}, in which techniques from ordinary differential equations (ODEs) and a time-discretization procedure were used, this significantly shortens the arguments. A further benefit of our approach is that it immediately entails maximal $L^p$-regularity in time, with $1<p<\infty$. This is non-obvious from the methods in \cite{GGKO} and turns out to be the crucial ingredient to obtain the larger range \eqref{mobility-range-2} of mobility exponents $n$ compared to \eqref{mobility-range-1}, for which up to now no well-posedness results have been available. In the spatial variables still weighted Hilbert-Sobolev spaces are used, as is also the case for most of the existing well-posedness results \cite{BringmannGiacomelliKnuepferOtto2016,Esselborn2016,GiacomelliKnuepferOtto2008,GGKO,Gnann2015,G2016,GnannPetrache,Knuepfer2011,Knuepfer2015} except for \cite{Degtyarev2017,GiacomelliKnuepfer2010,John2015,Seis2018}. We mention that a corresponding theory of weak solutions without uniqueness results in complete wetting has been developed in \cite{BerettaBertschDalPasso1995,BernisFriedman1990,BertozziPugh1996,BertschDalPassoGarckeGruen1998,Gruen2004a} while qualitative properties have been investigated in \cite{Bernis1996a,Bernis1996b,DalPassoGiacomelliGruen2001,DalPassoGiacomelliShishkov2001,DeNittiFischer2022,Fischer2013,Fischer2014,Fischer2016,GiacomelliGruen2006,GiacomelliShishkov2005,Gruen2003,Gruen2004b,HulshofShishkov1998}. Results for partial-wetting boundary conditions are so far contained in \cite{BertschGiacomelliKarali2005,Degtyarev2017,Esselborn2016,Knuepfer2011,Knuepfer2015,Knuepfer2023,KnuepferMasmoudi2013,KnuepferMasmoudi2015,MajdoubMasmoudiTayachi2021,Mellet2015,Otto1998}.

\subsection{Notation}
We write $a\lesssim_{P}b$ if there exists a constant $0< C < \infty$ only depending on the parameters in the set $P$ and $n$ such that $a\leq Cb$. Similarly, we write $a\sim_P b$ if both $a\lesssim_P b$ and $b \lesssim_P a$. If $P = \emptyset$ or $P = \{n\}$, or if the dependence is specified in the text, the subscript $P$ is omitted. Consistently, we will not specify the dependence of constants on the parameter $n$ throughout the paper.

\medskip
We define $\mathbbm{1}_{A}(x) = 1$ for $x\in A$ and $\mathbbm{1}_{A}(x) = 0$ else, where $A$ is some set.

\medskip

For $\alpha \in \R$ we write $\floor\alpha := \max\{m \in \Z \colon m \le \alpha\}$ and $\ceil\alpha := \min\{m \in \Z \colon m \ge \alpha\}$.

\medskip

We denote by $\mathcal{L}(X)$ the bounded linear operators $X \to X$.

\medskip

For a domain $\Omega \subseteq \R^d$ and $1\leq p \le \infty$, $L^p(\Omega)$ denotes the Lebesgue space of $p$-integrable functions $\Omega \to \R$ and we write $L^p(\Omega;\C)$ for their complex-valued analogue. The convention for Sobolev spaces $H^k(\Omega)$, $W^{k,p}(\Omega)$, and Besov spaces $B^s_{p,q}(\Omega)$ is analogous. In case of Bochner spaces, we write $L^p(\Omega;X)$ etc.

\medskip

For a domain $\Omega \subseteq \R^d$ we write $C_\mathrm{c}^\infty(\Omega)$ for the space of real-valued test functions, that is, functions $\Omega \to \R$ that are infinitely often differentiable and have compact support in $\Omega$. We denote by $C_\mathrm{c}^\infty(\Omega;\C)$ their complex-valued analogue.

\medskip

We denote by $\mathcal{S}(\R)$ the complex-valued Schwartz space and use the normalizations
\begin{align*}
    (\mathcal{F}f)(\xi) &:= (2\pi)^{-\tfrac{1}{2}}\int_{\R}f(x)e^{-ix\xi} \, \d x\quad\text{ for } \xi\in\R \text{ and } f \in \SSS(\R), \\
     (\mathcal{F}^{-1}f)(x) &= (2\pi)^{-\tfrac{1}{2}}\int_{\R}f(\xi)e^{ix\xi} \, \d\xi\quad\text{ for } x\in \R \text{ and } f \in \SSS(\R),
\end{align*}
for the Fourier transform and its inverse.

\subsection{Outline}
The rest of the paper consists of the following parts:

\medskip

In \S\ref{section_3} the free-boundary problem \eqref{TFE} is transformed onto a fixed domain (cf.~\S\ref{sec_derivation_Cauchy}), the functional-analytic setting is introduced (cf.~\S\ref{sec_Functional-Analytic_setting}), and the main result, Theorem~\ref{Main_thm}, is formulated and discussed (cf.~\S\ref{sec_main_result}).

\medskip

\S\ref{section_4} treats the maximal regularity for the linear Cauchy problem. This is divided in the following subsections: in \S\ref{sec_inhomogeneous_equation} the inhomogeneous equation of the abstract Cauchy problem is treated. Then, in \S\ref{sec_homogeneous_equation} the homogeneous equation of the abstract Cauchy problem is solved. \S\ref{sec_Linear_Maximal_Reg_est} and \S\ref{sec_higher_regularity} treat parabolic maximal $L^p$-regularity in time and higher regularity in the spatial variables, respectively.

\medskip

In \S\ref{sec-nonlinear} the nonlinear problem is treated. This is split into suitable embeddings and nonlinear estimates in \S\ref{ssec:non-est} and the proof of the main result in \S\ref{sec_proof_main_result}. 

\medskip

The paper ends with concluding remarks on coercivity in \S\ref{sec-conclusion}.

\section{Setting and Main Result}\label{section_3}
\subsection{Derivation of the Nonlinear Cauchy Problem}\label{sec_derivation_Cauchy}
Here, we reformulate the free-boundary problem \eqref{TFE} as a nonlinear Cauchy problem using the von Mises transform and additional coordinate transformations for the range of mobility exponents $n$ as in \eqref{mobility-range-2}.

\subsubsection{Reformulation for \texorpdfstring{$n\in(1,\tfrac{3}{2})$}{n in (1, 3/2)}}\label{sssec:0n32}
In this case, the generic solution of the free-boundary problem has up to rescaling and translation the form $h = z_+^2$. Note that a stationary solution being generic is not a contradiction to complete-wetting boundary conditions since we are only interested in the behavior close to the contact line $Z(t)$ and since for $h = z_+^2$ it is the asymptotic as $z \to \infty$ that keeps the contact line eventually at rest, see \cite{BringmannGiacomelliKnuepferOtto2016,GiacomelliKnuepferOtto2008,John2015} for the case $n = 1$. For compactly supported droplets in complete wetting, however, the fluid film will indeed not stop spreading, see e.g. \cite{CarlenUlusoy2007,CarlenUlusoy2014,CarrilloToscani2002,Gnann2015,MatthesMcCannSavare2009,Seis2018,SeisWinkler2024} for rigorous results if $n = 1$. See \cite{BelgacemGnannKuehn2016,Bernis1992} for special solutions if $n \in (1,\frac 3 2)$, where close to the contact line the leading-order behavior is the same as for $h = z_+^2$. This stands in contrast to the partial-wetting regime, where for compactly supported solutions an inverted parabola is the stationary long-time asymptotic, see \cite{MajdoubMasmoudiTayachi2021}.  Considering the case of compactly supported droplets for general mobilities would introduce another length scale and would make special solutions non-explicit. This would greatly complicate the analysis on coercivity of the linearized dynamics, for which we find quite explicit characterizations below in Lemma~\ref{lem_coercive_range}.

\medskip

We linearize around this profile using the von Mises transform, that is, we introduce the new dependent variable $Z = Z(t,y)$ through
 \begin{equation}\label{eq_transform_1<n<3/2}
    h(t,Z(t,y)) = y^2 \quad \text{for }t,y>0.
\end{equation}
Note that the profile $y^2$ is strictly monotone for $y>0$ so that $Z = Z(t,y)$ is well defined through the implicit function theorem provided $h(t,\cdot)$ is strictly increasing for any $t \ge 0$ fixed. Differentiating equation \eqref{eq_transform_1<n<3/2} with respect to $t$  gives by the chain rule
\begin{equation}\label{eq_transform_diff}
    h_t +h_z Z_t = 0 \stackrel{\eqref{TFEa}}{\iff} h_z Z_t -(h^n h_{zzz})_z= 0 \quad \text{for }t,y>0.
\end{equation}
On the other hand, differentiating $h(t,Z(t,y))$ with respect to $y$ using $z = Z(t,y)$, we see that $h_y =h_z  Z_y$ and 
\begin{equation}\label{eq_derivative}
    \partial_z = Z_y^{-1}\partial_y.
\end{equation}
Using \eqref{eq_transform_1<n<3/2} and \eqref{eq_derivative} in \eqref{eq_transform_diff} we deduce (omitting parentheses here and in what follows, that is, differential operators act on everything to their right)
\begin{equation*}
    2 y Z_y^{-1} Z_t - Z_y^{-1} \partial_y y^{2n} Z_y^{-1} \partial_y Z_y^{-1} \partial_y Z_y^{-1} 2y = 0 \quad \text{for } t,y>0,
\end{equation*}
which is equivalent to
\begin{equation}\label{eq_equivalent}
    Z_t - y^{-1} \partial_y y^{2n} Z_y^{-1} \partial_y Z_y^{-1} \partial_y Z_y^{-1} y = 0\quad \text{for } t, y>0.
\end{equation}
We now introduce the new variable
\begin{equation}\label{def-h-1n32}
H:= Z_y^{-1}
\end{equation}
and note that $Z_{yt} = -H^{-2}H_t$. Note that in the new variables, the quadratic profile $h=z^2$ corresponds to $H=1$. Using the definition of $H$ and differentiating \eqref{eq_equivalent} with respect to $y$, we get
\begin{equation*}
    H_t + H^2\partial_y y^{-1} \partial_y y^{2n}H \partial_y H \partial_y H y = 0 \quad t,y>0,
\end{equation*}
which, on using the commutation relation
\begin{equation*}
y \partial_y y^\gamma = y^\gamma (y \partial_y + \gamma) \quad \text{for } \gamma \in \R,
\end{equation*}
can be rewritten as
\begin{equation}\label{eq_PDE_H_2}
   H_t + y^{2n-4} H^2(y\partial_y+2n-3)(y\partial_y+2n-1)H y\partial_y H(y\partial_y +1)H=0 \quad\text{for }t,y>0.
\end{equation}
We apply a further change of variables by setting
\begin{equation}\label{def-x-1n32}
x:=\tfrac{y^{4-2n}}{(4-2n)^4},
\end{equation}
so that $y\partial_y = (4-2n)D$ with $D := x\partial_x$. Rewriting \eqref{eq_PDE_H_2} accordingly gives
\begin{equation*}
    H_t + x^{-1} \mathcal{M}_n(H,H,H,H,H) = 0,
\end{equation*}
where
\begin{equation*}
    \mathcal{M}_n(H_1,H_2,H_3,H_4,H_5) = H_1 H_2\big(D+\tfrac{2n-3}{4-2n}\big)\big(D+\tfrac{2n-1}{4-2n}\big)H_3DH_4\big(D+\tfrac{1}{4-2n}\big)H_5.
\end{equation*}
Linearizing around the quadratic profile
\begin{equation}\label{def-u-1n32}
u:=H-1
\end{equation}
gives the nonlinear Cauchy problem
\begin{subequations}\label{eq_CP_n_1_3_2}
\begin{alignat}{2}
    u_t+x^{-1}\mathfrak p_n(D)u &= \mathcal{N}_n(u)\qquad  &&\text{for } t,x>0,\\
    u &=u^{(0)} \qquad&&\text{for } x>0 \text{ and at } t = 0,
\end{alignat}
\end{subequations}
with the linear operator
\begin{align}\nonumber
    \mathfrak p_n(D)u &= \mathcal{M}_n(u,1,\dots,1)+\dots +\mathcal{M}_n(1,\dots,1,u)\\
    &= D\big(D-\tfrac{3-2n}{4-2n}\big)\big(D-\tfrac{1-2n}{4-2n}\big)\big(D-\tfrac{-1}{2-n}\big)u \label{eq_def_p(D)_1<n<3/2}
\end{align}
and the nonlinearity
\begin{equation}\label{eq_nonlin_1<n<3/2}
    \mathcal{N}_n(u) = -x^{-1}\mathcal{M}_n(u+1,\dots,u+1) + x^{-1}\mathfrak p_n(D)u.
\end{equation}
Hence, $\mathfrak p_n(\zeta) = \prod_{j = 1}^4 (\zeta-\gamma_j)$ is a fourth order polynomial with roots in increasing order given by
\begin{equation*}
\gamma_1 := -\tfrac{1}{2-n},\quad \gamma_2:=\tfrac{1-2n}{4-2n},\quad\gamma_3:=0,\quad \gamma_4:=\tfrac{3-2n}{4-2n}.
\end{equation*}
Finally, we note that that we do not have to impose boundary conditions on the Cauchy problem \eqref{eq_CP_n_1_3_2} since the boundary conditions \eqref{TFEb} and \eqref{TFEc} are implicitly fulfilled through \eqref{eq_transform_1<n<3/2} provided $\sup_{t,x > 0} |u|$ is sufficiently small, which through \eqref{def-h-1n32} and \eqref{def-u-1n32} in particular entails that $Z$ is a small Lipschitz perturbation of $Z = y + \mathrm{const.}$ Smallness of $\sup_{t,x > 0} |u|$ is a natural assumption to ensure that the von Mises transform \eqref{eq_transform_1<n<3/2} is diffeomorphic. Also observe that otherwise $Z_x \stackrel{\eqref{def-h-1n32},\eqref{def-u-1n32}}{=} 1/(1+u)$ can become unbounded.

\subsubsection{Reformulation for \texorpdfstring{$n\in(\tfrac{3}{2},3)$}{n in (3/2,3)}}\label{sssec:32n3}
For mobility exponents $n\in(\tfrac{3}{2},3)$ the generic solution of the free-boundary problem \eqref{TFE} is a travelling wave $h(t,z) = h_{\text{TW}}(x)$, where $x=z-Vt$, with $Z_t(t)=V<0$ the constant velocity of the fluid film. This change of coordinates implies
\[
    \partial_z h=\tfrac{\d h_{\text{TW}}}{\d x} \quad \text{ and }\quad\partial_t h = -V \tfrac{\d h_{\text{TW}}}{\d x},
\]
which turns \eqref{TFE} into the ordinary boundary-value problem
\begin{subequations}
\begin{alignat}{2}
    -V\tfrac{\d h_{\text{TW}}}{\d x} +\tfrac{\d}{\d x} \big(h_{\text{TW}}^n \tfrac{\d^3 h_{\text{TW}}}{\d x^3}\big)&=0\qquad && \text{for } x>0, \label{ode-tw}\\
    h_{\text{TW}}=\tfrac{\d h_{\text{TW}}}{\d x}&=0\qquad&& \text{at } x=0, \label{tw-bc1}\\
    h_{\text{TW}}^{n-1}\tfrac{\d^3 h_{\text{TW}}}{\d x^3}&=-V\qquad&& \text{at } x=0, \label{tw-bc2}
\end{alignat}
\end{subequations}
where we have assumed $Z(0)=0$ by translation invariance. Integrating the ODE \eqref{ode-tw} and appealing to the boundary conditions \eqref{tw-bc1} and \eqref{tw-bc2} gives
\begin{subequations}\label{eq_ODE_TW}
\begin{alignat}{2}
h_{\text{TW}}^{n-1}\tfrac{\d^3 h_{\text{TW}}}{\d x^3}&=V\qquad && \text{for } x>0,\\
h_{\text{TW}}=\tfrac{\d h_{\text{TW}}}{\d x}&=0\qquad&& \text{at } x=0.
\end{alignat}
\end{subequations}
By a rescaling of $x$ we may assume without loss of generality that the velocity $V$ of the traveling wave only depends on $n$. In particular, we may assume that this velocity is $V= - \frac{3}{n}(\frac{3}{n}-1)(2 - \frac{3}{n})$. This particular choice of $V$ ensures that
\begin{equation}\label{sol-tw}
h_{\text{TW}}=x^{\frac3n}
\end{equation}
solves \eqref{eq_ODE_TW}. Note that this choice reduces for linear Navier slip, i.e., $n=2$, to the velocity $V=-\tfrac{3}{8}$ of the travelling wave, see e.g. \cite{GGKO}.

\medskip

The next step is to study perturbations of solutions to \eqref{TFE} around the traveling wave $x^{\frac{3}{n}}$ with the von Mises transform
\begin{equation}\label{eq_transform_>3/n}
    h(t,Z(t,x)):=x^{\frac{3}{n}},
\end{equation}
which is carried out in \cite[\S1.3]{GGKO} for $n=2$. From \eqref{eq_transform_>3/n} we infer
\begin{equation*}
    h_z Z_x= \partial_x h = \tfrac{3}{n}x^{\frac{3}{n}-1}\quad\text{ and }\quad
    h_t + h_z Z_t = 0.
\end{equation*}
Using this in \eqref{eq_transform_>3/n} together with \eqref{TFEa}, we obtain
\begin{equation}\label{zt-32n3}
    Z_t = \tfrac{n}{3}x^{1-\frac{3}{n}}\partial_x x^3 Z_x^{-1}\partial_x Z_x^{-1} \partial_x Z_x^{-1} \partial_x x^{\frac{3}{n}}.
\end{equation}
We again introduce the variable
\begin{equation}\label{def-h-32n3}
H:= Z_x^{-1}
\end{equation}
and thus $\partial_t H = -H^{2}Z_{xt}$. Note that the constant solution $H = H_\mathrm{TW} = 1$ corresponds to the traveling-wave profile $h_{\mathrm{TW}}$. Writing $D=x\partial_x$ and using the commutation relation
\begin{equation*}
Dx^{\gamma}=x^{\gamma}(D+\gamma) \quad \text{for } \gamma \in \R,
\end{equation*}
we obtain
\begin{equation*}
    \partial_t H + x^{-1}\mathcal{M}_n(H,H,H,H,H)=0,
\end{equation*}
where
\begin{equation*}
    \mathcal{M}_n(H_1, H_2, H_3, H_4, H_5) := H_1H_2D(D+\tfrac{3}{n})H_3(D+\tfrac{3}{n}-2)H_4(D+\tfrac{3}{n}-1)H_5.
\end{equation*}
Linearizing around
\begin{equation}\label{def-u-32n3}
u:=H-1
\end{equation}
gives the nonlinear Cauchy problem
\begin{subequations}\label{eq_CP_n_3_2_3}
\begin{alignat}{2}
    u_t+x^{-1}\mathfrak p_n(D)u &= \mathcal{N}_n(u)\qquad  &&\text{for } t,x>0,\\
    u &=u^{(0)} \qquad&&\text{for } x>0 \text{ and at } t = 0,
\end{alignat}
\end{subequations}
with the linear operator
\begin{align}\nonumber
    \mathfrak p_n(D)u &:= \mathcal{M}_n(u,1,\dots,1)+\dots +\mathcal{M}_n(1,\dots,1,u)\\
    &= D (D + \tfrac 3 n) \big(D^2 + (\tfrac 9 n - 4) D - 3 (2-\tfrac 3 n) (\tfrac 3 n-1)\big) u \nonumber \\
    &= D\big(D+\tfrac{3}{n}\big)\left(D-\omega_1\right)\left(D-\omega_2\right)u, \label{eq_def_p(D)_3/2<n<3}
\end{align}
where
\begin{equation*}
\omega_1:= \tfrac{4n-9-\sqrt{-27+36n-8n^2}}{2n} \quad\text{ and }\quad  \omega_2:=\tfrac{4n-9+\sqrt{-27+36n-8n^2}}{2n},
\end{equation*}
and the nonlinearity
\begin{equation}\label{eq_nonlin_3/2<n<3}
    \mathcal{N}_n(u) = -x^{-1}\mathcal{M}_n(u+1,\dots,u+1) + x^{-1}\mathfrak p_n(D)u.
\end{equation}
Hence, $\mathfrak p_n(\zeta)$ is a fourth-order polynomial and for $n\in (\tfrac{3}{2},3)$ the zeros are ordered as follows (from small to large)
\begin{equation*}
\gamma_1:=-\tfrac{3}{n},\quad \gamma_2:=\omega_1,\quad\gamma_3:=0,\quad \gamma_4:=\omega_2.
\end{equation*}
Again, we note that \eqref{eq_CP_n_3_2_3} does not require boundary conditions as these are implicitly fulfilled through \eqref{eq_transform_>3/n} provided $\sup_{t,x > 0} |u|$ is sufficiently small, which through \eqref{def-h-32n3} and \eqref{def-u-32n3} in particular entails that $Z$ is a small Lipschitz perturbation of $Z = y + \mathrm{const.}$ Smallness of $\sup_{t,x > 0} |u|$ ensures that the von Mises transform \eqref{sol-tw} is diffeomorphic. Further consider that otherwise $Z_x \stackrel{\eqref{def-h-32n3},\eqref{def-u-32n3}}{=} 1/(1+u)$ can become unbounded.

\subsubsection{The Nonlinear Cauchy Problem}\label{sec_Non_Cauchy_Problem}
Both \eqref{eq_CP_n_1_3_2} and \eqref{eq_CP_n_3_2_3} result in the Cauchy problem
\begin{subequations}\label{eq_nonlin_cauchy}
\begin{alignat}{2}
    u_t+x^{-1}\mathfrak p(D)u &= \mathcal{N}(u)\qquad  &&t,x>0,\\
    u|_{t=0} &=u^{(0)} \qquad&&x>0,
\end{alignat}
\end{subequations}
where
\begin{equation}\label{eq_def_p(D)}
    \mathfrak p(D):=\mathfrak p_n(D)=(D-\gamma_1)(D-\gamma_2)(D-\gamma_3)(D-\gamma_4)
\end{equation}
with
\begin{subequations}\label{eq_zeros_12}
\begin{equation}\label{eq_zeros_1}
\gamma_1 = - \tfrac{1}{2-n},\quad \gamma_2=\tfrac{1-2n}{2(2-n)},\quad\gamma_3=0,\quad \gamma_4=\tfrac{3-2n}{2(2-n)} =: \beta \quad \text{for } n\in (1,\tfrac{3}{2}),
\end{equation}
and
\begin{equation}\label{eq_zeros_2}
\gamma_1=-\tfrac{3}{n},\;\; \gamma_2= \tfrac{4n-9-\sqrt{-27+36n-8n^2}}{2n},\;\;\gamma_3=0,\;\; \gamma_4=\tfrac{4n-9+\sqrt{-27+36n-8n^2}}{2n} =: \beta \quad \text{for } n\in (\tfrac{3}{2},3).
\end{equation}
\end{subequations}
The fact that in both cases $n\in(1,\tfrac{3}{2})$ and $n\in(\tfrac{3}{2},3)$ we have the root $\gamma_3 = 0$ is in line with \eqref{TFEa} being in divergence form.
The nonlinear right-hand side is given by
\begin{equation}\label{eq_def_nonlin}
    \mathcal{N}(u):=\mathcal{N}_n(u) = -x^{-1}\mathcal{M}_n(u+1,\dots,u+1) + x^{-1}\mathfrak p_n(D)u,
\end{equation}
where
\begin{align}\label{eq_def_M}
    &\mathcal{M}_n(H_1, H_2, H_3, H_4, H_5)=\\&\quad\begin{cases}H_1 H_2\big(D+\tfrac{2n-3}{4-2n}\big)\big(D+\tfrac{2n-1}{4-2n}\big)H_3DH_4\big(D+\tfrac{1}{4-2n}\big)H_5\;\; &\text{for }n\in(1, \tfrac{3}{2}).\\
    H_1H_2D\big(D+\tfrac{3}{n}\big)H_3\big(D+\tfrac{3}{n}-2\big)H_4\big(D+\tfrac{3}{n}-1\big)H_5\;\; &\text{for }n\in(\tfrac{3}{2},3).
\end{cases}\nonumber
\end{align}
In what follows, the subscripts $n$ will usually be omitted.

\pagebreak
%

\subsection{Norms and Spaces}\label{sec_Functional-Analytic_setting}
\subsubsection{Weighted Sobolev Spaces}
For $k\in \N_0$ and $\alpha\in\R$ we use the weighted inner products
\begin{equation}\label{eq_inner}
    (\phi, \psi)_{k,\alpha} = \sum _{j=0}^k\int_0^{\infty}x^{-2\alpha}D^j\phi \overline{D^j\psi} \tfrac{\d x}{x}, \quad  (\phi, \psi)_{\alpha} :=  (\phi, \psi)_{0,\alpha},
\end{equation}
where $\phi,\psi \in C_\mathrm{c}^{\infty}((0,\infty);\K)$ with $\K \in \{\R,\C\}$. We write
$\verti{\cdot}_{k,\alpha}$ and $\verti{\cdot}_\alpha := \verti{\cdot}_{0,\alpha}$, where $\verti{\phi}_{k,\alpha} := \sqrt{(\phi,\phi)_{k,\alpha}}$, for the induced norms. The space $H_{k,\alpha}^\K$ is defined as the closure of $C_\mathrm{c}^{\infty}((0,\infty);\K)$ with respect to $\verti{\cdot}_{k,\alpha}$. For $k \ge 1$ we define inductively $H_{-k,\alpha-1}^\K$ as the dual of $H_{-k+4,\alpha}^\K$ relative to $H_{-k+2,\alpha-\frac 1 2}^\K$.
We write $H_\alpha^\K := H_{0,\alpha}^\K$ and
\begin{equation}\label{script-hka}
\mathscr{H}_{k,\alpha} := H_{k,\alpha - \frac{1}{2}}^\K \cap H_{k+2,\alp}^\K.
\end{equation}
For $\alpha = 0$ we obtain an isometry with the standard Sobolev spaces $H_{k,0}^\K \simeq H^k(\R;\K) = W^{k,2}(\R;\K)$ on mapping $u \mapsto (s \mapsto u(e^s))$, and for $k = 0$ and $\alpha =0$ it holds $H_0^\K \simeq L^2(\R;\K)$. We write $H_{k,\alpha} := H_{k,\alpha}^\R$ and $H_\alpha := H_{0,\alpha}$, where $k \in \Z$ and $\alpha \in \R$ in the real-valued case.

\subsubsection{Interpolation Spaces}
For $k \in \N_0$, $\alpha \in \R$, $\vartheta \in (0,1)$, and $1 < p < \infty$, we introduce the real interpolation space
\begin{equation}\label{def-trace-space}
H_{k+2-4\vartheta,\alpha - \vartheta, p} := (H_{k-2,\alpm},H_{k+2,\alp})_{1-\vartheta,p}
\end{equation}
(cf.~\cite[Def.~1.2.2]{Lunardi}) with norm
\begin{equation}\label{def_k,alpha,p_norm}
    \verti{u}_{k+2-4\vartheta,\alpha - \vartheta, p}:= \verti{u}_{(H_{k-2,\alpm},H_{k+2,\alp})_{1-\vartheta,p}},
\end{equation}
where we use the $K$-method of real interpolation in what follows (cf.~\cite[\S1.2.1]{Lunardi}). We define for $k > 0$, $k \notin \N$, and $\alpha \in \R$ the fractional weighted Sobolev space
\begin{equation}\label{sobolev-frac}
H_{k,\alpha} := (H_{\floor k, \alpha},H_{\floor k +4,\alpha})_{\frac{k-\floor k}{4},2}
\end{equation}
with induced norm $\verti{\cdot}_{k,\alpha}$.

\medskip

Note that the real interpolation of weighted $L^p$-spaces with respect to $(\cdot,\cdot)_{\theta,q}$ in the non-diagonal case $p \ne q$ is a delicate matter of interpolation of operators, see e.g.~\cite{Gilbert1972}. We have the following characterization and embeddings of the interpolation norm in terms of fractional Sobolev spaces that can be deduced by direct elementary arguments:
\begin{lemma}\label{lemma_char_interpol}
For $k \in \N$ with $k \ge 2$, $\alpha \in \R$, $\vartheta \in (0,1)$, and $1 < p < \infty$ it holds
\begin{equation}\label{eq_char_interpol}
    \verti{u}_{k + 2-4\vartheta,\alpha - \vartheta,p} \gtrsim_{k,\alpha,\vartheta,p} \begin{cases} \big\lVert s \mapsto e^{-(\alpha-\vartheta) s} u(e^s)\big\rVert_{W^{k+2-4\vartheta,2}(\R)} & \text{for } p \le 2, \\
     \big\lVert s \mapsto e^{-(\alpha-\vartheta) s} u(e^s)\big\rVert_{W^{k + \frac 3 2 + \frac 1 p - 4 \vartheta,p}(\R)} & \text{for } p > 2, \end{cases} \quad u \in H_{k+2-4\vartheta,\alpha-\vartheta,p},
\end{equation}
where equivalence holds for $p = 2$, that is, we have $H_{k+2-4\vartheta,\alpha-\vartheta,2} = H_{k+2-4\vartheta,\alpha-\vartheta}$ (equivalence of norms). Additionally, we have
\begin{equation}\label{eq_lower_bound_interpol}
      \verti{u}_{k+2-4\vartheta,\alpha-\vartheta,p} \gtrsim_{k,\alpha,\vartheta, p} \big\lVert s \mapsto e^{-(\alpha-\vartheta) s} u(e^s)\big\rVert_{W^{m,\infty}(\R)} \quad \text{for } u \in H_{k+2-4\vartheta,\alpha-\vartheta,p},
\end{equation}
where $m = k+\frac 3 2-4\vartheta$ if $k+\frac 3 2-4\vartheta \notin \Z$ and $m < k+\frac 3 2-4\vartheta$ else. Furthermore, for $\eta(x) := x^\kappa \mathbbm{1}_{[0,1]}(x) + x^{-\kappa} (1-\mathbbm{1}_{[0,1]}(x))$, where $\kappa > 0$, we have for $p > 2$
\begin{equation}\label{lower_bound_l2eta}
\verti{u}_{k+2-4\vartheta,\alpha-\vartheta,p}^2 \gtrsim_{k,\kappa,\alpha,\vartheta,p} \sum_{j = 0}^m \int_0^\infty \eta^2 x^{-2(\alpha-\vartheta)} (D^j u)^2 \tfrac{\d x}{x}
\end{equation}
for $m := \floor{k + \tfrac 3 2 + \tfrac 1 p - 4 \vartheta}$, $u \in H_{k+2-4\vartheta,\alpha-\vartheta,p}$.
\end{lemma}
The proof of Lemma~\ref{lemma_char_interpol} is given in \S\ref{sec_homogeneous_equation}.

\subsubsection{Parabolic Spaces}
For treating the nonlinear problem in \S\ref{sec-nonlinear}, we introduce norms $\vertiii{\cdot}$ for the solution $u$, $\vertiii{\cdot}_1$ for the right-hand side $f$, and $\vertiii{\cdot}_0$ for the initial data $u^{(0)}$ appearing in \eqref{eq_nonlin_cauchy}. Therefore, let $\beta = \gamma_4$ be the largest zero of $\mathfrak p(D)$, as defined in \eqref{eq_zeros_12}. Suppose $k, \tilde k \in \N_0$, $1<p<\infty$ such that $\frac 1 p < \beta$, and $0 < \delta < \tilde\delta < \min\{\tfrac{1}{p}, \beta - \frac 1 p\}$. We define
\begin{subequations}\label{triple_bar_norms}
\begin{align}\nonumber
    \vertiii{u}^p:= &\sup_{t \ge 0} \Big[\verti{u}^p_{\tilde k+8-\frac 4 p,-\tilde\delta,p}+\verti{u-u_0}^p_{\tilde k+8-\frac 4 p,\tilde\delta,p}+t^{p\beta-1} \verti{u-u_0}^p_{k + 8 - \frac 4 p,\beta - \frac 1 p - \delta,p} \\ &+t^{p\beta-1} \verti{u-u_0}^p_{k + 8 - \frac 4 p, \beta - \frac 1 p + \delta,p}\Big] + \int_0^\infty \Big[\verti{\partial_t u}^p_{\tilde k+4,-1+\frac{1}{p}-\tilde\delta}+\verti{\partial_t u}^p_{\tilde k+4,-1+\frac{1}{p}+\tilde\delta} \nonumber \\
    &+ t^{p\beta-1} \verti{\partial_t u}^p_{k+4,\beta-1-\delta}+t^{p\beta-1} \verti{\partial_t u}^p_{k+4,\beta-1+\delta}+ \verti{u-u_0}^p_{\tilde k+8,\frac{1}{p}-\tilde\delta} \nonumber \\
    &+ \verti{u-u_0}^p_{\tilde k+8,\frac{1}{p}+\tilde\delta}
    + t^{p\beta-1} \verti{u-u_0}^p_{k+8,\beta-\delta} + t^{p\beta-1} \verti{u-u_0 -u_{\beta}x^{\beta}}^p_{k+8,\beta+\delta} \Big] \d t, \label{eq_triple_bar_norm} \\
    \vertiii{u^{(0)}}_0^p :=&\verti{u^{(0)}}^p_{\tilde k+8-\frac 4 p,-\tilde\delta,p}+\verti{u^{(0)}-u_0^{(0)}}^p_{\tilde k+8-\frac 4 p,\tilde\delta,p}, \label{eq_def_norm_u^0} \\
    \vertiii{f}^p_{1}:=&\int_0^\infty \Big[\verti{f}^p_{\tilde k+4,-1+\frac{1}{p}-\tilde\delta}+\verti{f}^p_{\tilde k+4,-1+\frac{1}{p}+\tilde\delta}+t^{p\beta-1}\verti{f}^p_{k+4, -1+\beta -\delta}+t^{p\beta-1}\verti{f}^p_{k+4,-1+\beta +\delta}\Big] \d t. \label{eq_norm_1_rhs}
\end{align}
\end{subequations}
Here, $u_0(t) := \lim_{x \searrow 0} u(t,x)$ for $t \ge 0$, $u_\beta(t) := \lim_{x \searrow 0} x^{-\beta} (u(t,x)-u_0(t))$ for $t \ge 0$, and $u_0^{(0)} := \lim_{x \searrow 0} u^{(0)}(x)$ are spatial boundary traces of $u$ and $u^{(0)}$. The norms \eqref{triple_bar_norms} are well-defined if
\begin{align*}
u^{(0)} &\in H_{\tilde k+8-\frac 4 p,-\tilde\delta,p}, \qquad u^{(0)} - u^{(0)}_0 \in H_{\tilde k+4-\frac 4 p,\tilde\delta,p}, \\
f &\in L^p\big(0,\infty;H_{\tilde k+4,-1+\frac{1}{p}-\tilde\delta} \cap H_{\tilde k+4,-1+\frac{1}{p}+\tilde\delta}\big), \\
\big(t \mapsto t^{\beta - \frac 1 p} f(t)\big) &\in L^p\big(0,\infty;H_{k+4, -1+\beta -\delta} \cap H_{k+4,-1+\beta +\delta}\big), \\
u &\in BC^0\big([0,\infty);H_{\tilde k+8-\frac 4 p,-\tilde\delta,p}\big), \\
u-u_0 &\in BC^0\big([0,\infty);H_{\tilde k+8-\frac 4 p,\tilde\delta,p}\big), \\
u-u_0 &\in  L^p\big(0,\infty;H_{\tilde k+8,\frac{1}{p}-\tilde\delta} \cap H_{\tilde k+8,\frac{1}{p}+\tilde\delta}\big), \\
\big(t \mapsto t^{\beta - \frac 1 p} (u(t)-u_0(t))\big) &\in BC^0\big([0,\infty);H_{k + 8 - \frac 4 p,\beta - \frac 1 p - \delta,p} \cap H_{k + 8 - \frac 4 p,\beta - \frac 1 p + \delta,p}\big), \\
\partial_t u &\in L^p\big(0,\infty;H_{\tilde k+4,-1+\frac{1}{p}-\tilde\delta} \cap H_{\tilde k+4,-1+\frac{1}{p}+\tilde\delta}\big), \\
\big(t \mapsto t^{\beta-\frac 1 p} \partial_t u(t)\big) &\in L^p\big(0,\infty;H_{k+4,\beta-1-\delta} \cap H_{k+4,\beta-1+\delta}\big), \\
\big(t \mapsto t^{\beta-\frac 1 p} (u(t)-u_0(t))\big) &\in L^p(0,\infty;H_{k+8,\beta-\delta}), \\
\big(t \mapsto t^{\beta-\frac 1 p} (u(t)-u_0(t)-u_\beta(t) x^\beta)\big) &\in L^p(0,\infty;H_{k+8,\beta+\delta}).
\end{align*}
In this case, $u_0^{(0)}$, $u_0$, and $u_\beta$ are uniquely determined. We denote the corresponding Banach spaces for the initial data $u^{(0)}$, solution $u$ and right-hand side $f$ by $U^{(0)}(\tilde k, \tilde\delta, p)$, $U(k,\tilde k, \delta, \tilde\delta, p)$, and $F(k,\tilde k, \delta, \tilde\delta, p)$, respectively. On a finite time interval $I = [0,T)$ with $T \in (0,\infty)$, we write $U(k,\tilde k, \delta, \tilde\delta, p,I)$ and $F(k,\tilde k, \delta, \tilde\delta, p,I)$ (see Proposition~\ref{prop_sol_Au=v} below). Relevant embeddings are listed in Lemma~\ref{lem_coeff_est}.

\subsection{The Main Result}\label{sec_main_result}
\subsubsection{Well-posedness and asymptotic stability}

In order to prove our main theorem, Theorem~\ref{Main_thm}, we need the following result, which is proved in \S\ref{sec_inhomogeneous_equation}.
\begin{lemma}[Coercivity]\label{lem_coercive_range}
Let $w\in H_{\alp}^\C$. There exists a constant $K_\alpha > 0$ such that
\begin{equation}\label{coercivity-pd}
\Re (w,\mathfrak p(D)w)_{\alp} \ge K_\alpha \verti{w}_{2,\alpha}^2 \quad \text{for all } w \in C_\mathrm{c}^\infty((0,\infty);\C)
\end{equation}
if
\begin{subequations}\label{coerc_range}
\begin{align}
    \alp &\in \big(\tfrac{1-2n}{2(2-n)},0\big)\cap \Big[\tfrac{1 - 2 n - \frac{1}{\sqrt 3} \sqrt{13 - 12n + 4 n^2}}{4 (2-n)},\tfrac{1 - 2 n + \frac{1}{\sqrt 3} \sqrt{13 - 12n + 4 n^2}}{4 (2-n)}\Big] && \text{for }n \in \left(1, \tfrac{3}{2}\right), \label{coerc_range_n132}\\
    \alp &\in \big(\tfrac{4n-9-\sqrt{-27+36n-8n^2}}{2n},0\big)\cap \big[\tfrac{n-3}{n}-\tfrac{1}{\sqrt{2n}}, \tfrac{n-3}{n}+\tfrac{1}{\sqrt{2n}}\big] && \text{for }n \in \left( \tfrac{3}{2},3\right). \label{coerc_range_n323}
\end{align}
\end{subequations}
We will refer to the set of all $\alpha \in \R$ for which \eqref{coercivity-pd} holds true as the \emph{coercivity range} of $\mathfrak p(D) = \mathfrak p_n(D)$, see Figure~\ref{fig:coercivity_range}.
\end{lemma}
\begin{figure}[h]
     \centering
     \begin{subfigure}[b]{0.45\textwidth}
         \centering
         \includegraphics[width=1.1\textwidth]{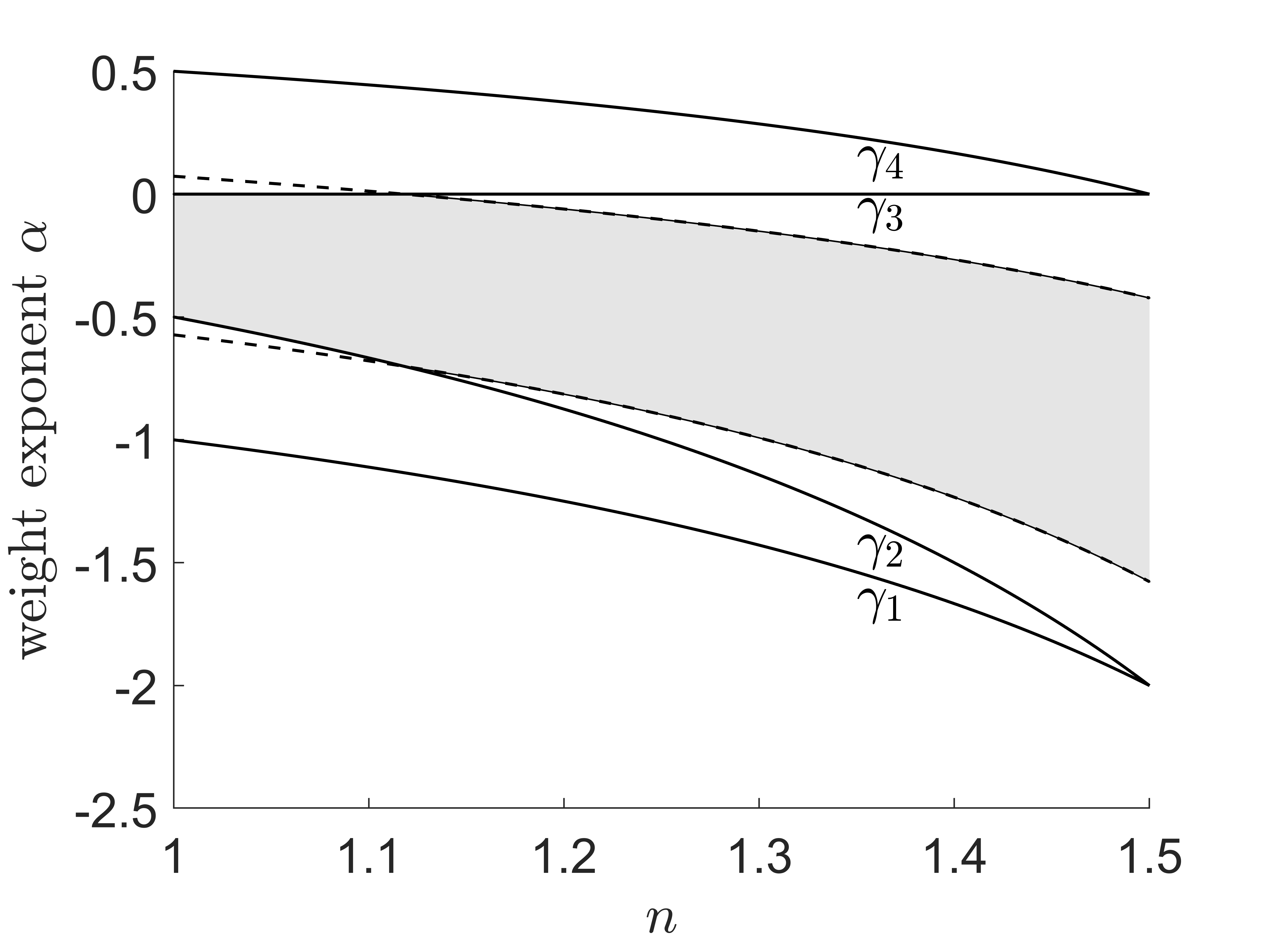}
         \caption{$n\in (1,\tfrac{3}{2})$}
         \label{fig:coercivity_range_a}
     \end{subfigure}
     \hfill
     \begin{subfigure}[b]{0.45\textwidth}
         \centering
         \includegraphics[width=1.1\textwidth]{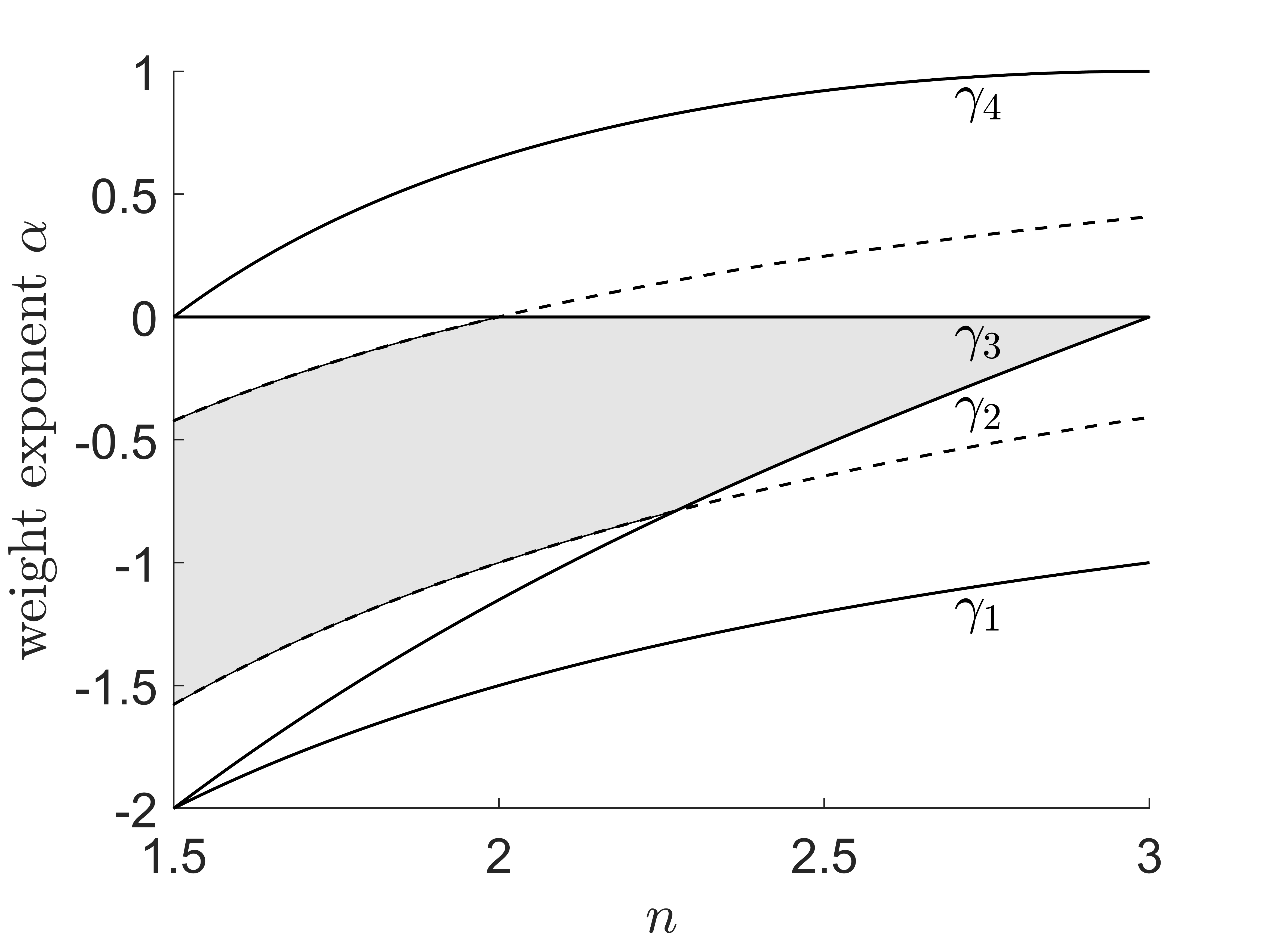}
         \caption{$n\in(\tfrac{3}{2},3)$}
         \label{fig:coercivity_range_b}
     \end{subfigure}
        \caption{For the two different cases of $n$ the zeros $\gamma_1,\dots,\gamma_4=\beta$ of $\mathfrak p(D)$ (solid lines) and the upper and lower bound in \eqref{coerc_range} (dashed lines) are shown. The coercivity range for $\alp$ contains the shaded area.}
        \label{fig:coercivity_range}
\end{figure}

A full characterization of the coercivity range of $\mathfrak p(D)$ (which is analytically rigorous at least for $n < \frac 3 2$) is carried out in \S\ref{sec-conclusion}, but is not necessary for the main result of this paper, Theorem~\ref{Main_thm}:
\begin{theorem}\label{Main_thm}
For every $n\in(1,3) \setminus \{\frac 3 2\}$, choose $1 < p < \infty$ such that $\frac{1}{p}<\beta=\gamma_4$. Take $\tilde \delta > \delta > 0$ such that $0<\tilde\delta< \min\{-\gamma_2, \beta - \frac 1 p, \frac 1 p, 1 - \frac 1 p\}$. Let $k, \tilde k \in \N_0$ such that
$\tilde k > k + \frac 1 2 + \frac 3 p$. Further suppose that $\alpha \in \{\frac 1 p -1 \pm \tilde \delta, \beta -1 \pm \delta\}$ are in the coercivity range (cf.~\eqref{coerc_range} of Lemma~\ref{lem_coercive_range} for a sufficient criterion).

\medskip

Then there exists an $\varepsilon = \eps(k,\tilde k,\delta,\tilde\delta,p) >0$ such that for all $u^{(0)} \in U^{(0)}(\tilde k, \tilde\delta, p)$ with $\vertiii{u^{(0)}}_0<\varepsilon$, the nonlinear Cauchy problem \eqref{eq_nonlin_cauchy} has a unique classical solution $u \in U(k,\tilde k, \delta, \tilde\delta, p)$. This solution satisfies the a-priori estimate
\begin{equation}\label{apriori-main}
\vertiii{u} \lesssim_{k,\tilde k,\delta,\tilde \delta,p} \vertiii{u^{(0)}}_0.
\end{equation}
Furthermore, $[0,\infty) \owns t \mapsto \vertiii{u(t)}_0$ is continuous and $\vertiii{u(t)}_0 \to 0$ as $t \to \infty$, which implies that the stationary solution $y^2$ or traveling wave \eqref{sol-tw}, respectively, is asymptotically stable.
\end{theorem}
%

\subsubsection{Discussion}
Note that Theorem~\ref{Main_thm} is the generalization of \cite[Theorem~3.1]{GGKO}, which (cf.~\cite[Remark~3.4]{GGKO}) applies for the range \eqref{mobility-range-1} of mobility exponents $n$. Notice that the proved regularity is sufficient to control four classical derivatives in space of $u$ for $x > 0$. The time derivative is attained in terms of semi-group theory in the strong sense with values in a Banach space with point-wise control, thus ensuring classical differentiability in time for $x > 0$. Hence, the regularity attained is sufficient to satisfy the boundary conditions \eqref{TFEb} and \eqref{TFEc} classically (see also \S\ref{sssec:reg_free} and \S\ref{sssec:original}). The main difference between \cite[Theorem~3.1]{GGKO} and Theorem~\ref{Main_thm} is that here, a larger range $n \in (1,3) \setminus \{\frac 3 2\}$ of mobility exponents is covered (in particular, up to the boundary values $n = 1$, Greenspan's slip condition \cite{Greenspan1978} or the lubrication approximation of Darcy's law in the Hele-Shaw cell \cite{GiacomelliOtto2003,KnuepferMasmoudi2013,KnuepferMasmoudi2015,MatiocProkert2012}, and $n = 3$, the no-slip case \cite{ConstantinDupontGoldsteinKadanoffShelleyZhou1993, GuentherProkert2008}). This is possible because we use $L^p$-norms in time (cf.~\eqref{triple_bar_norms}), allowing to meet $\frac{1}{p}<\beta$, in comparison with $L^2$-norms in time in \cite{GGKO} leading to the constraint $\beta > \frac 1 2$ in \cite{GGKO}. From \eqref{eq_zeros_1} we infer that $\beta < \frac 1 2$ for $n \in (1,\frac 3 2)$ and that $\beta \to 0$ as $n \to \frac 3 2$, explaining the lower bound in \eqref{mobility-range-1} using the $L^2$ in time approach of \cite{GGKO}. Furthermore, our analysis requires that both $\beta-1$ and $\frac 1p -1$ are in the coercivity range of $\mathfrak p(D)$ (cf.~Lemma~\ref{lem_coercive_range}). The latter constraint for $p = 2$ leads to the upper bound on $n$ in \eqref{mobility-range-1}, which is necessary for the approach of \cite{GGKO} to apply. Choosing $p \in (1,\infty)$ suitably also removes this constraint, so that Theorem~\ref{Main_thm} applies for the natural range \eqref{mobility-range-2}. For these reasons, and in order to simplify the arguments, the $L^p$ in time approach appears natural in our situation to cover the physically natural range of slip conditions determined by \eqref{mobility-range-2}.

\medskip

Notice that because of \eqref{eq_lower_bound_interpol} of Lemma~\ref{lemma_char_interpol}, \eqref{eq_triple_bar_norm}, \eqref{eq_def_norm_u^0}, and \eqref{apriori-main} of Theorem~\ref{Main_thm}, the norm $\sup_{t,x \ge 0} \verti{u(t,x)}$ is controlled and sufficiently small if $\vertiii{u^{(0)}}_0$ is sufficiently small. Because of \eqref{def-h-1n32}, \eqref{def-x-1n32}, and \eqref{def-u-1n32}, this entails that for $n \in (1,\frac 3 2)$ the function $[0,\infty) \owns x \mapsto Z(t,x)$ is for fixed $t \ge 0$ strictly increasing in $x$, so that the transformation \eqref{eq_transform_1<n<3/2} is well defined. Likewise, because of \eqref{def-h-32n3} and \eqref{def-u-32n3} the function $[0,\infty) \owns x \mapsto Z(t,x)$ is for fixed $t \ge 0$ strictly increasing in $x$, so that the transformation \eqref{eq_transform_>3/n} is well-defined. An optimal result, in the flavor of \cite{John2015,Seis2018} for the thin-film equation and \cite{Kienzler2016,Koch1999} for the second-order degenerate-parabolic porous-medium equation, would only require $\sup_{x \ge 0} \verti{u^{(0)}(x)}$ to be small and is left open for future work.

\subsubsection{Regularity at the free boundary\label{sssec:reg_free}}
In view of Lemma~\ref{lem_coeff_est} and the definition of the norm $\vertiii{\cdot}$ in \eqref{eq_triple_bar_norm}, it follows from Theorem~\ref{Main_thm} that the solution $u$ has an expansion
\begin{equation}\label{partial_regularity}
u(t,x) = u_0(t) + u_\beta(t) x^\beta + o(x^{\beta + \delta}) \quad \text{as } x \searrow 0, \text{ almost everywhere in time } t,
\end{equation}
where $[0,\infty) \owns t \to u_0(t)$ is continuous with $u_0(t) \to 0$ as $t \to \infty$ and $u_\beta(t) = o(t^{-\beta})$ as $t \to \infty$. Note that the asymptotic expression \eqref{partial_regularity} as well as asymptotic expressions below are understood to hold also on taking a controlled number of derivatives in $x$. The asymptotic \eqref{partial_regularity} is in line with findings in \cite{BelgacemGnannKuehn2016,GiacomelliGnannOtto2013}, where, after factoring off the leading-order behavior, source-type self-similar solutions were extended to analytic functions in $(x_1,x_2)$ on identifying $x_1 = x$ and $x_2 = x^\beta$. A corresponding regularity result for \eqref{TFE} with $n = 2$ was proved in \cite{G2016}, that is, a generalized Taylor expansion of $u$ in terms of $x_1$ and $x_2$ was proved to any order for positives times $t > 0$, but does not have to be controlled at time $t = 0$. In this sense, \eqref{partial_regularity} is a partial-regularity result and a result analogous to \cite{G2016}, a generalized degenerate-parabolic higher-order smoothing effect, is expected to be true also for $n \in (1,3) \setminus \{\frac 3 2\}$. Yet, the proof of \cite{G2016} does not in general carry over to other mobilities than the quadratic one, as there it was crucial that the coercivity range contains the interval $(-1,0)$ (see \eqref{coerc_range_n323} in Lemma~\ref{lem_coercive_range} and the more general characterization in \S\ref{sec-conclusion}).

\subsubsection{Transformation into the original variables\label{sssec:original}}
\paragraph{Case $n \in (1,\frac 3 2)$}
We deduce that
\begin{equation}\label{as-h-y-1n32}
H \stackrel{\eqref{def-u-1n32}}{=} u+1 \stackrel{\eqref{def-x-1n32},\eqref{eq_zeros_1},\eqref{partial_regularity}}{=} 1 + u_0(t) + \tfrac{u_\beta(t)}{(4-2n)^{2 \frac{3-2n}{2-n}}} y^{3-2n} + o(y^{3-2n + (4-2n)\delta}) \quad \text{as } y \searrow 0, 
\end{equation}
almost everywhere in time $t$. From \eqref{eq_equivalent} and \eqref{def-h-1n32} we get
\[
Z_t = y^{2n-3} (y \partial_y + 2 n - 1) H y \partial_y H (y \partial_y + 1) H,
\]

so that with \eqref{as-h-y-1n32} and analogous expressions for $y\partial_y$ derivatives, we obtain the velocity of the contact line as
\[
Z_t(t,0) = \tfrac{2 (3-2n) (5-2n) (1+u_0(t))^2 u_\beta(t)}{(4-2n)^{2 \frac{3-2n}{2-n}}},
\]
almost everywhere in time $t$. We infer that $Z_t(t,0) = o(t^{-\beta})$ as $t \to \infty$. From  \eqref{def-h-1n32} and \eqref{as-h-y-1n32} we obtain
\[
Z_y = \tfrac{1}{1+u_0(t)} - \tfrac{u_\beta(t)}{(4-2n)^{2 \frac{3-2n}{2-n}} (1+u_0(t))^2} y^{3-2n} + o(y^{3-2n + (4-2n)\delta}) \quad \text{as } y \searrow 0,
\]
almost everywhere in time $t$, that is,
\[
Z(t,y) = Z(t,0) + \tfrac{1}{1+u_0(t)} y - \tfrac{u_\beta(t)}{(4-2n)^{\frac{8-5n}{2-n}} (1+u_0(t))^2} y^{4-2n} + o(y^{4-2n + (4-2n)\delta}) \quad \text{as } y \searrow 0,
\]
almost everywhere in time $t$, so that, introducing the distance $\tilde z := Z(t,y) - Z(t,0)$, by inversion
\[
y = (1+u_0(t)) \tilde z + \tfrac{(1+u_0(t))^{3-2n} u_\beta(t)}{(4-2n)^{\frac{8-5n}{2-n}}} \tilde z^{4-2n} + o(\tilde z^{4-2n + (4-2n)\delta}) \quad \text{as } \tilde z \searrow 0,
\]
almost everywhere in time $t$. With \eqref{eq_transform_1<n<3/2} this entails for the film height
\[
h(t,y) = (1+u_0(t))^2 \tilde z^2 \Big(1 + \tfrac{2 (1+u_0(t))^{2-2n} u_\beta(t)}{(4-2n)^{\frac{8-5n}{2-n}}} \tilde z^{3-2n} + o(\tilde z^{3-2n + (4-2n)\delta})\Big) \quad \text{as } \tilde z \searrow 0,
\]
almost everywhere in time $t$. This is qualitatively the same leading- and next-to-leading order asymptotic behavior as in \cite[Theorem~3.1]{BelgacemGnannKuehn2016} for source-type self-similar solutions.

\paragraph{Case $n \in (\frac 3 2,3)$}
We note that
\begin{equation}\label{as-h-y-32n3}
H \stackrel{\eqref{def-u-32n3}}{=} 1+u \stackrel{\eqref{partial_regularity}}{=} 1 + u_0(t) + u_\beta(t) x^\beta + o(x^{\beta+\delta}) \quad \text{as } x \searrow 0, \text{ almost everywhere in time } t.
\end{equation}
From \eqref{zt-32n3} we infer
\[
Z_t = (x \partial_x + \tfrac 3 n) H (x \partial_x + \tfrac 3 n - 2) H (x \partial_x + \tfrac 3 n - 1) H,
\]
which entails with \eqref{as-h-y-32n3} and analogous expressions for $x\partial_x$ derivatives
\[
Z_t(t,0) = - \tfrac 3 n (\tfrac 3 n - 1) (2 - \tfrac 3 n) (1+u_0(t))^3
\]
for the velocity of the contact line. In particular, we have $Z_t(t,0) \to (3 - \tfrac 3 n) (2 - \tfrac 3 n) (\tfrac 3 n - 1)$ as $t \to \infty$. From \eqref{def-h-32n3} and \eqref{as-h-y-32n3} it then follows
\[
Z_x = \tfrac{1}{1+u_0(t)} - \tfrac{u_\beta(t)}{(1+u_0(t))^2} x^\beta + o(x^{\beta+\delta}) \quad \text{as } x \searrow 0,
\]
almost everywhere in time $t$. Integration gives
\[
Z(t,x) = Z(t,0) + \tfrac{1}{1+u_0(t)} x - \tfrac{u_\beta(t)}{(1+\beta) (1+u_0(t))^2} x^{1+\beta} + o(x^{1+\beta+\delta}) \quad \text{as } x \searrow 0,
\]
almost everywhere in time $t$, so that after inversion it holds for the distance $\tilde z := Z(t,x) - Z(t,0)$ to the free boundary
\[
x = (1+u_0(t)) \tilde z + \tfrac{u_\beta(t) (1+u_0(t))^{\beta}}{1+\beta} \tilde z^{1+\beta} + o(\tilde z^{1+\beta+\delta}) \quad \text{as } \tilde z \searrow 0.
\]
With help of \eqref{eq_transform_>3/n} this amounts to
\[
h = (1+u_0(t))^{\frac 3 n} \tilde z^{\frac 3 n} \big(1 + \tfrac{3 u_\beta(t) (1+u_0(t))^{\beta-1}}{n (1+\beta)} \tilde z^\beta + o(\tilde z^{\beta+\delta})\big) \quad \text{as } \tilde z \searrow 0.
\]
This is in line with the asymptotics for source-type self-similar solutions as in \cite[Theorem~1]{GiacomelliGnannOtto2013} in leading- and next-to-leading order.

\subsubsection{Further questions}
Here, we only mention two selected questions of interest for future work, the first concerning non-Newtonian thin-film equations as studied for instance in \cite{AnsiniGiacomelli2004}. Power-law or Ostwald-de-Waele liquids lead to a changed degeneracy compared to the one in \eqref{TFEa}, so that the leading-order terms in the dynamics around a steady state such as a stationary solution have features of a $p$-Laplacian including a degeneracy. It is widely open how to transfer the techniques developed here for fourth-order degenerate-parabolic equations to such settings.

\medskip

The second question concerns droplet rupture as studied for instance in \cite{ConstantinDupontGoldsteinKadanoffShelleyZhou1993,ConstantinElgindiNguyenVicol2018}. In \cite{ConstantinDupontGoldsteinKadanoffShelleyZhou1993} it is shown that under suitable boundary conditions, initially positive solutions to \eqref{TFEa} with $n = 1$ pinch off in finite or infinite time. To our knowledge, a corresponding result for nonlinear mobilities is lacking so far. We further point out that our setting may allow for an extension of such solutions beyond pinch off.

\section{Maximal Regularity for the Linear Problem}\label{section_4}
Consider the linear inhomogeneous Cauchy problem
\begin{subequations}\label{eq_full_cauchy}
\begin{alignat}{2}
    \partial_t u-Au &= f\qquad  &&\text{for } t,x>0,\\
    u &=u^{(0)} \qquad&&\text{for } x>0 \text{ and at } t = 0,
\end{alignat}
\end{subequations}
where
\begin{subequations}\label{def-a-op}
\begin{equation}\label{a-op-pd}
A=-x^{-1}\mathfrak p(D)
\end{equation}
and where $\mathfrak p(D)$ is the fourth-order polynomial operator as introduced in \S\ref{sec_Non_Cauchy_Problem}. Here,
\begin{equation}\label{a-op-domain}
A\colon X\supset D(A)\to X \quad \text{with } X:=H_{k-2,\alpm}^\K \text{ and } D(A) := H_{k-2,\alpha-1}^\K \cap H_{k+2,\alpha}^\K.
\end{equation}
\end{subequations}

The next lemma characterizes the elliptic regularity and closedness of the operator $A$.
\begin{lemma}\label{lem_ineq_hardy}
Suppose $k \in \N$ with $k \ge 2$, and $\alpha \in \R \setminus \{\gamma \colon \gamma \text{ is a root of } \mathfrak p(\zeta)\}$. Then we have for $A \stackrel{\eqref{a-op-pd}}{=} -x^{-1}\mathfrak p(D)$ that
\begin{equation}\label{eq_ineq_hardy}
    \verti{Au}_{k-2,\alpm} \sim_{k,\alpha} \verti{u}_{k+2,\alp} \quad \text{for all } u\in D(A).
\end{equation}
In particular, $A$ is a closed and densely defined operator.
\end{lemma}
\begin{proof}
We prove \eqref{eq_ineq_hardy} by applying Hardy's inequality \cite[Theorem~1]{Masmoudi2011} iteratively for $u\in C_\mathrm{c}^{\infty}((0,\infty);\K)$. Note that $\mathfrak p(D)$ is a fourth-order polynomial which we will rewrite as $\mathfrak p(D)=(D-\gamma)\Tilde{\mathfrak p}(D)$, where $\gamma$ is one of the zeros of $\mathfrak p(D)$ and $\Tilde{\mathfrak p}(D)$ is the remaining third-order polynomial. Since  we have $\gamma \ne \alpha$, by application of Hardy's inequality we get
\begin{align*}
    \verti{Au}_{k-2,\alpm}^2 &= \verti{x^{-1+\gamma} D x^{-\gamma} \tilde{\mathfrak p}(D)u}^2_{k-2,\alpm} \\
    &\sim_{k,\alpha} \verti{D x^{-\gamma} \tilde{\mathfrak p}(D)u}^2_{k-2,\alpha-\gamma} \\
    &= \sum_{j=0}^{k-2}\int_0^{\infty}x^{2\gamma-2\alpha}(D^{j+1} x^{-\gamma}\Tilde{\mathfrak p}(D)u)^2\tfrac{\d x}{x}\\
    &\gtrsim \sum_{j=0}^{k-2}\int_0^{\infty}x^{2\gamma-2\alpha}( D^j x^{-\gamma}\Tilde{\mathfrak p}(D)u)^2\tfrac{\d x}{x}\\
    &= \sum_{j=0}^{k-2}\int_0^{\infty}x^{-2\alpha}( (D-\gamma)^j \Tilde{\mathfrak p}(D)u)^2\tfrac{\d x}{x} \sim_{k,\alpha} \verti{\Tilde{\mathfrak p}(D)u}_{k-2,\alp}^2.
\end{align*}
We also have
\begin{equation*}
    \verti{Au}_{k-2,\alpm} \sim_{k,\alpha} \verti{(D-\gamma)\Tilde{\mathfrak p}(D)}_{k-2,\alp}\geq\verti{D\Tilde{\mathfrak p}(D)u}_{k-2,\alp}-|\gamma|\verti{\Tilde{\mathfrak p}(D)u}_{k-2,\alp},
\end{equation*}
and combining the two inequalities gives
\begin{align*}
\verti{Au}_{k-2,\alpm} \gtrsim_{k,\alpha} \verti{\Tilde{\mathfrak p}(D)u}_{k-1,\alp}.
\end{align*}
Repeating this argument three more times gives the desired estimate \eqref{eq_ineq_hardy}. Since trivially $\verti{Au}_{k-2,\alpm} \lesssim_{k,\alpha} \verti{u}_{k+2,\alpha}$, by density of $C_\mathrm{c}^{\infty}((0,\infty);\K)$ in $H_{k-2,\alpm}^\K \cap H_{k+2,\alp}^\K$ we obtain \eqref{eq_ineq_hardy} also for $u \in H_{k-2,\alpm}^\K \cap H_{k+2,\alp}^\K$.
\end{proof}

The Cauchy problem \eqref{eq_full_cauchy} has the formal mild solution
\begin{equation*}
    u(t)=e^{tA}u^{(0)}+\int_0^t e^{(t-s)A}f(s)\d s,
\end{equation*}
which suggests to split \eqref{eq_full_cauchy} into
\begin{subequations}\label{eq_cauchy,u(0)neq0}
\begin{align}
    \partial_t u^{(1)} -A u^{(1)} = 0, \label{hom_eq}\\
    u^{(1)}(0)=u^{(0)},
\end{align}
\end{subequations}
with formal solution $u^{(1)}(t)=e^{tA}u^{(0)}$, and
\begin{subequations}\label{eq_cauchy,f}
\begin{align}
    \partial_t u^{(2)}-Au^{(2)}=f,\label{pde_u2}\\
    u^{(2)}(0)=0,
\end{align}
\end{subequations}
with formal mild solution $u^{(2)}(t)=\int_0^t e^{(t-s)A}f(s)\d s$.

\medskip

The goal is to prove a maximal-regularity estimate for the linear problem \eqref{eq_full_cauchy}. To achieve this, the two problems \eqref{eq_cauchy,u(0)neq0} and \eqref{eq_cauchy,f} will be dealt with separately. The corresponding resolvent problem of \eqref{eq_cauchy,f} will be treated in \S\ref{sec_inhomogeneous_equation} and problem \eqref{eq_cauchy,u(0)neq0} is treated in \S\ref{sec_homogeneous_equation} with standard semi-group and interpolation theory. In \S\ref{sec_Linear_Maximal_Reg_est} the results will be combined to obtain a maximal $L^p$-regularity estimate. The section is concluded in \S\ref{sec_higher_regularity}, in which elliptic regularity estimates yield control on the singular expansion of $u$.

\subsection{Inhomogeneous Equation}\label{sec_inhomogeneous_equation}
Applying the Laplace transform in time to \eqref{eq_full_cauchy} and absorbing the Laplace-transformed initial data $u^{(0)}$ into $f$, we obtain the resolvent equation
\begin{equation}\label{eq_resolvent_equation}
\lambda u -Au = f,
\end{equation}
with $\lambda \in \{z\in \C\; |\; \Re z\geq 0\}$ and $A \stackrel{\eqref{a-op-pd}}{=} -x^{-1}\mathfrak p(D)$, with the fourth-order polynomial $p$. We test \eqref{eq_resolvent_equation} in $(\cdot,\cdot)_{\alpha}$ with
\begin{equation}\label{eq_test function}
\phi := \sum_{j=0}^k c_j(-D+2\alpha -1)^j D^j v,\qquad c_j>0,
\end{equation}
where $v \in C_\mathrm{c}^{\infty}((0,\infty);\C)$. For $w \in C^\infty_\mathrm{c}((0,\infty);\C)$ we define the bilinear form
\begin{align}
    B_{\lambda,k,\alpha}(v,w) &:= \left(\phi, \lambda w - A w\right)_{\alpha-\frac 1 2} \nonumber \\
    &\phantom{:}= \sum_{j=0}^k c_j \big((-D+2\alpha-1)^jD^j v,\lambda w - A w\big)_{\alpha-\frac 1 2}, \label{def-bkalpha}
\end{align}
where $B_{\lambda,k,\alpha}\colon \mathscr{H}_{k,\alpha}\times\mathscr{H}_{k,\alpha}\to \C$.  The choice of the test function $\phi$ in \eqref{eq_test function} is motivated by looking for solutions with control on $k+2$ $D$-derivatives and the shift by $1-2\alpha$ ensures symmetry of the term proportional to $\lambda$ in $B_{\lambda,k,\alpha}(v,w)$ when integrating by parts. The constants $c_j$ are suitably chosen below to obtain coercivity. Our goal is to show that the resolvent equation has a unique classical solution and that $A \stackrel{\eqref{a-op-pd}}{=} -x^{-1}\mathfrak p(D)$ generates an analytic semigroup. We start by showing that the resolvent equation \eqref{eq_resolvent_equation} has a unique weak solution by using the Lax-Milgram theorem (see e.g. \cite[Theorem~6.2]{Alt}).
\begin{lemma}[cf.~\cite{GGKO}]\label{lem_coercive}
The bilinear form
\[
C^\infty_\mathrm{c}((0,\infty);\C) \times C^\infty_\mathrm{c}((0,\infty);\C) \owns (v,w) \mapsto (v, \mathfrak p(D)w)_{\alpha}
\]
is coercive with respect to $\verti{\cdot}_{\alpha}$, i.e., $\Re(w,\mathfrak p(D)w)_{\alp} \gtrsim_\alpha \verti{w}_{2,\alp}^2$ for $w \in C_\mathrm{c}^\infty((0,\infty))$, if the following conditions hold:
\begin{subequations}\label{conditions_coerc_range}
\begin{align}
    \alpha \in (-\infty,\gamma_1)\cup (\gamma_2,\gamma_3)\cup(\gamma_4,\infty)\label{eq_condition_coerc_1},\\
    |\alpha-m(\gamma)|\leq \tfrac{1}{\sqrt{3}}\sigma(\gamma).\label{eq_condition_coerc_2}
\end{align}
\end{subequations}
Here, $m(\gamma)$ denotes the algebraic mean of the zeros $\gamma_l$ of $\mathfrak p(D)$, i.e.,
\begin{equation*}
    m(\gamma) := \tfrac{1}{4}\sum_{l=1}^4 \gamma_l,
\end{equation*}
and $\sigma(\gamma)$ the nonnegative root of the variance
\begin{equation*}
    \sigma^2(\gamma) = \tfrac{1}{4}\sum_{l=1}^4 \gamma^2_l - m^2(\gamma) = \tfrac{1}{4}\sum_{l=1}^4 (\gamma_l - m(\gamma))^2.
\end{equation*}
\end{lemma}
\begin{proof} This is a straight-forward computation, see \cite[Proposition~5.3]{GGKO}.
\end{proof}
Note that (as one can see from the proof in \cite[Proposition~5.3]{GGKO}), the criterion \eqref{conditions_coerc_range} is sufficient but in general not necessary to prove coercivity, see the discussion in \S\ref{sec-conclusion}.

\medskip

The above lemma amounts in our case to:
\begin{proof}[Proof of Lemma~\ref{lem_coercive_range}]
First consider $n\in (1,\tfrac{3}{2})$. By \eqref{eq_zeros_1} the zeros of $\mathfrak p(D)$ are in ascending order given by
\[
\gamma_1 = -\tfrac{1}{2-n}, \quad \gamma_2 = \tfrac{1-2n}{2(2-n)}, \quad \gamma_3 = 0, \quad \gamma_4 = \tfrac{3-2n}{2(2-n)}.
\]
Hence, condition~\eqref{eq_condition_coerc_1} gives
\begin{equation*}
    \alp \in \big(-\infty,-\tfrac{1}{2-n}\big)\cup \big(\tfrac{1-2n}{2(2-n)},0\big)\cup \big(\tfrac{3-2n}{2(2-n)},\infty\big).
\end{equation*}
Furthermore, an elementary calculation shows
\begin{equation*}
    m(\gamma) = \tfrac{1-2n}{4(2-n)}\quad \text{ and }\quad
    \sigma^2(\gamma) = \tfrac{13-12n+4n^2}{16(2-n)^2}.
\end{equation*}
Thus \eqref{eq_condition_coerc_2} reads in the case at hand
\begin{equation}\label{eq_coer_condition2_n_1_32}
    \big|\alp - \tfrac{1-2n}{4(2-n)}\big|\leq \tfrac{1}{\sqrt{3}} \tfrac{\sqrt{13-12n+4n^2}}{4 (2-n)}.
\end{equation}
By combining the above criteria, we obtain that the coercivity range contains \eqref{coerc_range_n132}, see Figure~\ref{fig:coercivity_range_a}.

\medskip

For the case $n\in(\tfrac{3}{2},3)$ we have the zeros \eqref{eq_zeros_2}
\begin{equation*}
\gamma_1=-\tfrac{3}{n},\quad \gamma_2= \tfrac{4n-9-\sqrt{-27+36n-8n^2}}{2n} ,\quad \gamma_3=0,\quad \gamma_4= \tfrac{4n-9+\sqrt{-27+36n-8n^2}}{2n},
\end{equation*}
ordered from smallest to largest. Hence, to satisfy \eqref{eq_condition_coerc_1} we get the condition
\begin{equation*}
    \alp \in \big(-\infty,-\tfrac{3}{n}\big)\cup (\gamma_2,0)\cup (\gamma_4,\infty),
\end{equation*}
and
\begin{equation*}
    m(\gamma) = \tfrac{n-3}{n}\quad \text{ and }\quad
    \sigma^2(\gamma) = \tfrac{3}{2n},
\end{equation*}
which gives the second constraint \eqref{eq_condition_coerc_2}, that is,
\begin{equation}\label{eq_coer_condition2_n_32_3}
    \big|\alpha-\tfrac{n-3}{n}\big|\leq \tfrac{1}{\sqrt{2n}}.
\end{equation}
We thus need to take \eqref{coerc_range_n323} to fulfill both conditions, see Figure~\ref{fig:coercivity_range_b}.
\end{proof}
\begin{lemma}\label{lem_bdd+coercivity}
For $\alp$ in the coercivity range of $\mathfrak p(D)$ (cf.~\eqref{coerc_range} of Lemma~\ref{lem_coercive_range} for a sufficient criterion) and $K_\alpha$ as in \eqref{coercivity-pd}, there exist $d_{j,\alpha} \in (0,\infty)$, $j = 0,\ldots,k$, independent of $\boldsymbol{c} := (c_j)_{j = 1}^k$, such that if
\begin{equation}\label{constraint-cj}
c_j \ge \tfrac{2}{K_\alpha} \Big(1+\sum_{\ell = j+1}^k c_\ell d_{\ell,\alpha}\Big) \quad \text{for $j = 0,\ldots,k$},
\end{equation}
and if $\lambda \in \C$ with $\Re\lambda \ge 0$, the problem
\begin{equation}\label{lax-milgram-eq}
B_{\lambda,k,\alpha}(v,u) = \langle v, \overline f \rangle_{H_{k+2,\alpha}^\C \times H_{k-2,\alpha-1}^\C} \quad \text{for all } v \in \mathscr{H}_{k,\alpha}
\end{equation}
has for all $f \in H_{k-2,\alpha-1}^\C$ a unique solution $u \in \mathscr{H}_{k,\alpha}$ if $\lambda \ne 0$ and $u \in H_{k+2,\alpha}^\C$ if $\lambda = 0$. This solution satisfies the estimate
\begin{equation}\label{eq_strong_coerc_estimate}
    \verti{\lambda} \verti{u}_{k,\alpha-\frac 1 2}^2 + \verti{u}^2_{k+2,\alp} \lesssim_{\boldsymbol{c},k,\alpha} \verti{f}^2_{k-2,\alpha-1}.
\end{equation}
\end{lemma}
\begin{proof}
We assume that $v,w \in C_\mathrm{c}^{\infty}((0,\infty);\C)$ and prove boundedness and coercivity, enabling us to apply the Lax-Milgram theorem. We first note that through integration by parts it follows
\begin{align}
B_{\lambda,k,\alpha}(v,w) \ \ \quad &\stackrel{\mathclap{\eqref{a-op-pd},\eqref{def-bkalpha}}}{=} \quad \ \ \sum_{j=0}^k c_j \lambda((-D+2\alpha-1)^jD^jv,w)_{\alpha-\frac{1}{2}} \nonumber \\
&\phantom{=} \quad \ \ + \sum_{j=0}^k c_j ((-D+2\alpha-1)^jD^jv, \mathfrak p(D) w)_\alpha \nonumber \\
    &\stackrel{\mathclap{\eqref{eq_def_p(D)}}}{=} \quad \ \ \sum_{j=0}^k c_j \lambda(D^j v,D^j w)_{\alpha-\frac{1}{2}} \nonumber \\
    &\phantom{=} \quad \ \ + \sum_{j=0}^k c_j (D^j (D-2\alpha+\gamma_1) (D-2\alpha+\gamma_2) v, (D-1)^j (D-\gamma_3) (D-\gamma_4) w)_\alpha. \label{rewrite-blka}
\end{align}
\medskip

\textbf{Boundedness.} Using the Cauchy-Schwarz inequality it follows from \eqref{rewrite-blka} that
\begin{align}\label{lax-milgram-bound}
    |B_{\lambda, k, \alpha}(v,w) |&\lesssim_{\mathbf c, \boldsymbol\gamma} |\lambda| \verti{v}_{k,\alpha-\tfrac{1}{2}}\verti{w}_{k,\alpha-\tfrac{1}{2}}+\verti{v}_{k+2,\alp}\verti{w}_{k+2,\alp} \stackrel{\eqref{script-hka}}{\lesssim} \verti{v}_{\mathscr{H}_{k,\alpha}}\verti{w}_{\mathscr{H}_{k,\alpha}},
\end{align}
where $\boldsymbol\gamma := (\gamma_j)_{j = 1}^4$.

\medskip

\textbf{Coercivity.} For $v = w$ we obtain from \eqref{rewrite-blka} that
\begin{equation}\label{coercivity-bil}
\Re(B_{\lambda, k,\alpha}(v,v)) = \sum_{j=0}^k c_j \Re\lambda \verti{D^j v}_{\alpha-\frac{1}{2}}^2 + \sum_{j=0}^k c_j \Re\big((D^j v, \mathfrak p(D) (D-1)^j v)_\alpha\big).
\end{equation}
We notice that because of $c_j > 0$ for all $j = 0,\ldots,k$ we have
\begin{equation}\label{coerc-proof-1}
\sum_{j=0}^k c_j \Re\lambda \verti{D^j v}_{\alpha-\frac{1}{2}}^2 \gtrsim_{\boldsymbol{c},k} \Re\lambda \verti{v}_{k,\alpha-\frac 1 2}^2.
\end{equation}
For the last term in \eqref{coercivity-bil} integration by parts and coercivity (see Lemma~\ref{lem_coercive_range}) entail
\begin{align*}
    &\sum_{j=0}^k c_j\left(D^j v, \mathfrak p(D) (D-1)^j v\right)_{\alpha} \\
    & \quad = \ \sum_{j = 0}^k c_j (D^j v, \mathfrak p(D) D^j v)_{\alp} + \sum_{j = 0}^k \sum_{l=0}^{j-1} c_j (-1)^{j-l}{j \choose l} (D^j v, \mathfrak p(D) D^l v)_{\alp} \\
    & \quad \stackrel{\mathclap{\eqref{coercivity-pd}}}{\geq} \ \sum_{j = 0}^k c_j K_\alpha \verti{D^j v}^2_{2,\alp} \\
    & \quad \phantom{\le} \  - \sum_{j = 0}^k  \sum_{l=0}^{j-1} c_j {j \choose l} |(D^j (D-2\alpha+\gamma_1) (D-2\alpha+\gamma_2) v, D^l (D-\gamma_3) (D-\gamma_4) v)_{\alp}|,
\end{align*}
where $K_\alpha$ is the coercivity constant in \eqref{coercivity-pd}. By Young's inequality we infer that there exist $d_{j,\alpha} < \infty$ for $j = 1,\ldots,k-1$ such that
\begin{align*}
& \sum_{l = 0}^{j-1} {j \choose l} |(D^j (D-2\alpha+\gamma_1) (D-2\alpha+\gamma_2) v, D^l (D-\gamma_3) (D-\gamma_4) v)_{\alp}| \\
& \quad \le \tfrac{c_j K_\alpha}{2} \verti{D^j v}_{2,\alpha}^2 + c_j d_{j,\alpha} \verti{v}_{j+1,\alpha}^2,
\end{align*}
where $d_{j,\alpha}$ is independent of $c_j$. This entails
\[
\sum_{j=0}^k c_j\left(D^j v, \mathfrak p(D) (D-1)^j v\right)_{\alpha} \ge \sum_{j = 0}^k \tfrac{c_j K_\alpha}{2} \verti{D^j v}^2_{2,\alp} - \sum_{j = 1}^k c_j d_{j,\alpha} \verti{v}_{j+1,\alpha}^2.
\]
We define inductively in descending order in $j \in \{0,\ldots,k\}$ the coefficients $c_j$ meeting \eqref{constraint-cj}, so that
\begin{equation}\label{coerc-proof-2}
\sum_{j=0}^k c_j\left(D^j v, \mathfrak p(D) (D-1)^j v\right)_{\alpha} \ge \verti{v}_{k+2,\alpha}^2.
\end{equation}
Estimates~\eqref{coerc-proof-1} and \eqref{coerc-proof-2} in \eqref{coercivity-bil} entail
\begin{equation}\label{coerc_est_re}
\Re(B_{\lambda, k,\alpha}(v,v))
\gtrsim_{\boldsymbol{c},k} \Re \lambda \verti{v}^2_{k,\alpha-\frac{1}{2}} + \verti{v}^2_{k+2,\alp},
\end{equation}
where $\mathbf{c}$ depends on $\alpha$. Furthermore,
\begin{align}
\verti{\Im\lambda} \verti{v}^2_{k,\alpha-\frac{1}{2}} &\lesssim_{\boldsymbol{c}} \sum_{j=0}^k c_j \verti{\Im\lambda} \verti{D^j v}_{\alpha-\frac{1}{2}}^2 \nonumber \\
&\stackrel{\mathclap{\eqref{a-op-pd},\eqref{lax-milgram-bound}}}{\le} \quad \ \ \verti{\Im(B_{\lambda, k,\alpha}(v,v))} + \sum_{j=0}^k c_j \verti{\Im\big((D^j v, \mathfrak p(D) (D-1)^j v)_\alpha\big)} \nonumber \\
&\stackrel{\mathclap{\eqref{coerc_est_re}}}{\lesssim_{\boldsymbol{c},k}} \verti{B_{\lambda, k,\alpha}(v,v)} \label{coerc_est_im}
\end{align}
Estimates~\eqref{coerc_est_re} and \eqref{coerc_est_im} entail
\begin{equation}\label{coerc_est_full}
\verti{B_{\lambda, k,\alpha}(v,v)}
\gtrsim_{\boldsymbol{c},k} \verti{\lambda} \verti{v}^2_{k,\alpha-\frac{1}{2}} + \verti{v}^2_{k+2,\alp}.
\end{equation}
\medskip

\textbf{Lax-Milgram Solution.} From \eqref{lax-milgram-bound} and \eqref{coerc_est_full} we obtain a unique solution $u \in \mathscr{H}_{k,\alpha}$ if $\lambda \ne 0$ and  $u \in H_{k+2,\alpha}^\C$ if $\lambda = 0$ of \eqref{lax-milgram-eq}. Setting $u = v$ in \eqref{lax-milgram-eq} and using \eqref{coerc_est_full} yields \eqref{eq_strong_coerc_estimate}.
\end{proof}
\begin{lemma}\label{lemma_density}
The set
\begin{equation*}
    \bigg\{\sum_{j=0}^k c_j (-D+2\alpha-1)^j D^j v|\; v\in C_\mathrm{c}^{\infty}((0,\infty);\C)\bigg\}
\end{equation*}
is dense in $H_{\alpha-\frac 1 2}^\C$ if
\begin{equation}\label{cond-cj-2}
c_j > - \sum_{j<m\leq k}c_m {m \choose 2(m-j)}(-1)^{m-j}(2\alpha-1)^{2(m-j)} \quad \text{for all $j = 0,\ldots,k$.}
\end{equation}
\end{lemma}
\begin{proof}
Note that by commuting with $x^{-2\alpha+1}$, the assertion of the lemma is equivalent to
\begin{equation*}
    \bigg\{\sum_{j=0}^k c_j (-D)^j(D+2\alpha-1)^j x^{-2\alpha+1} v|\; v \in C_\mathrm{c}^{\infty}((0,\infty);\C)\bigg\}
\end{equation*}
being dense in $L^2(\R;\C)$. After passing to $x^{-2\alpha+1} v(x) =\psi(e^s)$ with $s =\log x$, this is equivalent to
\begin{equation*}
    \bigg\{\sum_{j=0}^k c_j (-\partial_s)^j(\partial_s+2\alpha-1)^j \psi|\;\psi\in \mathcal{S}(\R)\bigg\}
\end{equation*}
being dense in $L^2(\R;\C)$. Plancherel's theorem entails that it is sufficient to show that
\begin{equation*}
    \bigg\{\sum_{j=0}^k c_j (-i\xi)^j(i\xi+2\alpha-1)^j \psi|\;\psi\in \mathcal{S}(\R)\bigg\}
\end{equation*}
is dense in $L^2(\R;\C)$. Therefore, we show that the mapping
\begin{equation}
   T \colon \mathcal S(\R) \to \mathcal S(\R), \qquad \psi\mapsto \sum_{j=0}^k c_j (-i\xi)^j(i\xi+2\alpha-1)^j \psi
\end{equation}
is surjective. For this, take $\tilde \psi \in \mathcal S(\R)$. Then it suffices to show that the polynomial
\[
\mathfrak q(\xi) := \sum_{j=0}^k c_j (-i\xi)^j(i \xi +2\alpha-1)^j
\]
has no real zeros in $\xi$ since $\psi(\xi) = \tilde{\psi}(\xi)/\mathfrak q(\xi)$ satisfies $T \psi = \tilde \psi$ and $\psi \in \SSS(\R)$. Indeed, we have
\begin{align*}
    \Re \sum_{j=0}^k c_j (-i\xi)^j(i \xi +2\alpha-1)^j &= \sum_{j=0}^k c_j \Re (\xi^2-i\xi(2\alpha -1))^j\\
    &= \sum_{j=0}^k c_j \sum_{l=0}^{\lfloor\tfrac{j}{2} \rfloor}{j \choose 2l}\xi^{2(j-l)}(-1)^l (2\alpha-1)^{2l}\\
    &=: d_0 + d_1 \xi^2 + d_2 \xi^4 +\dots + d_k \xi^{2k}.
\end{align*}
We find
\begin{align*}
    d_{j} = c_j +\sum_{j<m\leq k}c_m {m \choose 2(m-j)}(-1)^{m-j}(2\alpha-1)^{2(m-j)},
\end{align*}
which is positive if \eqref{cond-cj-2} holds true. It then follows that
\begin{equation*}
    d_0 + d_1 \xi^2 + d_2 \xi^4 +\dots + d_k \xi^{2k} >0 \quad \text{for all $\xi \in \R$,}
\end{equation*}
so that $\mathfrak q(\xi)$ has no real zeros.
\end{proof}
\begin{lemma}\label{lem-resolvent}
For $\alp$ in the coercivity range of $\mathfrak p(D)$ (cf.~\eqref{coerc_range} of Lemma~\ref{lem_coercive_range} for a sufficient criterion), $k \ge 2$, $K_\alpha$ as in \eqref{coercivity-pd}, and $\lambda \in \C$ with $\Re\lambda \ge 0$, the problem
\begin{equation}\label{resolvent-problem}
\lambda u - A u = f
\end{equation}
has for all $f \in H_{k-2,\alpha-1}^\C$ a unique strong solution $u \in H_{k-2,\alpha-1}^\C \cap \mathscr{H}_{k,\alpha}$ if $\lambda \ne 0$ and $u \in H_{k+2,\alpha}^\C$ if $\lambda = 0$. This solution satisfies the estimate
\begin{equation}\label{eq_strong_coerc_estimate_2}
   \verti{\lambda} \verti{u}_{k-2,\alpha -1}+ \verti{\lambda}^{\frac 1 2} \verti{u}_{k,\alpha-\frac 1 2} + \verti{u}_{k+2,\alp} \lesssim_{k,\alpha} \verti{f}_{k-2,\alpha-1}.
\end{equation}
\end{lemma}
\begin{proof}
We choose $\boldsymbol{c} = (c_j)_{j = 0}^k$ iteratively in descending order in $j$ such that \eqref{constraint-cj} and \eqref{cond-cj-2} are satisfied. By Lemma~\ref{lem_bdd+coercivity} it follows that there exists a unique solution strong solution $u \in \mathscr{H}_{k,\alpha}$ if $\lambda \ne 0$ and $u \in H_{k+2,\alpha}$ if $\lambda = 0$ to \eqref{lax-milgram-eq}. In view of \eqref{def-bkalpha}, the density result, Lemma~\ref{lemma_density}, together with $(\lambda u - A u - f) \ind_K \in H_{\alpha-\frac 1 2}^\C$ for any $K \Subset (0,\infty)$, entails that \eqref{resolvent-problem} must be satisfied in the strong sense. Furthermore,
\begin{align*}
|\lambda| \verti{u}_{k-2,\alpha-1} \stackrel{\eqref{resolvent-problem}}{\le} \verti{f}_{k-2,\alpha-1} + \verti{Au}_{k-2,\alpha-1} \stackrel{\eqref{a-op-pd}}{\lesssim} \verti{f}_{k-2,\alpha-1} + \verti{u}_{k+2,\alpha}
    \stackrel{\eqref{eq_strong_coerc_estimate}}{\lesssim_{k,\alpha}} \verti{f}_{k-2,\alpha-1},
\end{align*}
that is, the resolvent estimate holds for all $\lambda \in \C$ such that $\Re\lambda \ge 0$. This inequality in combination with \eqref{eq_strong_coerc_estimate} entails \eqref{eq_strong_coerc_estimate_2}.
\end{proof}
\begin{proposition}[Maximal regularity, homogeneous initial data]\label{prop_MR}
Suppose that $\alp$ is in the coercivity range of $\mathfrak p(D)$ (cf.~\eqref{coerc_range} of Lemma~\ref{lem_coercive_range} for a sufficient criterion) and $k \ge 2$. Then the operator $A:=-x^{-1}\mathfrak p(D)\colon H_{k-2,\alpm} \supset D(A) \to H_{k-2,\alpm}$ with $D(A) \stackrel{\eqref{a-op-domain}}{=} H_{k-2,\alpm} \cap  H_{k+2,\alp}$ is the generator of a bounded analytic semigroup. In particular, $A$ has maximal $L^p$-regularity, where $1 < p < \infty$. This entails that for every $f \in L^p(0,\infty;H_{k-2,\alpm})$ there exists a unique solution $u^{(2)}$ to \eqref{eq_cauchy,f} such that $\partial_t u^{(2)} \in L^p(0,\infty;H_{k-2,\alpm})$ and $A u^{(2)} \in L^p(0,\infty;H_{k-2,\alpm})$. This solution satisfies the maximal-re\-gu\-la\-ri\-ty estimate
\begin{equation}\label{mr-homogeneous-initial}
\|\partial_t u^{(2)}\|_{L^p(0,\infty;H_{k-2,\alpm})}+\|Au^{(2)}\|_{L^p(0,\infty;H_{k-2,\alpm})} \lesssim_{k,\alpha, p} \|f\|_{L^p(0,\infty;H_{k-2,\alpm})}.
\end{equation}
\end{proposition}
\begin{proof}
For $\lambda \in \C$ with $\Re\lambda \ge 0$, Lemma~\ref{lem-resolvent} yields for every $f \in H_{k-2,\alpha-1}^\C$ a unique solution $u \in \mathscr H_{k,\alpha}^\C$ for $\lambda \ne 0$ and $u \in H_{k+2,\alpha}^\C$ if $\lambda = 0$ to \eqref{resolvent-problem} satisfying \eqref{eq_strong_coerc_estimate_2}. From \eqref{eq_strong_coerc_estimate_2} it follows that the resolvent estimate holds in the half plane $\{\lambda\in\C: \Re \lambda \geq 0\}$. By \cite[Proposition~2.1.11]{Lunardi}, we know that this now also holds in a sector with spectral angle larger than $\frac \pi 2$.  Thus $A$ generates a bounded analytic semigroup (\cite[Definition 2.0.2]{Lunardi}). From \cite{DeSimon1964} it then follows that $A$ has maximal $L^p$-regularity, thus in particular satisfying \eqref{mr-homogeneous-initial}.
\end{proof}
%

\subsection{Homogeneous Equation}\label{sec_homogeneous_equation}
In this subsection we will show estimates for \eqref{eq_cauchy,u(0)neq0}. For this we will often use the fact that $A:=-x^{-1}\mathfrak p(D)\colon H_{k-2,\alpm} \supset D(A)\to H_{k-2,\alpm}$, with $D(A) \stackrel{\eqref{a-op-domain}}{=} H_{k-2,\alpm}\cap H_{k+2,\alp}$,
generates an analytic semigroup (cf.~Proposition~\ref{prop_MR}). In what follows, we let $0 < \varrho < 1$, $1 < p < \infty$, and use the space (cf.~\cite[(2.2.3)]{Lunardi})
\begin{subequations}\label{def-dap}
\begin{equation}
D_A(\varrho,p) := \big\{v \in H_{k-2,\alpha-1} \colon \tau \mapsto w(\tau) := \lvert \tau^{1-\varrho-\frac 1 p} A e^{\tau A} v \rvert_{k-2,\alpha-1} \in L^p(0,1)\big\},
\end{equation}
with norm
\begin{equation}\label{eq_norm_D_A}
\vertii{v}_{D_A(\varrho,p)} = \verti{v}_{k-2,\alpha-1} + [v]_{D_A(\varrho,p)} := \verti{v}_{k-2,\alpha-1} + \vertii{w}_{L^p(0,1)}.
\end{equation}
\end{subequations}
\begin{proposition}[A-priori estimates, homogeneous right-hand side]\label{prop_interpol_ineq_T=infty}
For $1 < p < \infty$, $k \in \mathbb{N}$ and  $k \ge 2$, $\alpha$ in the coercivity range of $\mathfrak p(D)$ (cf.~\eqref{coerc_range} of Lemma~\ref{lem_coercive_range} for a sufficient criterion), and $u^{(0)} \in D_A(1-\frac 1 p,p)$, problem~\eqref{eq_cauchy,u(0)neq0} has a solution $u^{(1)} \colon (0,\infty)^2 \to \R$ such that $\partial_t u^{(1)}, A u^{(1)} \in L^p(0,\infty;H_{k-2,\alpm})$, and the a-priori estimate
\begin{equation}\label{apriori-u1}
    \vertii{\partial_t u^{(1)}}_{L^p(0,\infty;H_{k-2,\alpm})} + \|A u^{(1)}\|_{L^p(0,\infty;H_{k-2,\alpm})} \lesssim_{k,\alpha,p} \lVert u^{(0)} \rVert_{D_A(1-\frac{1}{p},p)}
\end{equation}
holds true.
\end{proposition}
\begin{proof}
Since by Proposition~\ref{prop_MR} the operator $A$ generates an analytic semigroup, estimate~\eqref{apriori-u1} follows from standard semigroup theory. We have by \cite[Proposition~2.1.1~(iv)]{Lunardi} that
\begin{align*}
  \vertii{\partial_t u^{(1)}}_{L^p(0,\infty;H_{k-2,\alpm})} &= \vertii{Au^{(1)}}_{L^p(0,\infty;H_{k-2,\alpm})}\\
   &= \vertii{Au^{(1)}}_{L^p(0,1;H_{k-2,\alpm})} +\vertii{Au^{(1)}}_{L^p(1,\infty;H_{k-2,\alpm})}.
\end{align*}
By definition, $\vertii{Au^{(1)}}_{L^p(0,1;H_{k-2,\alpm})} = [u^{(0)}]_{D_A(1-\frac{1}{p},p)}$. Furthermore,
\begin{align*}
  \vertii{Au^{(1)}}_{L^p(1,\infty;H_{k-2,\alpm})}^p = \int_1^{\infty} \verti{A e^{tA} u^{(0)}}^p_{k-2,\alpha - 1}\d t \lesssim \int_1^{\infty}t^{-p} |u^{(0)}|_{k-2,\alpha -1}\d t,
\end{align*}
where we have used \cite[Proposition~2.1.1~(iii)]{Lunardi} in the last step and that the growth bound $\omega$ therein is zero because of \eqref{eq_strong_coerc_estimate_2}. Estimate~\eqref{apriori-u1} follows from \eqref{eq_norm_D_A}.
\end{proof}

In the next lemma, we find a characterization for the $D_A(1-\frac{1}{p},p)$ space.
\begin{lemma}\label{lem-equiv-trace}
For $k \in \N$ with $k \ge 2$, $\alpha \in \R$, and $1 < p < \infty$ it holds
\begin{equation}\label{equiv-trace}
D_A(1-\tfrac{1}{p},p) = (H_{k-2,\alpha-1},D(A))_{1-\frac 1 p,p} = H_{k+2-\frac 4 p,\alpha - \frac 1 p, p} \cap H_{k-2,\alpha-1}
\end{equation}
with equivalent norms, where the respective constants of embeddings only depend on $p$.
\end{lemma}
\begin{proof}
The first equality in \eqref{equiv-trace} with equivalent norms follows from the trace method (cf.~\cite[Proposition 2.2.2]{Lunardi}). By \cite[Theorem~1.1]{Haasse} for any pair of Banach spaces $Y_1$, $Y_2$ it holds
\begin{equation}\label{eq_interpol_intersect}
(Y_1, Y_2\cap Y_1)_{1-\frac{1}{p},p} = (Y_1,Y_2)_{1-\frac{1}{p},p} \cap Y_1,
\end{equation}
which because of \eqref{def-trace-space} entails the second equality in \eqref{equiv-trace} with equivalent norms.
\end{proof}

We now give the proof of Lemma~\ref{lemma_char_interpol}, which in view of Lemma~\ref{lem-equiv-trace} leads to a characterization of the $D_A(1-\frac{1}{p},p)$ space for the initial data.

\begin{proof}[Proof of Lemma~\ref{lemma_char_interpol}]
For $u \in (H_{k-2,\alpm},H_{k+2,\alp})_{1-\vartheta,p}$ we have for $p \le 2$ by \cite[Theorem~3.4.1~(b)]{Bergh_Loefstroem} and with help of the $K$-method
\begin{align*}
    &\verti{u}_{k + 2-4\vartheta,\alpha - \vartheta,p}^2 \\
    & \quad \gtrsim_p \vertii{u}_{(H_{k-2,\alpm},H_{k+2,\alp})_{1-\vartheta,2}}^2 \\
    & \quad = \int_0^{\infty} \tau^{-2(1- \vartheta)} \inf_{u=u^{(1)}+u^{(2)}} \big(\verti{u^{(1)}}_{k-2,\alpm}+\tau\verti{u^{(2)}}_{k+2,\alp}\big)^2 \tfrac{\d \tau}{\tau} \\
    & \quad \sim \int_0^{\infty}\inf_{u=u^{(1)}+u^{(2)}} \big(\tau^{- 2(1 - \vartheta)} \verti{u^{(1)}}_{k-2,\alpm}^2 + \tau^{2 \vartheta} \verti{u^{(2)}}_{k+2,\alp}^2\big) \tfrac{\d \tau}{\tau} \\
    & \quad = \int_0^{\infty} \int_0^\infty \inf_{u=u^{(1)}+u^{(2)}} \Big(\tau^{-2(1-\vartheta)} x^{-2(\alpha-1)} \sum_{j = 0}^{k-2} (D^j u^{(1)})^2 + \tau^{2\vartheta} x^{-2\alpha} \sum_{j = 0}^{k+2} (D^j u^{(2)})^2\Big) \tfrac{\d \tau}{\tau} \, \tfrac{\d x}{x},
\end{align*}
where Fubini's theorem was used in the last step and equivalence holds in the first step and hereafter if $p = 2$. With help of the substitution $\tau \mapsto x \tau$ and again applying Fubini's theorem, we obtain
\begin{align*}
    &\verti{u}_{k + 2-4\vartheta,\alpha - \vartheta,p}^2 \\
    & \quad \gtrsim_p \int_0^{\infty} \int_0^\infty \inf_{u=u^{(1)}+u^{(2)}} \Big(\tau^{-2(1-\vartheta)} x^{-2(\alpha-\vartheta)} \sum_{j = 0}^{k-2} (D^j u^{(1)})^2 + \tau^{2\vartheta} x^{-2(\alpha-\vartheta)} \sum_{j = 0}^{k+2} (D^j u^{(2)})^2\Big) \tfrac{\d \tau}{\tau} \, \tfrac{\d x}{x} \\
    & \quad \sim_{k,\alpha,\vartheta} \int_0^{\infty} \inf_{(s \mapsto e^{-(\alpha-\vartheta) s} u(e^s)) = v^{(1)} + v^{(2)}} \int_\R \Big(\tau^{-2(1-\vartheta)} \sum_{j=0}^{k-2}(\partial_s^j v^{(1)})^2 + \tau^{2\vartheta} \sum_{j=0}^{k+2} (\partial_s^j v^{(2)})^2\Big) \d s \, \tfrac{\d \tau}{\tau} \\
    & \quad \sim \int_0^\infty \tau^{-2(1-\vartheta)} \inf_{(s \mapsto e^{-(\alpha-\vartheta) s} u(e^s)) = v^{(1)} + v^{(2)}} \big(\vertii{v^{(1)}}_{W^{k-2,2}(\R)} + \tau \vertii{v^{(2)}}_{W^{k+2,2}(\R)}\big)^2 \tfrac{\d \tau}{\tau} \\
    & \quad = \vertii{s \mapsto e^{-(\alpha-\vartheta) s} u(e^s)}_{(W^{k-2,2}(\R),W^{k+2,2}(\R))_{1-\vartheta,2}}^2 \sim \vertii{s \mapsto e^{-(\alpha-\vartheta) s} u(e^s)}_{W^{k+2-4\vartheta,2}(\R)}^2.
\end{align*}
For $p > 2$ we use the Sobolev embedding and obtain
\begin{align*}
&\verti{u}_{k + 2-4\vartheta,\alpha - \vartheta,p}^p \\
    & \quad \sim_p \int_0^\infty \inf_{u = u^{(1)} + u^{(2)}} \big(\tau^{-p(1-\vartheta)} \verti{u^{(1)}}_{k-2,\alpha-1}^p + \tau^{p\vartheta} \verti{u^{(2)}}_{k+2,\alpha}^p\big) \tfrac{\d\tau}{\tau} \\
    & \quad \sim_{k,\alpha,\vartheta, p} \int_0^\infty \inf_{u = u^{(1)} + u^{(2)}} \big(\tau^{-p(1-\vartheta)} \vertii{s \mapsto e^{-(\alpha-1) s} u^{(1)}(e^s)}_{W^{k-2,2}(\R)}^p + \tau^{p\vartheta} \vertii{s \mapsto e^{-\alpha s} u^{(2)}(e^s)}_{W^{k+2,2}(\R)}^p\big) \tfrac{\d\tau}{\tau} \\
    & \quad \gtrsim_{k,p} \int_0^\infty \inf_{u = u^{(1)} + u^{(2)}} \big(\tau^{-p(1-\vartheta)} \vertii{s \mapsto e^{-(\alpha-1) s} u^{(1)}(e^s)}_{W^{m_1,p}(\R)}^p + \tau^{p\vartheta} \vertii{s \mapsto e^{-\alpha s} u^{(2)}(e^s)}_{W^{m_2,p}(\R)}^p\big) \tfrac{\d\tau}{\tau},
\end{align*}
where $m_1 = k-\frac 5 2 + \frac 1 p$ and $m_2 = k+\frac 3 2 + \frac 1 p$. Using real interpolation for the Sobolev-Slobodeckij spaces and applying Fubini's theorem yields with the scaling $\tau \mapsto x \tau$
\begin{align*}
&\verti{u}_{k + 2-4\vartheta,\alpha - \vartheta,p}^p \\
    & \quad \gtrsim_{k, \alpha, \vartheta, p} \int_0^\infty \int_0^\infty \int_0^\infty \inf_{u = u^{(1)} + u^{(2)}} \Big(\tau^{-p(1-\vartheta)} x^{-p (\alpha-1)} \inf_{u^{(1)} = v^{(1)} + v^{(2)}} \big(\sigma^{-\frac p 2 + 1} \sum_{j = 0}^{k-3} \verti{D^j v^{(1)}}^p + \sigma^{\frac p 2 + 1} \sum_{j = 0}^{k-2} \verti{D^j v^{(2)}}^p\big) \\
    & \quad \phantom{\gtrsim \int_0^\infty \int_0^\infty \int_0^\infty} + \tau^{p\vartheta} x^{-p\alpha} \inf_{u^{(2)} = w^{(1)}+w^{(2)}} \big(\sigma^{-\frac p 2 + 1} \sum_{j = 0}^{k+1} \verti{D^j w^{(1)}}^p + \sigma^{\frac p 2 + 1} \sum_{j = 0}^{k+2} \verti{D^j w^{(2)}}^p\big)\Big) \tfrac{\d \sigma}{\sigma} \, \tfrac{\d \tau}{\tau} \, \tfrac{\d x}{x} \\
        & \quad = \int_0^\infty \int_0^\infty \int_0^\infty x^{-p(\alpha-\vartheta)} \inf_{u = u^{(1)} + u^{(2)}} \Big(\tau^{-p(1-\vartheta)} \inf_{u^{(1)} = v^{(1)} + v^{(2)}} \big(\sigma^{-\frac p 2 + 1} \sum_{j = 0}^{k-3} \verti{D^j v^{(1)}}^p + \sigma^{\frac p 2 + 1} \sum_{j = 0}^{k-2} \verti{D^j v^{(2)}}^p\big) \\
    & \quad \phantom{\gtrsim \int_0^\infty \int_0^\infty \int_0^\infty} + \tau^{p\vartheta} \inf_{u^{(2)} = w^{(1)}+w^{(2)}} \big(\sigma^{-\frac p 2+1} \sum_{j = 0}^{k+1} \verti{D^j w^{(1)}}^p + \sigma^{\frac p 2 + 1} \sum_{j = 0}^{k+2} \verti{D^j w^{(2)}}^p\big)\Big) \tfrac{\d \sigma}{\sigma} \, \tfrac{\d \tau}{\tau} \, \tfrac{\d x}{x} \\
            & \quad \sim_{k, \alpha, \vartheta, p} \int_0^\infty \inf_{u = u^{(1)} + u^{(2)}} \Big(\tau^{-p(1-\vartheta)} \vertii{s \mapsto e^{-(\alpha-\vartheta) s} u^{(1)}(e^s)}_{W^{m_1,p}(\R)}^p + \tau^{p\vartheta} \vertii{s \mapsto e^{-(\alpha-\vartheta) s} u^{(2)}(e^s)}_{W^{m_2,p}(\R)}^p\Big) \tfrac{\d \tau}{\tau} \\
            & \quad \sim_p \vertii{s \mapsto e^{-(\alpha-\vartheta) s} u(e^s)}_{W^{k + \frac 3 2 + \frac 1 p - 4 \vartheta,p}(\R)}^p.
\end{align*}
Thus we have proved \eqref{eq_char_interpol}. Estimate~\eqref{eq_lower_bound_interpol} follows by Sobolev embedding.

\medskip

In order to prove \eqref{lower_bound_l2eta} for $p > 2$, we notice that
\begin{align*}
\sum_{j = 0}^m \int_0^\infty \eta^2 x^{-2(\alpha-\vartheta)} (D^j u)^2 \tfrac{\d x}{x} & \le \sum_{j = 0}^m \Big(\int_0^\infty \eta^{\frac{2p}{p-2}} \, \tfrac{\d x}{x}\Big)^{\frac{p-2}{p}} \Big(\int_0^\infty x^{-p(\alpha-\vartheta)} \verti{D^j u}^p \, \tfrac{\d x}{x}\Big)^{\frac 2 p} \\
&\lesssim_{k,\kappa,\alpha,\vartheta,p} \vertii{s \mapsto e^{-(\alpha-\vartheta) s} u(e^s)}_{W^{m,p}(\R)} \stackrel{\eqref{eq_char_interpol}}{\lesssim_{k,\alpha,\vartheta,p}} \verti{u}_{k + 2-4\vartheta,\alpha - \vartheta,p},
\end{align*}
where H\"older's inequality was used.
\end{proof}
%

\subsection{Parabolic Maximal \texorpdfstring{$L^p$}{Lp}-Regularity}\label{sec_Linear_Maximal_Reg_est}
In what follows we show maximal-regularity estimates for \eqref{eq_full_cauchy}. For this, we use the estimates derived in \S\ref{sec_inhomogeneous_equation} and \S\ref{sec_homogeneous_equation}.
\begin{proposition}\label{proposition_estimate_almostMR}
Let $I := [0,T]$ where $T \in (0,\infty)$ or $I := [0,\infty)$, suppose that $M \in \N$, and assume that for $m \in \{1,\ldots,M\}$ we have $k_m \in \N$ with $k_m \ge 2$, $\alpha_m$ lies in the coercivity range of $\mathfrak p(D)$ (cf.~\eqref{coerc_range} of Lemma~\ref{lem_coercive_range} for a sufficient criterion), and $\mu_m \ge 0$. Furthermore, suppose that $\mu_m = 0$ for at least one $m \in \{1,\ldots,M\}$,
\[
u^{(0)}\in \bigcap_{\substack{m \in \{1,\ldots,M\} \\ \mu_m = 0}}  H_{k_m+2- \frac 4 p,\alpha_m - \frac 1 p,p},
\]
and $f \colon (0,\infty)^2 \to \R$ satisfies $(t \mapsto t^{\frac{\mu_m}{p}} f) \in L^p(0,\infty; H_{k_m-2,\alpha_m -
1})$ for all $m \in \{1,\ldots,m\}$. Then there exist unique locally integrable solutions $u = u^{(1)} + u^{(2)} \colon (0,\infty)^2 \to \R$ to \eqref{eq_full_cauchy} (where $u^{(1)}$ solves \eqref{eq_cauchy,u(0)neq0} and $u^{(2)}$ solves \eqref{eq_cauchy,f}) such that
$(t \mapsto t^{\frac{\mu_m}{p}} \partial_t u^{(\ell)}) \in L^p(0,\infty;H_{k_m-2,\alpha_m - 1})$ and $(t \mapsto t^{\frac{\mu_m}{p}} u^{(\ell)}) \in L^p(0,\infty;H_{k_m+2,\alpha_m})$ for all $m \in \{1,\ldots,m\}$ and $\ell \in \{1,2\}$. These solutions satisfy the estimates
\begin{align}\nonumber
& \sum_{\ell = 1}^2 \Big(\sup_{t \in I} t^{\mu_m} \verti{u^{(\ell)}}_{k_m+2-\frac 4 p,\alpha_m - \frac 1 p,p}^p + \int_I t^{\mu_m} \big(\verti{\partial_t u^{(\ell)}}^p_{k_m-2,\alpha_m-1}+\verti{u^{(\ell)}}^p_{k_m+2,\alp_m}\big) \, \d t\Big) \\
& \quad \lesssim_{(k_m,\alpha_m)_{m = 1}^M,p} \delta_{\mu_m,0} \verti{u^{(0)}}_{k_m+2-\frac 4 p,\alpha_m - \frac 1 p,p}^p + \mu_m \sum_{\ell = 1}^2 \int_I t^{\mu_m-1} \verti{u^{(\ell)}}_{k_m+2-\frac 4 p,\alpha_m - \frac 1 p,p}^p \, \d t \nonumber \\
&\phantom{\quad \lesssim_{(k_m,\alpha_m)_{m = 1}^M,p}} + \int_I t^{\mu_m} \verti{f}^p_{k_m-2,\alpha_m - 1} \, \d t
\label{eq_second_MR_estimate}
\end{align}
for all $m \in \{1,\ldots,m\}$, where $\delta_{\mu_m,0} = 1$ if $\mu_m = 0$ and $\delta_{\mu_m,0} = 0$ if $\mu_m > 0$.
\end{proposition}

\begin{proof}
By approximation we may assume $u^{(0)} \in C_\mathrm{c}^\infty((0,\infty))$ and $f \in C_\mathrm{c}^\infty((0,\infty)^2)$. As a first step, we fix $m \in \{1,\ldots,M\}$. From Proposition~\ref{prop_MR} we get a unique locally integrable solution $u^{(2)} \colon (0,\infty)^2 \to \R$ to \eqref{eq_cauchy,f} satisfying the following maximal-regularity estimate:
\begin{equation*}
    \|\partial_t u^{(2)}\|_{L^p(0,\infty;H_{k_m-2,\alpha_m-1})} + \|Au^{(2)}\|_{L^p(0,\infty;H_{k_m-2,\alpha_m-1})} \lesssim_{k_m,\alpha_m,p} \|f\|_{L^p(0,\infty;H_{k_m-2,\alpha_m-1})}.
\end{equation*}
Furthermore, \eqref{equiv-trace} of Lemma~\ref{lem-equiv-trace}, the trace method \cite[Proposition~1.2.10]{Lunardi}, and Lemma~\ref{lem_ineq_hardy} entail
\begin{align*}
&\sup_{t \ge 0} \| u^{(2)}\|_{D_{A_m}(1-\frac 1 p,p)} \\
& \quad \lesssim_p \ \ \sup_{t \ge 0} \|u^{(2)}\|_{(H_{k_m-2,\alpha_m-1},D(A_m))_{1-\frac 1 p,p}} \\
& \quad \lesssim_p \ \ \| \partial_t u^{(2)}\|_{L^p(0,\infty;H_{k_m-2,\alpha_m-1})} + \| u^{(2)}\|_{L^p(0,\infty;D(A_m))} \\
& \quad \lesssim_{k,\alpha,p} \ \ \| \partial_t u^{(2)}\|_{L^p(0,\infty;H_{k_m-2,\alpha_m-1})} + \| u^{(2)}\|_{L^p(0,\infty;H_{k_m-2,\alpha_m-1})} + \| A u^{(2)}\|_{L^p(0,\infty;H_{k_m-2,\alpha_m-1})},
\end{align*}
where
\[
A = A_m \colon H_{k_m-2,\alpha_m-1} \supset D(A_m) \stackrel{\eqref{a-op-domain}}{=} H_{k_m+2,\alpha_m} \cap H_{k_m-2,\alpha_m-1} \to H_{k_m-2,\alpha_m-1},
\]
where $D_{A_m}(1-\frac 1 p,p)$ is given by \eqref{def-dap} with $k = k_m$ and $\alpha = \alpha_m$. Hence, we have
\begin{subequations}\label{mr-intermediate}
\begin{align}\nonumber
& \sup_{t \ge 0} \| u^{(2)}\|_{D_{A_m}(1-\frac 1 p,p)} + \|\partial_t u^{(2)}\|_{L^p(0,\infty;H_{k_m-2,\alpha_m-1})} + \|A u^{(2)}\|_{L^p(0,\infty;H_{k_m-2,\alpha_m-1})} \\
& \quad \lesssim_{k_m,\alpha_m,p} \|f\|_{L^p(0,\infty;H_{k_m-2,\alpha_m-1})} + \| u^{(2)}\|_{L^p(0,\infty;H_{k_m-2,\alpha_m-1})}. \label{mr-intermediate-1}
\end{align}
The combination with \eqref{apriori-u1} of Proposition~\ref{prop_interpol_ineq_T=infty}, and \eqref{eq_ineq_hardy} of Lemma~\ref{lem_ineq_hardy} entails that  unique locally integrable solutions $u^{(1)},u^{(2)} \colon (0,\infty)^2 \to \R$ to \eqref{eq_cauchy,u(0)neq0} and \eqref{eq_cauchy,f}, respectively, exist such that $\partial_t u^{(\ell)} \in L^p(0,\infty;H_{k_m-2,\alpha_m - 1})$ and $u^{(\ell)} \in L^p(0,\infty;H_{k_m+2,\alpha_m})$, $\ell\in\{1,2\}$. This solution satisfies \eqref{mr-intermediate-1} and
\begin{align}
\vertii{\partial_t u^{(1)}}_{L^p(0,\infty;H_{k_m-2,\alpha_m-1})} + \vertii{u^{(1)}}^p_{L^p(0,\infty;H_{k_m+2,\alp_m})} \lesssim_{k_m,\alpha_m,p} \| u^{(0)}\|_{D_{A_m}(1-\frac{1}{p},p)}.
\end{align}
\end{subequations}
On the other hand, the same arguments also lead to unique locally integrable solutions $u^{(\ell)} \colon (0,\infty)^2 \to \R$ to \eqref{eq_cauchy,u(0)neq0} and \eqref{eq_cauchy,f}, respectively, such that 
\[
\partial_t u^{(\ell)} \in \bigcap_{m = 1}^M L^p(0,\infty;H_{k_m-2,\alpha_m - 1}) \quad \text{and} \quad u^{(\ell)} \in \bigcap_{m = 1}^ML^p(0,\infty;H_{k+2,\alpha_m}), \quad \text{for } \ell\in\{1,2\},
\]
so that by uniqueness for
every $m \in \{1,\ldots,M\}$ this solution is the same as the one for fixed $m \in \{1,\ldots,M\}$ and hence \eqref{mr-intermediate} is satisfied for every $m \in \{1,\ldots,M\}$.

\medskip

To lighten notation, we fix $m \in \{1,\ldots,M\}$ and write $A = A_m$, $k = k_m$, $\alpha = \alpha_m$, $\mu = \mu_m$, and $C = C_m$ from now on. Note that for any $\lambda > 0$ from \eqref{eq_cauchy,u(0)neq0} and \eqref{eq_cauchy,f}, respectively, we obtain the scaled problems
\begin{align*}
\partial_t (u^{(1)}(\lambda t,\lambda x)) - A (u^{(1)}(\lambda t,\lambda x)) &= 0 && \text{for } t,x > 0,\\
u(0,\lambda x) &= u^{(0)}(\lambda x) && \text{for } x > 0,
\end{align*}
and
\begin{align*}
\partial_t (u^{(2)}(\lambda t,\lambda x)) - A (u^{(2)}(\lambda t,\lambda x)) &= \lambda f(\lambda t,\lambda x) && \text{for } t,x > 0,\\
u^{(2)}(0,\lambda x) &= 0 && \text{for } x > 0.
\end{align*}
Hence, \eqref{mr-intermediate-1} entails after scaling
\begin{align*}
& \sup_{t \ge 0} \lambda^{1- p \alpha} \| u^{(2)} (t,\lambda\cdot) \|_{D_A(1-\frac 1 p,p)}^p + \|\partial_t u^{(2)}\|_{L^p(0,\infty;H_{k-2,\alpha-1})}^p + \|A u^{(2)}\|_{L^p(0,\infty;H_{k-2,\alpha-1})}^p \\
& \quad \lesssim_{k,\alpha,p} \|f\|_{L^p(0,\infty;H_{k-2,\alpha-1})}^p + \lambda^{-p} \| u^{(2)}\|_{L^p(0,\infty;H_{k-2,\alpha-1})}^p
\end{align*}
for any $\lambda > 0$. By \eqref{equiv-trace} of Lemma~\ref{lem-equiv-trace} it holds
\[
\lambda^{1-p \alpha} \| u^{(0)}(\lambda \cdot)\|_{D_A(1-\frac{1}{p},p)}^p \sim_{k,\alpha,p} \verti{u^{(0)}}_{k+2-\frac 4 p,\alpha - \frac 1 p,p}^p + \lambda^{1-p} \verti{u^{(0)}}_{k-2,\alpm}^p \to \verti{u^{(0)}}_{k+2-\frac 4 p,\alpha - \frac 1 p,p}^p
\]
as $\lambda \to \infty$. On the other hand, we obtain
\begin{align}
\lambda^{1 - p \alpha} \| u^{(0)}(\lambda \cdot)\|_{D_A(1-\frac{1}{p},p)}^p \ \ &\stackrel{\mathclap{\eqref{def-dap}}}{=} \ \ \lambda^{1-p\alpha} \verti{u^{(0)}(\lambda \cdot)}_{k-2,\alpha-1}^p + \lambda^{1 - p \alpha} \int_0^1 \verti{A e^{\tau A} (u^{(0)}(\lambda \cdot))}_{k-2,\alpha-1}^p \, \d\tau \nonumber \\
&\stackrel{\mathclap{\eqref{a-op-pd}}}{=} \ \ \lambda^{1-p\alpha} \verti{u^{(0)}}_{k-2,\alpm}^p + \lambda^{1-p\alpha} \int_0^1 \lambda^{p (\alpha-1)} \verti{\lambda A e^{\lambda \tau A} u^{(0)}}_{k-2,\alpha-1}^p \, \d\tau \nonumber \\
&= \ \ \lambda^{1-p\alpha} \verti{u^{(0)}}_{k-2,\alpm}^p + \int_0^\lambda \verti{A e^{\tau A} u^{(0)}}_{k-2,\alpha-1}^p \, \d\tau \nonumber \\
&\to \ \ \int_0^\infty \verti{A e^{\tau A} u^{(0)}}_{k-2,\alpha-1}^p \, \d\tau =: \verti{u^{(0)}}_{*}^p \nonumber \\
&\sim\verti{u^{(0)}}_{k+2-\frac 4 p,\alpha - \frac 1 p,p}^p < \infty \label{init-lambda-scale}
\end{align}
as $\lambda \to \infty$ by bounded monotone convergence. Hence, since
\[
\verti{u^{(1)}(t)}_*^p = \verti{e^{t A} u^{(0)}}_{*}^p \stackrel{\eqref{init-lambda-scale}}{=} \int_t^\infty \verti{A e^{\tau A} u^{(0)}}_{k-2,\alpha-1}^p \, \d\tau \stackrel{\eqref{init-lambda-scale}}{\le} \verti{u^{(0)}}_{*}^p
\]
for any $t \ge 0$ by definition, as $\lambda \to \infty$, \eqref{mr-intermediate} reduces to
\[
\sup_{t \ge 0} \big(\verti{u^{(1)}}_*^p+\verti{u^{(2)}}_*^p\big) + \int_0^{\infty}\big(\verti{\partial_t u^{(2)}}^p_{k-2,\alpm}+\verti{u^{(2)}}^p_{k+2,\alp}\big) \, \d t \le \verti{u^{(0)}}_{*}^p + C \int_0^{\infty} \verti{f}^p_{k-2,\alpm} \, \d t.
\]
By cutting off the right-hand side $f$, we infer that also
\begin{align*}
& \sup_{t \in [(j-1) \eps,j \eps]} \sum_{\ell = 1}^2 \verti{u^{(\ell)}}_{*}^p + \int_{(j-1)\eps}^{j\eps} \big(\verti{\partial_t u^{(2)}}^p_{k-2,\alpm}+\verti{u^{(2)}}^p_{k+2,\alp}\big) \, \d t \\
& \quad \le \verti{u((j-1) \eps)}_{*}^p + C \int_{(j-1)\eps}^{j\eps} \verti{f}^p_{k-2,\alpm} \, \d t
\end{align*}
must hold for any $\eps \in (0,\infty)$. Taking $T \in (0,\infty)$ and $\eps := \frac T N$, where $N \in \N$, we obtain after multiplying with $((j-1)\eps)^\mu$ and summation
\begin{align*}
& \sum_{\ell = 1}^2 \sum_{j = 1}^N \big(((j-1) \eps)^\mu \verti{u^{(\ell)}(j\eps)}_{*}^p - ((j-1) \eps)^\mu \verti{u^{(\ell)}((j-1)\eps)}_{*}^p\big) \\
& + \sum_{j = 1}^N \int_{(j-1)\eps}^{j\eps} ((j-1)\eps)^\mu \big(\verti{\partial_t u^{(2)}}^p_{k-2,\alpm}+\verti{u^{(2)}}^p_{k+2,\alp}\big) \, \d t \\
&\le C \sum_{j = 1}^N \int_{(j-1)\eps}^{j\eps} ((j-1)\eps)^\mu \verti{f}^p_{k-2,\alpm} \, \d t.
\end{align*}
By dominated convergence, it follows
\begin{align*}
\sum_{j = 1}^N \int_{(j-1)\eps}^{j\eps} ((j-1)\eps)^\mu \big(\verti{\partial_t u^{(2)}}^p_{k-2,\alpm}+\verti{u^{(2)}}^p_{k+2,\alp}\big) \d t &\to \int_0^T t^\mu \big(\verti{\partial_t u^{(2)}}^p_{k-2,\alpm}+\verti{u^{(2)}}^p_{k+2,\alp}\big) \, \d t, \\
\sum_{j = 1}^N \int_{(j-1)\eps}^{j\eps} ((j-1)\eps)^\mu \verti{f}^p_{k-2,\alpm} \, \d t &\to \int_0^T t^\mu \verti{f}^p_{k-2,\alpm} \, \d t
\end{align*}
as $N \to \infty$, as well as
\begin{align*}
& \sum_{\ell = 1}^2 \sum_{j = 1}^N \Big(((j-1) \eps)^\mu \verti{u^{(\ell)}(j\eps)}_{*}^p - ((j-1) \eps)^\mu \verti{u^{(\ell)}((j-1)\eps)}_{*}^p\Big) \\
& \quad = ((N-1) \eps)^\mu \sum_{\ell = 1}^2 \verti{u^{(\ell)}(T)}_*^p - \delta_{\mu,0} \verti{u^{(0)}}_*^p + \sum_{\ell = 1}^2 \sum_{j = 2}^{N-1} \big(((j-1)\eps)^\mu - (j\eps)^\mu\big) \verti{u^{(\ell)}(j\eps)}_*^p \\
& \quad \to T^\mu \sum_{\ell = 1}^2 \verti{u^{(\ell)}(T)}_*^p - \delta_{\mu,0} \verti{u^{(0)}}_*^p - \mu \sum_{\ell = 1}^2 \int_0^T t^{\mu-1} \verti{u^{(\ell)}(t)}_*^p \, \d t \quad \text{as } N \to \infty,
\end{align*}
by uniform convergence. Hence, after taking $T\to\infty$, we arrive at
\begin{align}\nonumber
& \sup_{t \ge 0} t^\mu \sum_{\ell = 1}^2 \verti{u^{(\ell)}}_*^p + \frac{1}{C} \int_0^\infty t^\mu \big(\verti{\partial_t u^{(2)}}^p_{k-2,\alpm}+\verti{u^{(2)}}^p_{k+2,\alp}\big) \, \d t \\
& \quad \le \delta_{\mu,0} \verti{u^{(0)}}_*^p + \mu \int_0^\infty t^{\mu-1} \sum_{\ell = 1}^2 \verti{u^{(\ell)}}_*^p \, \d t + C \int_0^\infty t^\mu \verti{f}^p_{k-2,\alpm} \, \d t. \label{est-u2}
\end{align}
Furthermore, 
\begin{align*}
\int_0^\infty t^\mu \big(\verti{\partial_t u^{(1)}}_{k-2,\alp-1}^p + \verti{u^{(1)}}_{k+2,\alp}^p\big) \, \d t \stackrel{\eqref{eq_ineq_hardy}, \eqref{hom_eq}}{\lesssim_{k,\alpha}} \int_0^\infty t^\mu \verti{A e^{t A} u^{(0)}}_{k-2,\alp-1}^p \, \d t,
\end{align*}
where Lemma~\ref{lem_ineq_hardy} was used. For $\mu = 0$ the right-hand side simply equals $\verti{u^{(0)}}_*^p$ (cf.~\eqref{init-lambda-scale}). For $\mu > 0$ we obtain through integration by parts for $T \in (0,\infty)$,
\begin{align*}
\int_0^T t^\mu \verti{A e^{t A} u^{(0)}}_{k-2,\alp-1}^p \, \d t =& - t^\mu \int_t^T \verti{A e^{\tau A} u^{(0)}}_{k-2,\alp-1}^p \d\tau \Big|_{t = 0}^T \\
& + \mu \int_0^T t^{\mu-1} \int_t^T \verti{A e^{\tau A} u^{(0)}}_{k-2,\alp-1}^p \d\tau \, \d t \\
= & \mu \int_0^T t^{\mu-1} \int_0^{T-t} \verti{A e^{\tau A} u^{(1)}}_{k-2,\alp-1}^p \d\tau \, \d t,
\end{align*}
so that in the limit $T \to \infty$ we get
\[
\int_0^\infty t^\mu \verti{A e^{t A} u^{(0)}}_{k-2,\alp-1}^p \, \d t = \mu \int_0^\infty t^{\mu-1} \int_0^\infty \verti{A e^{\tau A} u^{(1)}}_{k-2,\alp-1}^p \d\tau \, \d t = \mu \int_0^\infty t^{\mu-1} \verti{u^{(1)}}_*^p \, \d t.
\]
Hence, estimate~\eqref{est-u2} upgrades to
\begin{align*}
& \sum_{\ell = 1}^2 \Big(\sup_{t \ge 0} t^\mu \verti{u^{(\ell)}(t)}_*^p + \int_0^\infty t^\mu \big(\verti{\partial_t u^{(\ell)}}^p_{k-2,\alpm}+\verti{u^{(\ell)}}^p_{k+2,\alp}\big) \, \d t\Big) \\
& \quad \lesssim_{k,\alpha,p} \delta_{\mu,0} \verti{u^{(0)}}_*^p + \mu \sum_{\ell = 1}^2 \int_0^\infty t^{\mu-1} \verti{u^{(\ell)}}_*^p \, \d t + \int_0^\infty t^\mu \verti{f}^p_{k-2,\alpm} \, \d t. \label{est-u-u2}
\end{align*}
With help of \eqref{init-lambda-scale}, we end up with \eqref{eq_second_MR_estimate}.

\medskip

The uniqueness statement of the proposition follows from uniqueness under the condition that $\partial_t u^{(\ell)} \in L^p(0,\infty;H_{k_m-2,\alpha_m - 1})$ and $u^{(\ell)} \in L^p(0,\infty;H_{k_m+2,\alpha_m})$, $\ell \in \{1,2\}$, for an $m \in \{1,\ldots,M\}$ such that $\mu_m = 0$.
\end{proof}
%

\subsection{Higher Regularity}\label{sec_higher_regularity}
Here, we derive estimates that are suitable for treating the nonlinear equation in \S\ref{sec-nonlinear}. The key ingredient is the elliptic regularity of the operator $A$. Before we start with the rigorous setting, let us first give some heuristic ideas. Suppose that we have a sufficiently regular solution $u$ of \eqref{eq_full_cauchy}. Formally applying the operator $A$ to the Cauchy problem \eqref{eq_full_cauchy} gives
\begin{subequations}\label{eq_full_cauchy_extra_A}
\begin{align}
    \partial_t (Au)-A(Au) &= Af && \text{for } t,x > 0,\\
    (Au)|_{t=0} &=Au^{(0)} && \text{for } x > 0.
\end{align}
\end{subequations}
Then $Au$ solves \eqref{eq_full_cauchy_extra_A} and we can find $u$ from this solution by inverting $A$. In the proof we will argue in the opposite direction by finding a solution to
\begin{subequations}\label{eq_full_cauchy_v}
\begin{alignat}{2}
    \partial_t v-Av &= g\\
    v|_{t=0} &=v^{(0)},
\end{alignat}
\end{subequations}
where $v^{(0)}:=Au^{(0)}$ and $g:=Af$ are given, and $A$ will be inverted to find $u$ satisfying $A u = v$. The following inversion lemma is essential (which is analogous to \cite[Lemmata~7.2 and 7.4]{GGKO} or \cite[Proposition~3.1]{G2016} with a simplified proof):
\begin{lemma}[Elliptic Regularity]\label{lem-invert-A}
Suppose that $\ell_1,\ell_2,\ell_3 \ge 0$, let $\gamma_1 < \gamma_2 < 0$ be the two negative roots of $\mathfrak p(\zeta)$ given by \eqref{eq_zeros_12}, and assume that $\gamma_2 < \alpha_1 < 0 < \alpha_2 < \beta < \alpha_3$.
\begin{enumerate}
\item\label{invert-A-1} Suppose that $v \in H_{\ell_1,\alpha_1-1} \cap H_{\ell_2,\alpha_2-1}$. Then
\begin{subequations}\label{def-u-u0-ub-invert}
\begin{align}
u(x) &:= (Bv)(x) := - x^{\gamma_1} \int_0^x x_1^{\gamma_2-\gamma_1} \int_0^{x_1} x_2^{-\gamma_2} \int_{x_2}^\infty x_3^\beta \int_{x_3}^\infty x_4^{1-\beta} v(x_4) \, \tfrac{\d x_4}{x_4} \, \tfrac{\d x_3}{x_3} \, \tfrac{\d x_2}{x_2} \, \tfrac{\d x_1}{x_1}, \label{def-u-invert} \\
u_0 &:= - \tfrac{1}{\gamma_1 \gamma_2} \int_0^\infty x_1^\beta \int_{x_1}^\infty x_2^{1-\beta} v(x_2) \, \tfrac{\d x_2}{x_2} \, \tfrac{\d x_1}{x_1}, \label{def-u0-invert}
\end{align}
\end{subequations}
are well-defined with
\begin{equation}
u(x) \to u_0 \quad \text{as } x \searrow 0, \label{limit-u-invert-0} \end{equation}
and it holds
\begin{equation}\label{eq-au-v}
AB v = A u = A (u-u_0) = v,
\end{equation}
and
\begin{subequations}\label{u-inv-v-est}
\begin{align}
\verti{u}_{\ell_1+4,\alpha_1} &\sim_{\ell_1,\alpha_1} \verti{v}_{\ell_1,\alpha_1-1}, \label{u-inv-v-1} \\
\verti{u - u_0}_{\ell_2+4,\alpha_2} &\sim_{\ell_2,\alpha_2} \verti{v}_{\ell_2,\alpha_2-1}. \label{u-inv-v-2}
\end{align}
\end{subequations}

\item\label{invert-A-2} In the situation of \eqref{invert-A-1} additionally suppose that $v \in H_{\ell_3,\alpha_3-1}$. Then
\begin{equation}\label{def-ub-invert}
u_\beta := \tfrac{1}{\beta (\beta - \gamma_1) (\beta - \gamma_2)} \int_0^\infty x_1^{1-\beta} v(x_1) \, \tfrac{\d x_1}{x_1}
\end{equation}
is well-defined,
\begin{equation}
x^{-\beta} (u(x)-u_0) \to u_\beta \quad \text{as } x \searrow 0, \label{limit-u-invert-beta}
\end{equation}
and it holds
\begin{equation}\label{eq-au-v-beta}
A (u - u_0 - u_\beta x^\beta) = v
\end{equation}
and
\begin{equation}\label{u-inv-v-3}
\verti{u - u_0 - u_\beta x^\beta}_{\ell_3+4,\alpha_3} \sim_{\ell_3,\alpha_3} \verti{v}_{\ell_3,\alpha_3-1}.
\end{equation}

\item\label{invert-A-3} If $u \colon (0,\infty) \to \R$ is locally integrable and $u_0 \in \R$ is a real number such that $\verti{u}_{\ell_1+4,\alpha_1}$ and $\verti{u - u_0}_{\ell_2+4,\alpha_2}$ are finite, then
\begin{equation}\label{left-inv-a}
B A u = B A (u-u_0) = B A (u - u_0 - c x^\beta) = u,
\end{equation}
where $c \in \R$ is any real number.
\end{enumerate}
\end{lemma}
\begin{proof}
\textbf{Proof of \eqref{invert-A-1}.} We detail why $u$ and $u_0$, defined by \eqref{def-u-u0-ub-invert}, are well-defined. Indeed, it holds $x_4^{1-\beta} \verti{v(x_4)} = o(x_4^{\alpha_1-\beta})$ almost everywhere as $x_4 \to \infty$, so that because of $\alpha_1 < \beta$ the first integral $\int_{x_3}^\infty x_4^{1-\beta} \verti{v(x_4)} \, \tfrac{\d x_4}{x_4}$ is well-defined and $\int_{x_3}^\infty x_4^{1-\beta} \verti{v(x_4)} \, \tfrac{\d x_4}{x_4} = o(x_3^{\alpha_1-\beta})$ as $x_3 \to \infty$. Furthermore, $x_4^{1-\beta} \verti{v(x_4)} = o(x_4^{\alpha_2-\beta})$ almost everywhere as $x_4 \searrow 0$ and thus $\int_{x_3}^\infty x_4^{1-\beta} \verti{v(x_4)} \, \tfrac{\d x_4}{x_4} = o(x_3^{\alpha_2-\beta})$ as $x_3 \searrow 0$. Hence, we have
\[
\int_{x_2}^\infty x_3^\beta \int_{x_3}^\infty x_4^{1-\beta} \verti{v(x_4)} \, \tfrac{\d x_4}{x_4} \, \tfrac{\d x_3}{x_3} = o(x_2^{\alpha_1}) \quad \text{as } x_2 \to \infty
\]
and
\[
\int_0^\infty x_3^\beta \int_{x_3}^\infty x_4^{1-\beta} \verti{v(x_4)} \, \tfrac{\d x_4}{x_4} \, \tfrac{\d x_3}{x_3} \le \int_1^\infty o(x_3^{\alpha_1}) \, \tfrac{\d x_3}{x_3} + \int_0^1 o(x_3^{\alpha_2}) \, \tfrac{\d x_3}{x_3} < \infty,
\]
so that $u_0$ given by \eqref{def-u0-invert} is well-defined. Next, observe that
\begin{align*}
\int_0^{x_1} x_2^{-\gamma_2} \int_{x_2}^\infty x_3^\beta \int_{x_3}^\infty x_4^{1-\beta} \verti{v(x_4)} \, \tfrac{\d x_4}{x_4} \, \tfrac{\d x_3}{x_3} \, \tfrac{\d x_2}{x_2} &\le \int_0^{x_1} x_2^{-\gamma_2} \int_0^\infty x_3^\beta \int_{x_3}^\infty x_4^{1-\beta} \verti{v(x_4)} \, \tfrac{\d x_4}{x_4} \, \tfrac{\d x_3}{x_3} \, \tfrac{\d x_2}{x_2} \\
&\phantom{\le} \begin{cases}
= O(x_1^{-\gamma_2}) & \text{as } x_1 \searrow 0, \\
< \infty & \text{for all } x_1 > 0.
\end{cases}
\end{align*}
Thus,
\[
\int_0^x x_1^{\gamma_2-\gamma_1} \int_0^{x_1} x_2^{-\gamma_2} \int_{x_2}^\infty x_3^\beta \int_{x_3}^\infty x_4^{1-\beta} \verti{v(x_4)} \, \tfrac{\d x_4}{x_4} \, \tfrac{\d x_3}{x_3} \, \tfrac{\d x_2}{x_2} \, \tfrac{\d x_1}{x_1} \begin{cases} = O(x^{-\gamma_1}) & \text{as } x \searrow 0, \\
< \infty & \text{for all } x > 0.
\end{cases}
\]
Hence, $u$ given by \eqref{def-u-invert} is well-defined and finite. Applying the operator $A$ to $u$ and using \eqref{eq_def_p(D)}, \eqref{eq_zeros_12}, and \eqref{a-op-pd}, we infer that $A x^0 = 0$ and that indeed \eqref{eq-au-v} is satisfied.

\medskip

For investigating the limiting behavior, observe that
\begin{align*}
u(x) \ \ &\stackrel{\mathclap{\eqref{def-u-invert}}}{=} \ \ - \int_0^1 r_1^{\gamma_2-\gamma_1} \int_0^{r_1} r_2^{-\gamma_2} \int_{x r_2}^\infty x_3^\beta \int_{x_3}^\infty x_4^{1-\beta} v(x_4) \, \tfrac{\d x_4}{x_4} \, \tfrac{\d x_3}{x_3} \, \tfrac{\d r_2}{r_2} \, \tfrac{\d r_1}{r_1} \\
&\to - \Big(\int_0^1 r_1^{\gamma_2-\gamma_1} \int_0^{r_1} r_2^{-\gamma_2} \, \tfrac{\d r_2}{r_2} \, \tfrac{\d r_1}{r_1} \Big) \Big(\int_0^\infty x_3^\beta \int_{x_3}^\infty x_4^{1-\beta} v(x_4) \, \tfrac{\d x_4}{x_4} \, \tfrac{\d x_3}{x_3}\Big) \stackrel{\eqref{def-u0-invert}}{=} u_0 \quad \text{as } x \searrow 0,
\end{align*}
which shows \eqref{limit-u-invert-0}.

\medskip

In order to prove \eqref{u-inv-v-est} by interpolation (cf.~\eqref{sobolev-frac}), we can restrict ourselves to $\ell_j \in \N_0$. Further note that estimating $v$ against $u$ or $u-u_0$, respectively, is immediate by \eqref{eq-au-v}. By a standard interpolation argument it follows that
\begin{align*}
\verti{D^{\ell_1+4} u}_{\alpha_1} - C_1 \verti{u}_{\alpha_1} &\le 2 \verti{v}_{\ell_1,\alpha_1-1}, \\
\verti{D^{\ell_2+4} (u-u_0)}_{\alpha_2} - C_2 \verti{u-u_0}_{\alpha_2} &\le 2 \verti{v}_{\ell_2,\alpha_2-1},
\end{align*}
where $C_j$ only depends on $\ell_j$ and $\alpha_j$. Hence, it suffices to prove
\begin{subequations}\label{hardy-4th}
\begin{align}
\verti{u}_{\alpha_1} &\lesssim_{\ell_1,\alpha_1} \verti{v}_{\alpha_1-1}, \label{hardy-4th-u} \\
\verti{u-u_0}_{\alpha_2} &\lesssim_{\ell_2,\alpha_2} \verti{v}_{\alpha_2-1}. \label{hardy-4th-u-u0}
\end{align}
\end{subequations}
Therefore, observe because of $\verti{v}_{\alpha_1-1} < \infty$,
\begin{align*}
u \ \ &\stackrel{\mathclap{\eqref{limit-u-invert-0}}}{\to} \ \ u_0 && \text{as } x \searrow 0, \\
(D-\gamma_1) u \ \ &\stackrel{\mathclap{\eqref{def-u-invert}}}{=} \ \ - \int_0^1 r_1^{-\gamma_2} \int_{x r_1}^\infty x_2^\beta \int_{x_2}^\infty x_2^{1-\beta} v(x_3) \, \tfrac{\d x_3}{x_3} \, \tfrac{\d x_2}{x_2} \, \tfrac{\d r_1}{r_1} \\
&\to \ \ \tfrac{1}{\gamma_2} \int_0^\infty x_1^\beta \int_{x_1}^\infty x_2^{1-\beta} v(x_2) \, \tfrac{\d x_2}{x_2} \, \tfrac{\d x_1}{x_1} \stackrel{\eqref{def-u0-invert}}{=} - \gamma_1 u_0 && \text{as } x \searrow 0, \\
(D-\gamma_2) (D-\gamma_1) u  \ \ &\stackrel{\mathclap{\eqref{def-u-invert}}}{=} \ \ - \int_x^\infty x_1^\beta \int_{x_1}^\infty x_2^{1-\beta} v(x_2) \, \tfrac{\d x_2}{x_2} \, \tfrac{\d x_1}{x_1} \to 0 && \text{as } x \to \infty, \\
D (D-\gamma_2) (D-\gamma_1) u  \ \ &\stackrel{\mathclap{\eqref{def-u-invert}}}{=} \ \   x^\beta \int_x^\infty x_1^{1-\beta} v(x_1) \, \tfrac{\d x_1}{x_1} = o(x^{\alpha_1}) && \text{as } x \to \infty,
\end{align*}
where we have used that $u_0$ given by \eqref{def-u0-invert} is well-defined in the second-but-last limit. By using Hardy's inequality (see for instance \cite[Lemma~A.1]{GiacomelliKnuepferOtto2008}), we thus obtain
\[
\verti{u}_{\alpha_1} = \verti{x^{\gamma_1-\alpha_1} (x^{-\gamma_1} u)}_{0} \lesssim_{\alpha_1} \verti{x^{\gamma_1-\alpha_1} D (x^{-\gamma_1} u)}_{0} = \verti{(D-\gamma_1) u}_{\alpha_1},
\]
where we have used that $\gamma_1 < \alpha_1 < 0$ and $x^{-\gamma_1} u \to 0$ as $x \searrow 0$. Proceeding analogously, we infer
\[
\verti{(D-\gamma_1) u}_{\alpha_1} = \verti{x^{\gamma_2-\alpha_1} (x^{-\gamma_2}(D-\gamma_1) u)}_0 \lesssim_{\alpha_1} \verti{x^{\gamma_2-\alpha_1} D (x^{-\gamma_2}(D-\gamma_1) u)}_0 = \verti{(D-\gamma_2) (D-\gamma_1) u}_{\alpha_1},
\]
where we have used $\gamma_2 < \alpha_1 < 0$ and $x^{-\gamma_2} (D-\gamma_1) u \to 0$ as $x \searrow 0$. Furthermore,
\begin{align*}
\verti{(D-\gamma_2) (D-\gamma_1) u}_{\alpha_1} &= \verti{x^{-\alpha_1} (D-\gamma_2) (D-\gamma_1) u}_0 \\
&\lesssim_{\alpha_1} \verti{x^{-\alpha_1} D (D-\gamma_2) (D-\gamma_1) u}_0 = \verti{D (D-\gamma_2) (D-\gamma_1) u}_{\alpha_1}
\end{align*}
where we have used $\alpha_1 < 0$ and $(D-\gamma_2) (D-\gamma_1) u \to 0$ as $x \to \infty$. Additionally,
\begin{align*}
\verti{D (D-\gamma_2) (D-\gamma_1) u}_{\alpha_1} &= \verti{x^{\beta-\alpha_1} (x^{-\beta} D (D-\gamma_2) (D-\gamma_1) u)}_0 \\
&\lesssim_{\alpha_1} \verti{x^{\beta-\alpha_1} D (x^{-\beta} D (D-\gamma_2) (D-\gamma_1) u)}_0 \\
&= \verti{(D-\beta) D (D-\gamma_2) (D-\gamma_1) u}_{\alpha_1}
\end{align*}
where we have used $\alpha_1 < 0$ and $x^{-\beta} D (D-\gamma_2) (D-\gamma_1) u \to 0$ as $x \to \infty$. Thus,
\begin{align*}
\verti{u}_{\alpha_1} \lesssim_{\alpha_1} \verti{(D-\beta) D (D-\gamma_2) (D-\gamma_1) u}_{\alpha_1} \stackrel{\eqref{eq_def_p(D)}, \eqref{eq_zeros_12}}{=} \verti{\mathfrak p(D) u}_{\alpha_1} \stackrel{\eqref{a-op-pd}}{=} \verti{A u}_{\alpha_1-1} \stackrel{\eqref{eq-au-v}}{=} \verti{v}_{\alpha_1-1},
\end{align*}
which proves \eqref{hardy-4th-u}.

\medskip

From \eqref{def-u-u0-ub-invert} we infer because of $\verti{v}_{\alpha_1-1} < \infty$ and $\verti{v}_{\alpha_2-1} < \infty$,
\begin{equation}\label{rep-u-u0}
u - u_0 = x^{\gamma_1} \int_0^x x_1^{\gamma_2-\gamma_1} \int_0^{x_1} x_2^{-\gamma_2} \int_0^{x_2} x_3^\beta \int_{x_3}^\infty x_4^{1-\beta} v(x_4) \, \tfrac{\d x_4}{x_4} \, \tfrac{\d x_3}{x_3} \, \tfrac{\d x_2}{x_2} \, \tfrac{\d x_1}{x_1},
\end{equation}
so that
\begin{align*}
\verti{u - u_0} &\le \int_0^1 r_1^{\gamma_2-\gamma_1} \int_0^{r_1} r_2^{-\gamma_2} \int_0^{x r_2} x_3^\beta \int_{x_3}^\infty x_4^{1-\beta} \verti{v(x_4)} \, \tfrac{\d x_4}{x_4} \, \tfrac{\d x_3}{x_3} \, \tfrac{\d r_2}{r_2} \, \tfrac{\d r_1}{r_1} \\
&= \tfrac{1}{\gamma_1 \gamma_2} \int_0^x \int_{x_1}^\infty x_2^{1-\beta} \verti{v(x_2)} \, \tfrac{\d x_2}{x_2} \, \tfrac{\d x_1}{x_1} = o(x^{\alpha_2}) &&\hspace{-.7in} \text{as } x \searrow 0, \\
\verti{(D-\gamma_1) (u-u_0)} &\le \int_0^1 r_1^{-\gamma_2} \int_{0}^{x r_1} x_2^\beta \int_{x_2}^\infty x_3^{1-\beta} \verti{v(x_3)} \, \tfrac{\d x_3}{x_3} \, \tfrac{\d x_2}{x_2} \, \tfrac{\d r_1}{r_1} \\
&= - \tfrac{1}{\gamma_2} \int_0^x \int_{x_1}^\infty x_2^{1-\beta} \verti{v(x_2)} \, \tfrac{\d x_2}{x_2} \, \tfrac{\d x_1}{x_1} = o(x^{\alpha_2}) &&\hspace{-.7in} \text{as } x \searrow 0, \\
\verti{(D-\gamma_2) (D-\gamma_1) (u-u_0)} &\le \int_0^x x_1^\beta \int_{x_1}^\infty x_2^{1-\beta} \verti{v(x_2)} \, \tfrac{\d x_2}{x_2} \, \tfrac{\d x_1}{x_1} = o(x^{\alpha_2}) &&\hspace{-.7in} \text{as } x \searrow 0, \\
\verti{D (D-\gamma_2) (D-\gamma_1) (u-u_0)}  &\le x^\beta \int_x^\infty x_1^{1-\beta} \verti{v(x_1)} \, \tfrac{\d x_1}{x_1} = o(x^{\alpha_1}) &&\hspace{-.7in} \text{as } x \to \infty.
\end{align*}
By iteratively applying Hardy's inequality as before, taking the limiting behavior as $x \searrow 0$ or $x \to \infty$, respectively, into account, we thus obtain
\begin{align*}
\verti{u-u_0}_{\alpha_2} &\lesssim_{\alpha_2} \verti{(D-\gamma_1) (u-u_0)}_{\alpha_2} \lesssim_{\alpha_2} \verti{(D-\gamma_2) (D-\gamma_1) (u-u_0)}_{\alpha_2} \\
&\lesssim_{\alpha_2} \verti{D (D-\gamma_2) (D-\gamma_1) (u-u_0)}_{\alpha_2} \lesssim_{\alpha_2} \verti{(D-\beta) D (D-\gamma_2) (D-\gamma_1) (u-u_0)}_{\alpha_2} \\
&\stackrel{\mathclap{\eqref{eq_def_p(D)}, \eqref{eq_zeros_12}}}{=} \ \ \quad \verti{\mathfrak p(D) u}_{\alpha_2} \stackrel{\eqref{a-op-pd}}{=} \verti{A u}_{\alpha_2-1} \stackrel{\eqref{eq-au-v}}{=} \verti{v}_{\alpha_2-1},
\end{align*}
which proves \eqref{hardy-4th-u-u0}.

\medskip

\textbf{Proof of \eqref{invert-A-2}.} We have $x_1^{1-\beta} \verti{v(x_1)} = o(x_1^{\alpha_1-\beta})$ almost everywhere as $x_1 \to \infty$ and $x_1^{1-\beta} \verti{v(x_1)} = o(x_1^{\alpha_3-\beta})$ almost everywhere as $x_1 \searrow 0$, so that
\[
\int_0^\infty x_1^{1-\beta} \verti{v(x_1)} \, \tfrac{\d x_1}{x_1} \le \int_0^1 o(x_1^{\alpha_3-\beta}) \, \tfrac{\d x_1}{x_1} + \int_1^\infty o(x_1^{\alpha_1-\beta}) \, \tfrac{\d x_1}{x_1} < \infty,
\]
and thus $u_\beta$ given by \eqref{def-ub-invert} is well-defined. The limit \eqref{limit-u-invert-beta} follows from
\begin{align*}
x^{-\beta} (u(x)-u_0) \ \ &\stackrel{\mathclap{\eqref{rep-u-u0}}}{=} \ \ \int_0^1 r_1^{\gamma_2-\gamma_1} \int_0^{r_1} r_2^{-\gamma_2} \int_0^{r_2} r_3^\beta \int_{x r_3}^\infty x_4^{1-\beta} v(x_4) \, \tfrac{\d x_4}{x_4} \, \tfrac{\d r_3}{r_3} \, \tfrac{\d r_2}{r_2} \, \tfrac{\d r_1}{r_1} \\
&\to \Big(\int_0^1 r_1^{\gamma_2-\gamma_1} \int_0^{r_1} r_2^{-\gamma_2} \int_0^{r_2} r_3^\beta \, \tfrac{\d r_3}{r_3} \, \tfrac{\d r_2}{r_2} \, \tfrac{\d r_1}{r_1}\Big) \Big(\int_0^\infty x_4^{1-\beta} v(x_4) \, \tfrac{\d x_4}{x_4}\Big) \stackrel{\eqref{def-ub-invert}}{=} u_\beta
\end{align*}
as $x \searrow 0$. Because of \eqref{eq-au-v} and $A x^\beta = 0$ by \eqref{eq_def_p(D)}, \eqref{eq_zeros_12} and \eqref{def-a-op}, \eqref{eq-au-v-beta} follows.

\medskip

In order to prove the non-trivial estimate in \eqref{u-inv-v-3}, again by interpolation (cf.~\eqref{sobolev-frac}) we assume without loss of generality $\ell_3 \in \N_0$. In view of the interpolation inequality
\[
\verti{D^{\ell_3+4} (u-u_0-u_\beta x^\beta)}_{\alpha_3} - C_3 \verti{u-u_0-u_\beta x^\beta}_{\alpha_3} \le 2 \verti{v}_{\ell_3,\alpha_3-1}
\]
where $C_3 < \infty$ only depends on $\ell_3$ and $\alpha_3$, we only need to prove
\begin{equation}\label{hardy-4th-u-u0-ubeta}
\verti{u-u_0-u_\beta x^\beta}_{\alpha_3} \lesssim_{\ell_3,\alpha_3} \verti{v}_{\alpha_3-1}.
\end{equation}
From \eqref{def-ub-invert} and \eqref{rep-u-u0} we obtain
\[
u - u_0 - u_\beta x^\beta = - x^{\gamma_1} \int_0^x x_1^{\gamma_2-\gamma_1} \int_0^{x_1} x_2^{-\gamma_2} \int_0^{x_2} x_3^\beta \int_0^{x_3} x_4^{1-\beta} v(x_4) \, \tfrac{\d x_4}{x_4} \, \tfrac{\d x_3}{x_3} \, \tfrac{\d x_2}{x_2} \, \tfrac{\d x_1}{x_1},
\]
so that because of $\verti{v}_{\alpha_3-1} < \infty$ we have
\begin{align*}
u - u_0 - u_\beta x^\beta &= o(x^{\alpha_3}) &&\hspace{-.2in} \text{as } x \searrow 0, \\
(D-\gamma_1) (u-u_0-u_\beta x^\beta) &= - x^{\gamma_2} \int_0^x x_1^{-\gamma_2} \int_0^{x_1} x_2^\beta \int_0^{x_2} x_3^{1-\beta} v(x_3) \, \tfrac{\d x_3}{x_3} \, \tfrac{\d x_2}{x_2} \, \tfrac{\d x_1}{x_1} \\
&= o(x^{\alpha_3}) &&\hspace{-.2in} \text{as } x \searrow 0, \\
(D-\gamma_2) (D-\gamma_1) (u-u_0-u_\beta x^\beta) &= - \int_0^x x_1^\beta \int_0^{x_1} x_2^{1-\beta} v(x_2) \, \tfrac{\d x_2}{x_2} \, \tfrac{\d x_1}{x_1} = o(x^{\alpha_3}) &&\hspace{-.2in} \text{as } x \searrow 0, \\
D (D-\gamma_2) (D-\gamma_1) (u-u_0-u_\beta x^\beta) &= - x^\beta \int_0^x x_1^{1-\beta} v(x_1) \, \tfrac{\d x_1}{x_1} = o(x^{\alpha_3}) &&\hspace{-.2in} \text{as } x \searrow 0.
\end{align*}
An iterative application of Hardy's inequality as in the proof of part~\eqref{invert-A-1}, employing the limiting behavior as $x \searrow 0$, yields
\begin{align*}
\verti{u-u_0-u_\beta x^\beta}_{\alpha_3} &\lesssim_{\alpha_3} \verti{(D-\gamma_1) (u-u_0-u_\beta x^\beta)}_{\alpha_3} \lesssim_{\alpha_3} \verti{(D-\gamma_2) (D-\gamma_1) (u-u_0-u_\beta x^\beta)}_{\alpha_3} \\
&\lesssim_{\alpha_3} \verti{D (D-\gamma_2) (D-\gamma_1) (u-u_0-u_\beta x^\beta)}_{\alpha_3} \\
&\lesssim_{\alpha_3} \verti{(D-\beta) D (D-\gamma_2) (D-\gamma_1) (u-u_0-u_\beta x^\beta)}_{\alpha_3} \\
&\stackrel{\mathclap{\eqref{eq_def_p(D)}, \eqref{eq_zeros_12}}}{=} \ \ \quad \verti{\mathfrak p(D) u}_{\alpha_3} \stackrel{\eqref{a-op-pd}}{=} \verti{A u}_{\alpha_3-1} \stackrel{\eqref{eq-au-v}}{=} \verti{v}_{\alpha_3-1},
\end{align*}
which proves \eqref{hardy-4th-u-u0-ubeta}.

\medskip

\textbf{Proof of \eqref{invert-A-3}.} In order to show \eqref{left-inv-a}, observe that
\[
(B A u)(x) = x^{\gamma_1} \int_0^x x_1^{\gamma_2-\gamma_1} \int_0^{x_1} x_2^{-\gamma_2} \int_{x_2}^\infty x_3^\beta \int_{x_3}^\infty \partial_{x_4} w(x_4) \, \d x_4 \, \tfrac{\d x_3}{x_3} \, \tfrac{\d x_2}{x_2} \, \tfrac{\d x_1}{x_1},
\]
where $w(x_4) := x_4^{-\beta} x_4 \partial_{x_4} (x_4 \partial_{x_4}-\gamma_2) (x_4 \partial_{x_4} - \gamma_1) u(x_4)$. From $\verti{u}_{\ell_1+4,\alpha_1} < \infty$ we infer $w(x_4) = o(x_4^{\alpha_1-\beta})$ as $x_4 \to \infty$, so that
\[
(B A u)(x) = - x^{\gamma_1} \int_0^x x_1^{\gamma_2-\gamma_1} \int_0^{x_1} x_2^{-\gamma_2} \int_{x_2}^\infty \partial_{x_3} \big((x_3 \partial_{x_3}-\gamma_2) (x_3 \partial_{x_3} - \gamma_1) u(x_3)\big) \, \d x_3 \, \tfrac{\d x_2}{x_2} \, \tfrac{\d x_1}{x_1}.
\]
Again, from $\verti{u}_{\ell_1+4,\alpha_1} < \infty$ we infer $(x_3 \partial_{x_3}-\gamma_2) (x_3 \partial_{x_3} - \gamma_1) u(x_3) = o(x_3^{\alpha_1})$ as $x_3 \to \infty$ and thus
\[
(B A u)(x) = x^{\gamma_1} \int_0^x x_1^{\gamma_2-\gamma_1} \int_0^{x_1} \partial_{x_2} \big(x_2^{-\gamma_2} (x_2 \partial_{x_2} - \gamma_1) u(x_2)\big) \, \tfrac{\d x_2}{x_2} \, \tfrac{\d x_1}{x_1}.
\]
Because of $\verti{u - u_0}_{\ell_2+4,\alpha_2} < \infty$ and \eqref{def-ub-invert} we have $x_2^{-\gamma_2} (x_2 \partial_{x_2} - \gamma_1) u(x_2) = - \gamma_1 x_2^{-\gamma_2} u_0 (1+o(1))$ as $x_2 \searrow 0$, that is, we have
\[
(B A u)(x) = x^{\gamma_1} \int_0^x \partial_{x_1} \big(x_1^{-\gamma_1} u(x_1)\big) \, \d x_1.
\]
Again, due to $\verti{u - u_0}_{\ell_2+4,\alpha_2} < \infty$ and \eqref{def-ub-invert} we have $x_1^{-\gamma_1} u(x_1) = x_1^{-\gamma_1} u_0 (1+o(1))$ as $x_1 \searrow 0$, which entails $(B A u)(x) = u(x)$. Hence, using that $1$ and $x^\beta$ are in the kernel of $A$ (cf.~\eqref{eq_def_p(D)}, \eqref{eq_zeros_12}, and \eqref{a-op-pd}), we infer that \eqref{left-inv-a} holds true.
\end{proof}
\begin{lemma}[Elliptic regularity, interpolation norm]\label{lem-invert-A-interp}
Suppose that $1 < p < \infty$, $\tilde k \in \N_0$, $\tilde \delta > 0$, and let $\gamma_1 < \gamma_2 < 0$ be the two negative roots of $\mathfrak p(\zeta)$ given by \eqref{eq_zeros_12}. Further assume $\tilde\delta < \min\{- \gamma_2, \beta\}$. Suppose that $v \in H_{\tilde k + 4 - \frac 4 p, -1-\tilde\delta,p} \cap H_{\tilde k + 4 - \frac 4 p, -1+\tilde\delta,p}$. Then
\begin{subequations}\label{def-u-u0-ub-invert-interp}
\begin{align}
u(x) &:= (Bv)(x) := - x^{\gamma_1} \int_0^x x_1^{\gamma_2-\gamma_1} \int_0^{x_1} x_2^{-\gamma_2} \int_{x_2}^\infty x_3^\beta \int_{x_3}^\infty x_4^{1-\beta} v(x_4) \, \tfrac{\d x_4}{x_4} \, \tfrac{\d x_3}{x_3} \, \tfrac{\d x_2}{x_2} \, \tfrac{\d x_1}{x_1}, \label{def-u-invert-interp} \\
u_0 &:= - \frac{1}{\gamma_1 \gamma_2} \int_0^\infty x_1^\beta \int_{x_1}^\infty x_2^{1-\beta} v(x_2) \, \tfrac{\d x_2}{x_2} \, \tfrac{\d x_1}{x_1}, \label{def-u0-invert-interp}
\end{align}
\end{subequations}
are well-defined with
\begin{equation}\label{limit-u-invert-0-interp}
u(x) \to u_0 \quad \text{as } x \searrow 0,
\end{equation}
and it holds
\begin{equation}\label{eq-au-v-interp}
AB v = A u = A (u-u_0) = v,
\end{equation}
and
\begin{subequations}\label{u-inv-v-est-interp}
\begin{align}
\verti{u}_{\tilde k + 8 - \frac 4 p, -\tilde\delta,p} &\sim_{\tilde k, \tilde\delta,p} \verti{v}_{\tilde k + 4 - \frac 4 p, -1-\tilde\delta,p}, \label{u-inv-v-1-interp} \\
\verti{u - u_0}_{\tilde k + 8 - \frac 4 p, \tilde\delta,p} &\sim_{\tilde k, \tilde\delta,p} \verti{v}_{\tilde k + 4 - \frac 4 p, -1+\tilde\delta,p}. \label{u-inv-v-2-interp}
\end{align}
\end{subequations}
\end{lemma}
\begin{proof}
By estimate \eqref{lower_bound_l2eta} of Lemma~\ref{lemma_char_interpol} it follows that $v \in H_{\tilde k,-1-\delta} \cap H_{\tilde k,-1+\delta}$ if $0 < \delta < \tilde \delta$. Then by \eqref{def-u-u0-ub-invert}--\eqref{eq-au-v} of Lemma~\ref{lem-invert-A}~\eqref{invert-A-1} it follows that \eqref{def-u-u0-ub-invert-interp}--\eqref{eq-au-v-interp} hold true. It remains to prove \eqref{u-inv-v-est-interp}.

\medskip

In view of \eqref{def_k,alpha,p_norm} and \eqref{eq-au-v-interp}, estimating $v$ against $u$ and $u-u_0$, respectively, is trivial. In view of the reiteration theorem \cite[Theorem~3.5.3]{Bergh_Loefstroem} and Lemma~\ref{lemma_char_interpol} it holds for $1 > \eps > 0$ (equivalence of norms)
\begin{align}
H_{\tilde k + 4 - \frac 4 p,-1\pm\tilde\delta,p} \; &\stackrel{\mathclap{\eqref{def-trace-space}}}{=} \quad \; \Big(\big(H_{\tilde k,-2+\frac 1 p \pm \tilde\delta},H_{\tilde k+4,-1+\frac 1 p \pm \tilde \delta}\big)_{1-\frac 1 p \mp \eps \tilde\delta,2}, \big(H_{\tilde k,-2+\frac 1 p \pm \tilde\delta},H_{\tilde k+4,-1+\frac 1 p \pm \tilde \delta}\big)_{1-\frac 1 p \pm \eps\tilde\delta,2}\Big)_{\frac 1 2,p} \nonumber\\
&\stackrel{\mathclap{\eqref{sobolev-frac},\eqref{eq_char_interpol}}}{=} \quad \; \big(H_{\tilde k + 4 - \frac 4 p \mp 4\eps \tilde \delta,-1 \pm (1-\eps) \tilde\delta},H_{\tilde k + 4 - \frac 4 p \pm 4\eps \tilde \delta,-1 \pm (1+\eps) \tilde\delta}\big)_{\frac 1 2, p}.\label{eq_interpol}
\end{align}
It then follows with help of the $K$-method
\begin{align*}
\verti{v}_{\tilde k + 4 - \frac 4 p, -1\pm\tilde\delta,p}^p = \int_0^\infty \tau^{-p+1} \Big(\inf_{v = v^{(1)} + v^{(2)}} \big(\verti{v^{(1)}}_{\tilde k + 4 - \frac 4 p \mp 4\eps \tilde \delta,-1 \pm (1-\eps) \tilde\delta} + \tau \verti{v^{(2)}}_{\tilde k + 4 - \frac 4 p \pm 4\eps \tilde \delta,-1 \pm (1+\eps)  \tilde\delta}\big)\Big)^p \, \tfrac{\d\tau}{\tau}.
\end{align*}
Now using \eqref{u-inv-v-est} of Lemma~\ref{lem-invert-A}~\eqref{invert-A-1} and defining $u^{(j)}$ and $u^{(j)}_0$ through \eqref{def-u-u0-ub-invert-interp} with $v$ replaced by $v^{(j)}$, we infer for $\eps > 0$ sufficiently small
\begin{align*}
\verti{v}_{\tilde k + 4 - \frac 4 p, -1+\tilde\delta,p}^p &\gtrsim_{\tilde k,\tilde\delta,p} \int_0^\infty \tau^{-p+1} \Big(\inf_{u = u^{(1)} + u^{(2)}} \big(\verti{u^{(1)}-u^{(1)}_0}_{\tilde k + 8 - \frac 4 p - 4\eps \tilde \delta,(1-\eps) \tilde\delta} \\
&\phantom{\gtrsim_{\tilde k,\tilde\delta,p} \int_0^\infty \tau^{-p+1} \Big(\inf_{u = u^{(1)} + u^{(2)}} \big(} + \tau \verti{u^{(2)}-u^{(2)}_0}_{\tilde k + 8 - \frac 4 p + 4\eps \tilde \delta, (1+\eps) \tilde\delta}\big)\Big)^p \, \tfrac{\d\tau}{\tau} \\
&= \int_0^\infty \tau^{-p+1} \Big(\inf_{u - u_0 = w^{(1)} + w^{(2)}} \big(\verti{w^{(1)}}_{\tilde k + 8 - \frac 4 p - 4\eps \tilde \delta, (1-\eps) \tilde\delta} + \tau \verti{w^{(2)}}_{\tilde k + 8 - \frac 4 p + 4\eps \tilde \delta, (1+\eps) \tilde\delta}\big)\Big)^p \, \tfrac{\d\tau}{\tau}
\end{align*}
and
\begin{align*}
\verti{v}_{\tilde k + 4 - \frac 4 p, -1-\tilde\delta,p}^p \gtrsim_{\tilde k,\tilde\delta,p} \int_0^\infty \tau^{-p+1} \Big(\inf_{u = u^{(1)} + u^{(2)}} \big(&\verti{u^{(1)}}_{\tilde k + 8 - \frac 4 p + 4\eps \tilde \delta,- (1-\eps) \tilde\delta} + \tau \verti{u^{(2)}}_{\tilde k + 8 - \frac 4 p - 4\eps \tilde \delta,- (1+\eps) \tilde\delta}\big)\Big)^p \, \tfrac{\d\tau}{\tau}.
\end{align*}
Hence, \eqref{u-inv-v-est-interp} follows by replacing $\tilde k$ by $\tilde k + 4$ in \eqref{eq_interpol}.
\end{proof}
\begin{proposition}\label{prop_sol_Au=v}
Let $I = [0,T]$ with $T \in (0,\infty)$ or $I = [0,\infty)$. Assume that $1 < p < \infty$, $k, \tilde k \in \N_0$ with $\tilde k \ge k - 4 \beta + \frac 4 p$, and $\delta, \tilde\delta > 0$ such that $\frac 1 p < \beta$ and $\delta < \tilde\delta < \min\{-\gamma_2, \beta - \frac 1 p, \frac 1 p, 1 - \frac 1 p\}$. Further suppose that $\frac 1 p -1 \pm \tilde \delta$ and $\beta -1 \pm \delta$ are in the coercivity range of $\mathfrak p(D)$ (cf.~\eqref{coerc_range} of Lemma~\ref{lem_coercive_range} for a sufficient criterion). Lastly suppose that $u^{(0)} \colon (0,\infty) \to \R$, $u^{(0)}_0 \in \R$, and $f \colon (0,\infty)^2 \to \R$ are such that
\[
u^{(0)} \in H_{\tilde k+8-\frac 4 p,-\tilde\delta,p}, \qquad u^{(0)} - u^{(0)}_0 \in H_{\tilde k+8-\frac 4 p,\tilde\delta,p},
\]
and
\begin{align*}
f &\in L^p\big(I;H_{\tilde k+4,-1+\frac{1}{p}-\tilde\delta} \cap H_{\tilde k+4,-1+\frac{1}{p}+\tilde\delta}\big), \\
\big(t \mapsto t^{\beta - \frac 1 p} f(t)\big) &\in L^p\big(I;H_{k+4, -1+\beta -\delta} \cap H_{k+4,-1+\beta +\delta}\big),
\end{align*}
that is, $u^{(0)} \in U^{(0)}(\tilde k, \tilde\delta, p)$ and $f \in F(k,\tilde k, \delta, p,I)$. Then there exist $u \colon (0,\infty)^2 \to \R$ and $u_0, u_\beta \colon (0,\infty) \to \R$ locally integrable such that
\begin{align*}
u &\in BC^0\big(I;H_{\tilde k+8-\frac 4 p,-\tilde\delta,p}\big), \\
u-u_0 &\in BC^0\big(I;H_{\tilde k+8-\frac 4 p,\tilde\delta,p}\big) \cap L^p\big(I;H_{\tilde k+8,\frac{1}{p}-\tilde\delta} \cap H_{\tilde k+8,\frac{1}{p}+\tilde\delta}\big), \\
\big(t \mapsto t^{\beta - \frac 1 p} (u(t)-u_0(t))\big) &\in BC^0\big(I;H_{k + 8 - \frac 4 p,\beta - \frac 1 p - \delta,p} \cap H_{k + 8 - \frac 4 p,\beta - \frac 1 p + \delta,p}\big), \\
\partial_t u &\in L^p\big(I;H_{\tilde k+4,-1 +\frac{1}{p} -\tilde\delta} \cap H_{\tilde k+4,-1+\frac{1}{p}+\tilde\delta}\big), \\
\big(t \mapsto t^{\beta-\frac 1 p} \partial_t u(t)\big) &\in L^p\big(I;H_{k+4,-1 +\beta-\delta} \cap H_{k+4,-1+\beta+\delta}\big), \\
\big(t \mapsto t^{\beta-\frac 1 p} (u(t)-u_0(t))\big) &\in L^p(I;H_{k+8,\beta-\delta}), \\
\big(t \mapsto t^{\beta-\frac 1 p} (u(t)-u_0(t)-u_\beta(t) x^\beta)\big) &\in L^p(I;H_{k+8,\beta+\delta}),
\end{align*}
that is, $u \in U(k,\tilde k, \delta, p,I)$ and
\begin{align}
    \vertiii{u}_{I}^p := &\sup_{t \in I} \Big[\verti{u}^p_{\tilde k+8-\frac 4 p,-\tilde\delta,p}+\verti{u-u_0}^p_{\tilde k+8-\frac 4 p,\tilde\delta,p}+t^{p\beta-1} \verti{u-u_0}^p_{k + 8 - \frac 4 p,\beta - \frac 1 p - \delta,p} \nonumber \\
    &+t^{p\beta-1} \verti{u-u_0}^p_{k + 8 - \frac 4 p, \beta - \frac 1 p + \delta,p}\Big] + \int_I \Big[\verti{\partial_t u}^p_{\tilde k+4,-1+\frac{1}{p}-\tilde\delta}+\verti{\partial_t u}^p_{\tilde k+4,-1+\frac{1}{p}+\tilde\delta} \nonumber \\
    &+ t^{p\beta-1} \verti{\partial_t u}^p_{k+4,-1+\beta-\delta}+t^{p\beta-1} \verti{\partial_t u}^p_{k+4,-1+\beta+\delta}+ \verti{u-u_0}^p_{\tilde k+8,\frac{1}{p}-\tilde\delta} \nonumber \\
    &+ \verti{u-u_0}^p_{\tilde k+8,\frac{1}{p}+\tilde\delta}
    + t^{p\beta-1} \verti{u-u_0}^p_{k+8,\beta-\delta} + t^{p\beta-1} \verti{u-u_0 -u_{\beta}x^{\beta}}^p_{k+8,\beta+\delta} \Big] \d t < \infty,\label{sol_norm_i}
\end{align}
and such that \eqref{eq_full_cauchy} is classically satisfied. Furthermore, we have the a-priori estimate
\begin{equation}\label{a-priori-perturbed}
\vertiii{u}_I \lesssim_{k, \tilde k, \delta, \tilde\delta, p} \vertiii{u^{(0)}}_0 + \vertiii{f}_{1,I},
\end{equation}
where
\begin{align}
    \vertiii{u^{(0)}}_0^p \ \ &\stackrel{\mathclap{\eqref{eq_def_norm_u^0}}}{=} \ \ \verti{u^{(0)}}^p_{\tilde k+8-\frac 4 p,-\tilde\delta,p}+\verti{u^{(0)}-u_0^{(0)}}^p_{\tilde k+8-\frac 4 p,\tilde\delta,p}, \nonumber \\
    \vertiii{f}^p_{1,I} &:= \int_I \Big[\verti{f}^p_{\tilde k+4,-1+\frac{1}{p}-\tilde\delta}+\verti{f}^p_{\tilde k+4,-1+\frac{1}{p}+\tilde\delta}+t^{p\beta-1}\verti{f}^p_{k+4, -1+\beta -\delta}+t^{p\beta-1}\verti{f}^p_{k+4,-1+\beta +\delta}\Big] \d t, \label{f_norm_i}
\end{align}
and $C < \infty$ only depends on $k$, $\tilde k$, $\delta$, $\tilde\delta$, and $p$.
\end{proposition}
\begin{figure}[htb]
     \centering
     \begin{subfigure}[b]{0.45\textwidth}
         \centering
         \includegraphics[width=1.1\textwidth]{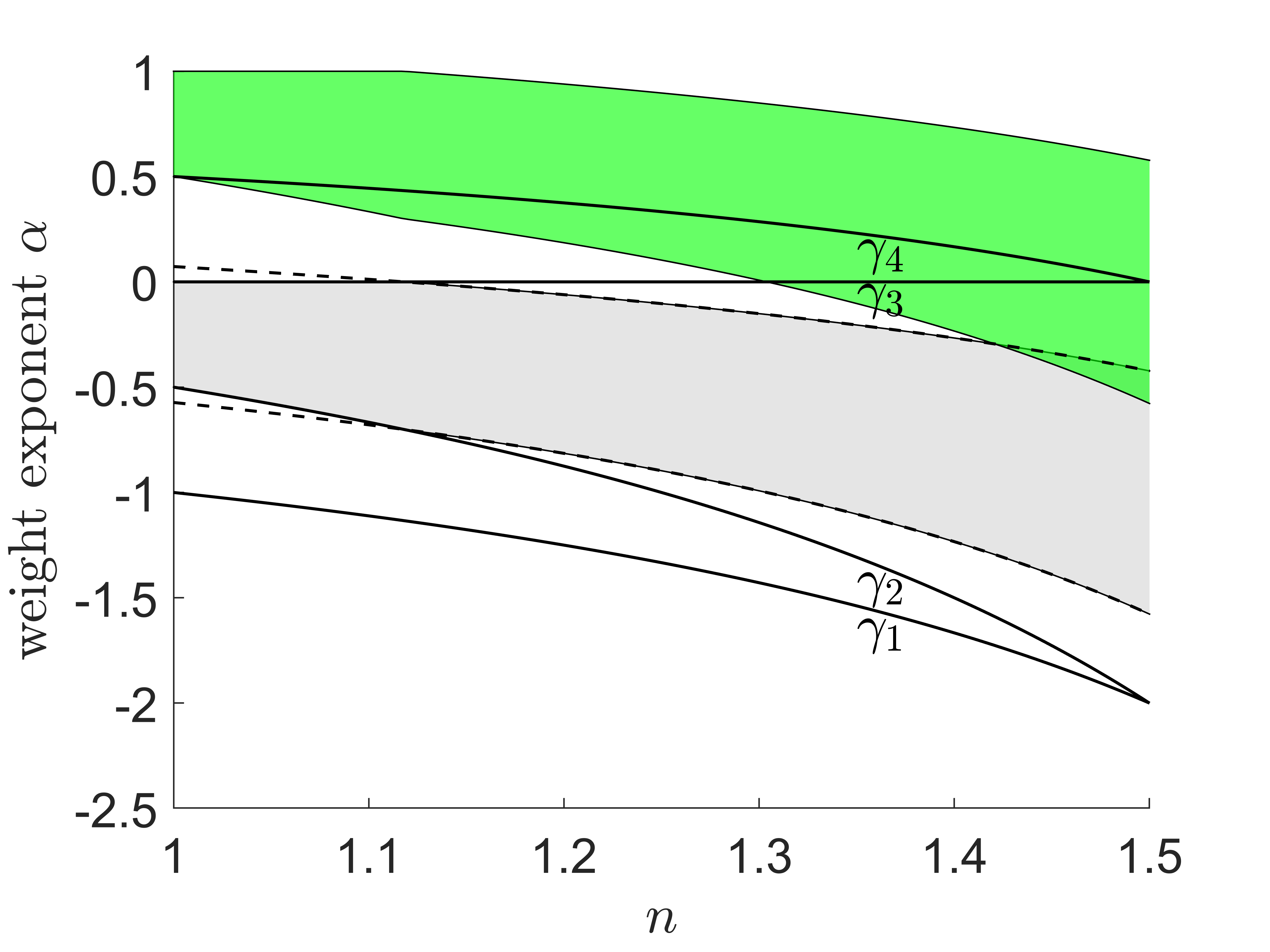}
         \caption{$n\in (1,\tfrac{3}{2})$}
         \label{fig:coercivity_range_a_shift+1}
     \end{subfigure}
     \hfill
     \begin{subfigure}[b]{0.45\textwidth}
         \centering
         \includegraphics[width=1.1\textwidth]{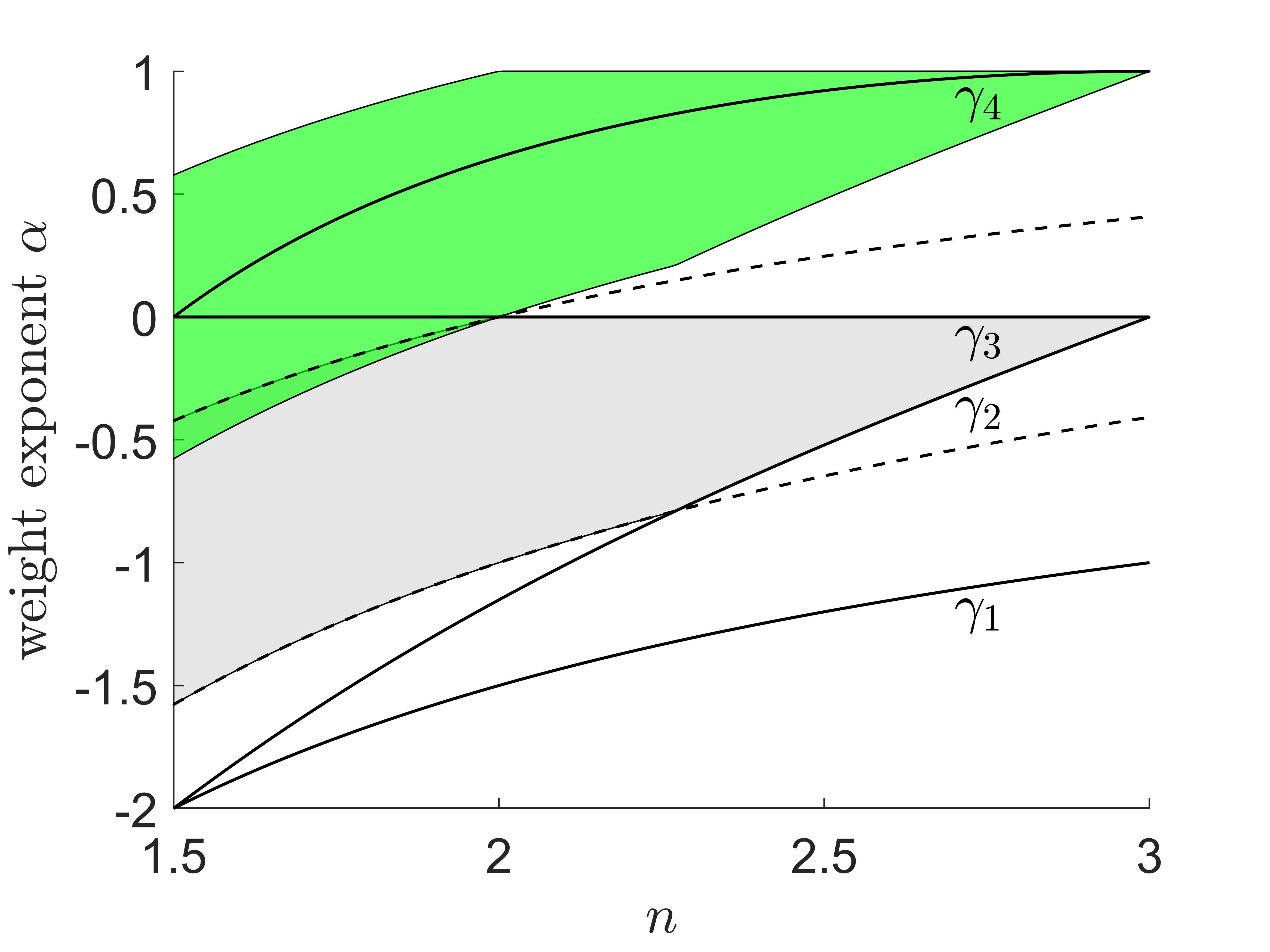}
         \caption{$n\in(\tfrac{3}{2},3)$}
         \label{fig:coercivity_range_b_shift+1}
     \end{subfigure}
        \caption{For the two different cases of $n$ the zeros $\gamma_1,\dots,\gamma_4$ of $\mathfrak p(D)$ (solid line) and the upper and lower bound in \eqref{eq_coer_condition2_n_1_32} and \eqref{eq_coer_condition2_n_32_3} (dashed line) are shown. The coercivity range for $\alp$ contains the shaded area according to \eqref{coerc_range} of Lemma~\ref{lem_coercive_range}. This area shifted up by 1 is shown in green (shaded darker, green color online).}
        \label{fig:coercivity_range_shift}
\end{figure}
\begin{proof}
Define $v^{(0)} := A u^{(0)} = A (u^{(0)} - u^{(0)}_0)$ (cf.~\eqref{eq_def_p(D)}, \eqref{eq_zeros_12}, \eqref{a-op-pd}), then we have $v^{(0)} \in H_{\tilde k + 4 - \frac 4 p,-1-\tilde\delta,p} \cap H_{\tilde k + 4 - \frac 4 p,-1+\tilde\delta,p}$. By approximation, we may assume that $v^{(0)} \in C^\infty_\mathrm{c}((0,\infty))$. Furthermore, define $g := A f$ so that
\begin{align*}
g &\in L^p\big(0,\infty;H_{\tilde k,-2+\frac{1}{p}-\tilde\delta} \cap H_{\tilde k,-2+\frac{1}{p}+\tilde\delta}\big), \\
\big(t \mapsto t^{\beta - \frac 1 p} g(t)\big) &\in L^p\big(0,\infty;H_{k,-2+\beta -\delta} \cap H_{k,-2+\beta +\delta}\big).
\end{align*}
By Proposition~\ref{proposition_estimate_almostMR} we have a unique classical solution $v = v^{(1)} + v^{(2)} \colon (0,\infty)^2 \to \R$ of the initial-value problem \eqref{eq_full_cauchy_v} (where $v^{(1)}$ solves \eqref{eq_full_cauchy_v} with $g = 0$ and $v^{(2)}$ solves \eqref{eq_full_cauchy_v} with $v^{(0)} = 0$) with regularity
\begin{align*}
v^{(\ell)} &\in BC^0\big([0,\infty);H_{\tilde k+4-\frac 4 p,-1-\tilde\delta,p} \cap H_{\tilde k+4-\frac 4 p,-1+\tilde\delta,p}\big), \\
\partial_t v^{(\ell)} &\in L^p\big(0,\infty;H_{\tilde k,-2+\frac{1}{p}-\tilde\delta} \cap H_{\tilde k,-2+\frac{1}{p}+\tilde\delta}\big), \\
v^{(\ell)} &\in L^p\big(0,\infty;H_{\tilde k+4,-1+\frac{1}{p}-\tilde\delta} \cap H_{\tilde k+4,-1+\frac{1}{p}+\tilde\delta}\big),
\end{align*}
and
\begin{align*}
\big(t \mapsto t^{\beta - \frac 1 p} v^{(\ell)}(t)\big) &\in BC^0\big(I;H_{k + 4 - \frac 4 p,-1+\beta - \frac 1 p - \delta,p} \cap H_{k + 4 - \frac 4 p,-1+\beta - \frac 1 p + \delta,p}\big), \\
\big(t \mapsto t^{\beta-\frac 1 p} \partial_t v^{(\ell)}(t)\big) &\in L^p\big(I;H_{k,-2+\beta-\delta} \cap H_{k,-2+\beta+\delta}\big), \\
\big(t \mapsto t^{\beta-\frac 1 p} v^{(\ell)}(t)\big) &\in L^p(I;H_{k+4,-1+\beta-\delta} \cap H_{k+4,-1+\beta+\delta})
\end{align*}
if $\verti{I} < \infty$, which by \eqref{eq_second_MR_estimate} satisfies the maximal-regularity estimates
\begin{subequations}\label{est-v-unw-w}
\begin{align}
&\sup_{t \in I} \sum_{\ell = 1}^2 \Big[\verti{v^{(\ell)}}^p_{\tilde k+4-\frac 4 p,-1-\tilde\delta,p}+\verti{v^{(\ell)}}^p_{\tilde k+4-\frac 4 p,-1+\tilde\delta,p}\Big] \nonumber \\
&+ \int_I \sum_{\ell = 1}^2 \Big[\verti{\partial_t v^{(\ell)}}^p_{\tilde k,-2+\frac{1}{p}-\tilde\delta}+\verti{\partial_t v^{(\ell)}}^p_{\tilde k,-2+\frac{1}{p}+\tilde\delta} + \verti{v^{(\ell)}}^p_{\tilde k+4,-1+\frac{1}{p}-\tilde\delta} + \verti{v^{(\ell)}}^p_{\tilde k+4,-1+\frac{1}{p}+\tilde\delta}\Big] \d t \nonumber \\
& \quad \lesssim_{\tilde k,\tilde \delta, p} \verti{v^{(0)}}^p_{\tilde k+4-\frac 4 p,-1-\tilde\delta,p}+\verti{v^{(0)}}^p_{\tilde k+4-\frac 4 p,-1+\tilde\delta,p} + \int_I \Big[\verti{g}^p_{\tilde k,-2+\frac{1}{p}-\tilde\delta}+ \verti{g}^p_{\tilde k,-2+\frac{1}{p}+\tilde\delta}\Big] \d t, \label{est-v-unweighted}
\end{align}
and
\begin{align}
&\sup_{t \in I} t^{p\beta-1} \sum_{\ell = 1}^2 \Big[\verti{v^{(\ell)}}^p_{k + 4 - \frac 4 p,-1+\beta - \frac 1 p - \delta,p} + \verti{v^{(\ell)}}^p_{k + 4 - \frac 4 p, -1+\beta - \frac 1 p + \delta,p}\Big] \nonumber \\
&+ \int_I t^{p\beta-1} \sum_{\ell = 1}^2 \big[\verti{\partial_t v^{(\ell)}}^p_{k,-2+\beta-\delta}+\verti{\partial_t v^{(\ell)}}^p_{k,-2+\beta+\delta} + \verti{v^{(\ell)}}^p_{k+4,-1+\beta-\delta} + \verti{v^{(\ell)}}^p_{k+4,-1+\beta+\delta} \big] \d t \nonumber \\
& \quad \lesssim_{k,\delta, p} \int_I t^{p\beta-1}\big[\verti{g}^p_{k,-2+\beta -\delta}+\verti{g}^p_{k,-2+\beta +\delta}\big] \d t \nonumber \\
&\quad \phantom{\lesssim_{k,\delta, p}} + (p \beta-1) \int_I t^{p\beta-2} \sum_{\ell = 1}^2 \Big[\verti{v^{(\ell)}}^p_{k + 4 - \frac 4 p, - 1 + \beta - \frac 1 p - \delta,p} + \verti{v^{(\ell)}}^p_{k + 4 - \frac 4 p, - 1 + \beta - \frac 1 p + \delta,p}\Big] \d t \label{est-v-weighted}
\end{align}
\end{subequations}
for $\verti{I} < \infty$. By \eqref{def_k,alpha,p_norm}, \eqref{sobolev-frac}, \eqref{eq_char_interpol} of Lemma~\ref{lemma_char_interpol}, and the reiteration theorem \cite[Theorem~3.5.3]{Bergh_Loefstroem} we have (with equivalence of norms)
\begin{align*}
H_{k+4-\frac 4 p,-1+\beta-\frac 1 p \pm \delta,p} &= \Big(\big(H_{k,-2+\beta\pm\delta}, H_{k+4,-1+\beta\pm\delta}\big)_{1-\frac{p\beta-1}{p},2}, H_{k+4,-1+\beta\pm\delta}\Big)_{\frac{p\beta-2}{p\beta-1}, p} \\
&= \Big(H_{k+4-4\beta + \frac 4 p,-1+\frac 1 p \pm \delta}, H_{k+4,-1+\beta \pm \delta}\Big)_{\frac{p\beta-2}{p\beta-1}, p}.
\end{align*}
This entails by interpolation (cf.~\cite[Theorem~3.1.2]{Bergh_Loefstroem}) and using $\tilde k \ge k - 4 \beta + \frac 4 p$,
\begin{align*}
& \int_I t^{p\beta-2} \Big[\verti{v^{(\ell)}}^p_{k + 4 - \frac 4 p, - 1 + \beta - \frac 1 p - \delta,p} + \verti{v^{(\ell)}}^p_{k + 4 - \frac 4 p, - 1 + \beta - \frac 1 p + \delta,p}\Big] \d t \\
& \quad \lesssim_{\beta,p} \int_I \big[\verti{v^{(\ell)}}_{\tilde k + 4, - 1 + \frac 1 p - \delta}^p + \verti{v^{(\ell)}}_{k + 4, - 1 + \frac 1 p + \delta}^p\big]^{\frac{1}{p \beta - 1}} \\
& \phantom{ \quad \lesssim_{\beta,p} \int_I} \times \big[t^{p\beta-1} \verti{v^{(\ell)}}_{\tilde k + 4, - 1 + \frac 1 p - \delta}^p + t^{p\beta-1} \verti{v^{(\ell)}}_{k + 4, - 1 + \frac 1 p + \delta}^p\big]^{\frac{p \beta - 2}{p \beta - 1}} \d t \\
& \quad \le \Big(\int_I \big[\verti{v^{(\ell)}}_{\tilde k + 4, - 1 + \frac 1 p - \delta}^p + \verti{v^{(\ell)}}_{k + 4, - 1 + \frac 1 p + \delta}^p\big] \, \d t\Big)^{\frac{1}{p \beta - 1}} \\
& \quad \phantom{\le} \times \Big(\int_I \big[t^{p\beta-1} \verti{v^{(\ell)}}_{\tilde k + 4, - 1 + \frac 1 p - \delta}^p + t^{p\beta-1} \verti{v^{(\ell)}}_{k + 4, - 1 + \frac 1 p + \delta}^p\big] \d t\Big)^{\frac{p \beta - 2}{p \beta - 1}},
\end{align*}
where H\"older's inequality was used in the second step. Hence, after applying Young's inequality, estimates~\eqref{est-v-unw-w} can be combined to
\begin{align}
&\sup_{t \in I} \Big[\verti{v}^p_{\tilde k+4-\frac 4 p,-1-\tilde\delta,p}+\verti{v}^p_{\tilde k+4-\frac 4 p,-1+\tilde\delta,p}+t^{p\beta-1} \verti{v}^p_{k + 4 - \frac 4 p,-1+\beta - \frac 1 p - \delta,p}+t^{p\beta-1} \verti{v}^p_{k + 4 - \frac 4 p, -1+\beta - \frac 1 p + \delta,p}\Big] \nonumber \\
&+ \int_I \Big[\verti{\partial_t v}^p_{\tilde k,-2+\frac{1}{p}-\tilde\delta}+\verti{\partial_t v}^p_{\tilde k,-2+\frac{1}{p}+\tilde\delta} + t^{p\beta-1} \verti{\partial_t v}^p_{k,-2+\beta-\delta} + t^{p\beta-1} \verti{\partial_t v}^p_{k,-2+\beta+\delta}\Big] \d t \nonumber \\
&+ \int_I \Big[\verti{v}^p_{\tilde k+4,-1+\frac{1}{p}-\tilde\delta} + \verti{v}^p_{\tilde k+4,-1+\frac{1}{p}+\tilde\delta}+t^{p\beta-1}\verti{v}^p_{k+4,-1+\beta-\delta} +t^{p\beta-1}\verti{v}^p_{k+4,-1+\beta+\delta} \Big] \d t \nonumber \\
& \quad \lesssim_{k,\tilde k, \delta, \tilde \delta, p} \verti{v^{(0)}}^p_{\tilde k+4-\frac 4 p,-1-\tilde\delta,p}+\verti{v^{(0)}}^p_{\tilde k+4-\frac 4 p,-1+\tilde\delta,p} \nonumber \\
& \quad \phantom{\lesssim_{k,\tilde k, \delta, \tilde \delta, p}} + \int_I \Big[\verti{g}^p_{\tilde k,-2+\frac{1}{p}-\tilde\delta}+\verti{g}^p_{\tilde k,-2+\frac{1}{p}+\tilde\delta}+t^{p\beta-1}\verti{g}^p_{k,-2+\beta -\delta}+t^{p\beta-1}\verti{g}^p_{k,-2+\beta +\delta}\Big] \d t.
\label{est-v-combined}
\end{align}
Trivially, the right-hand side of \eqref{est-v-combined} is bounded (up to a constant) by the right-hand side of \eqref{a-priori-perturbed}.
Define $u$ and $u_0$ through \eqref{def-u-u0-ub-invert}, and $u_\beta$ through \eqref{def-ub-invert} of Lemma~\ref{lem-invert-A}. Then by the elliptic regularity estimates~\eqref{u-inv-v-est} and \eqref{u-inv-v-3} of Lemma~\ref{lem-invert-A}, and estimates~\eqref{u-inv-v-est-interp} of Lemma~\ref{lem-invert-A-interp}, we infer from \eqref{est-v-weighted} that indeed \eqref{a-priori-perturbed} is satisfied for any $\verti{I} < \infty$, so that also $I = [0,\infty)$ holds true. Using \eqref{eq-au-v} and \eqref{left-inv-a} in \eqref{eq_full_cauchy_v}, we infer that \eqref{eq_full_cauchy} is classically satisfied.
\end{proof}
%

\section{The Nonlinear Problem}\label{sec-nonlinear}
In this section, the nonlinear problem \eqref{eq_nonlin_cauchy} will be treated. Subsequently, the main result of \S\ref{sec_main_result} will be proven in \S\ref{sec_proof_main_result}. Before this can be done, we need to show a few preliminary estimates.

\subsection{Nonlinear Estimates\label{ssec:non-est}}

%
\begin{lemma}\label{lem_coeff_est}
In the situation of Proposition~\ref{prop_sol_Au=v}, the coefficients $u_0$ and $u_{\beta}$, and the solution $u$ satisfy
\[
u_0\in BC^0(I), \ \ \big(t \mapsto t^{\beta-\frac 1 p} u_{\beta}(t)\big) \in L^p(I), \ \ u \in BC^0(I;U^{(0)}(\tilde k, \tilde \delta, p)), \ \ u \in BC^m(I \times [0,\infty)),
\]
where $m \in \N_0$ with $m < \tilde k + \frac{15}{2} - \frac 4 p$. In particular, we have $u_0(t) \to u_0^{(0)}$ as $t \searrow 0$. Furthermore, it holds $u_0(t) \to 0$, $\vertiii{u(t)}_0 \to 0$, and $\max_{j = 0,\ldots,m} \sup_{x > 0} \verti{D^j u(t,x)} \to 0$ as $t \to \infty$ if $I = [0,\infty)$. Additionally, the following inequality holds true
\begin{equation}\label{eq_ineq_coeff}
    \sup_{t \in I} \, \verti{u_0}^p + \int_I t^{p\beta-1} \verti{u_{\beta}}^p \, \d t + \vertii{u}_{m,I} \lesssim_{\tilde k, \delta, p} \vertiii{u}_I^p,
\end{equation}
where
\[
\vertii{u}_{m,I} := \max_{j = 0,\ldots,m} \sup_{(t,x) \in I \times (0,\infty)} \verti{D^j u(t,x)} \qquad \text{and} \qquad \vertii{\cdot}_m := \vertii{\cdot}_{m,[0,\infty)}.
\]
\end{lemma}
\begin{proof}
Define $\eta(x) := x^2 \mathbbm{1}_{[0,1]}(x) + x^{-2} (1-\mathbbm{1}_{[0,1]}(x))$. Then it holds
\begin{align*}
    |u_0(t)|^p &= \int_1^2 |u_0(t)|^p \d x \lesssim_{\delta, p} \int_1^2 \eta^2 x^{-2\delta} |u(t,x)-u_0(t)|^p \tfrac{\d x}{x} + \int_1^2 \eta^2  x^{2\delta} |u(t,x)|^p \tfrac{\d x}{x} \\
    &\lesssim_{\tilde k, \delta, p} \verti{u(t)-u_0(t)}^p_{\tilde k + 8 - \frac 4 p,-\tilde\delta,p}+\verti{u(t)}^p_{\tilde k + 8 - \frac 4 p,\tilde\delta,p}
\end{align*}
by \eqref{lower_bound_l2eta} of Lemma~\ref{lemma_char_interpol}. Taking the supremum gives in combination with estimate~\eqref{eq_lower_bound_interpol} of Lemma~\ref{lemma_char_interpol}
\begin{align}
\sup_{t \in I} |u_0(t)|^p + \vertii{u}_{m,I}^p &\lesssim_{\tilde k, \delta, p} \sup_{t \in I} \vertiii{u(t)}_0^p = \sup_{t \in I} \big(\verti{u(t)-u_0(t)}^p_{\tilde k + 8 - \frac 4 p,-\tilde\delta,p}+\verti{u(t)}^p_{\tilde k + 8 - \frac 4 p,\tilde\delta,p}\big) \label{est-u0-u-inf} \\
&\stackrel{\mathclap{\eqref{sol_norm_i}}}{\le} \ \ \vertiii{u}_I^p \stackrel{\eqref{a-priori-perturbed}}{\lesssim_{k,\tilde k, \delta, \tilde\delta, p}} \vertiii{u^{(0)}}_0^p + \vertiii{f}_{1,I}^p < \infty. \nonumber
\end{align}
Using \eqref{u-inv-v-est-interp} of Lemma~\ref{lem-invert-A-interp} and $Au =: v$, we have
\[
\sup_{t \in I} \vertiii{u(t)}_0^p \lesssim_{\tilde k, \delta, p} \sup_{t \in I} \big(\verti{v(t)}^p_{\tilde k + 4 - \frac 4 p,-1-\tilde\delta,p}+\verti{v(t)}^p_{\tilde k + 4 - \frac 4 p,-1+\tilde\delta,p}\big).
\]
By extending $f$ by $0$, we can without loss of generality assume $I = [0,\infty)$. Using the trace method for real interpolation (see e.g.~\cite[Definition~1.2.8]{Lunardi}) then gives
\begin{align*}
&\sup_{t \ge 0} \vertiii{u(t)}_0^p \\
&\lesssim_{\tilde k, \delta, p} \int_0^{\infty} \big(\verti{\partial_t v(t)}^p_{\tilde k,-2-\tilde\delta+ \frac 1 p}+\verti{\partial_t v(t)}^p_{\tilde k,-2+\tilde\delta+ \frac 1 p}+\verti{v(t)}^p_{\tilde k+4,-1-\tilde\delta+ \frac 1 p}+\verti{v(t)}^p_{\tilde k+4,-1-\tilde\delta+ \frac 1 p}\big) \d t.
\end{align*}
Applying an even reflection in time (with reflected quantities $u^r(t)$, $u_0^\textrm{r}$, and $v^\textrm{r}$) yields
\begin{align*}
&\sup_{t \ge 0} \vertiii{u^r(t)}_0^p \\
& \quad \lesssim_{\tilde k, \delta, p} \int_\R \big(\verti{\partial_t v^\textrm{r}(t)}^p_{\tilde k,-2-\tilde\delta+ \frac 1 p}+\verti{\partial_t v^\textrm{r}(t)}^p_{\tilde k,-2+\tilde\delta+ \frac 1 p}+\verti{v^\textrm{r}(t)}^p_{\tilde k+4,-1-\tilde\delta+ \frac 1 p}+\verti{v^\textrm{r}(t)}^p_{\tilde k+4,-1+\tilde\delta+ \frac 1 p}\big) \d t.
\end{align*}
Take a test function $\phi \in C^\infty_\mathrm{c}(\R)$ with $\int_{\R} \phi \, \d x = 1$ and mollify the left- and right-hand side with $\phi_{\varepsilon}(t):=\tfrac{1}{\varepsilon}\phi(\tfrac{t}{\varepsilon})$. We can then approximate the right-hand side with smooth and compactly supported functions in time. Hence, this also holds for the left-hand side, implying that $u \in BC^0(I;U^{(0)}(\tilde k, \tilde\delta, p))$ and $\vertiii{u(t)}_0 \to 0$ as $t \to \infty$ if $I = [0,\infty)$. Combining this with \eqref{est-u0-u-inf}, we also get $u_0\in BC^0(I)$, $u \in BC^m(I \times [0,\infty))$, as well as $u_0(t) \to 0$ as $t \to \infty$ and $\max_{j = 0,\ldots,m} \sup_{x > 0} \verti{D^j u(t,x)} \to 0$ as $t \to \infty$ if $I = [0,\infty)$.

\medskip

For $u_{\beta}$ we have
\begin{align}
|u_{\beta}(t)|^p &\lesssim \int_1^2 |u_{\beta}(t) x^{\beta}|^p \d x \lesssim_p \int_1^2 |u(t)-u_0(t)|^p \d x +\int_1^2 |u(t)-u_0(t)-u_{\beta}(t)x^{\beta}|^p \d x \nonumber \\
    &\lesssim_{\delta} \int_0^{\infty} x^{-2(\beta-\delta)}|u(t)-u_0(t)|^p \tfrac{\d x}{x}+\int_0^{\infty}x^{-2(\beta+\delta)}|u(t)-u_0(t)-u_{\beta}(t)x^{\beta}|^p\tfrac{\d x}{x} \nonumber \\
    &\leq \verti{u-u_0}^p_{\tilde{k}+8,\beta-\delta} + \verti{u-u_0-u_{\beta}x^{\beta}}^p_{\tilde{k}+8,\beta+\delta}. \label{eq_u_beta}
\end{align}
Multiplying both sides with $t^{p\beta-1}$ and integrating over $t \in I$ gives, in view of \eqref{sol_norm_i} and \eqref{a-priori-perturbed}, $(t \mapsto t^{\beta-\frac 1 p} u_\beta(t)) \in L^p(I)$, and finalizes the proof of estimate \eqref{eq_ineq_coeff}.
\end{proof}

We need the following nonlinear estimates, forming the $L^p_t$-generalization of \cite[Lemma~8.1]{GGKO} with a subtle difference when treating the with respect to the $x^\beta$-contribution supercritical terms in the norm $\vertiii{\mathcal{N}(u)}_{1,I}$ (see \eqref{est-k8delta} below).
\begin{lemma}\label{lemma_estimates_for_fixed_point}
Let $1 < p < \infty$ with $\frac 1 p < \beta$, $0< \delta < \tilde\delta < \min\{\beta - \frac 1 p, \frac 1 p\}$, and $k,\tilde k \in \N_0$ with
\begin{equation}\label{constraint-k}
\tilde k > k + \tfrac{1}{2} + \tfrac 3 p.
\end{equation}
For $u,u^{(1)},u^{(2)} \in U(k,\tilde k, \delta, p,I)$ the following estimates hold for the nonlinearity (cf.~\eqref{eq_def_nonlin})
\begin{subequations}
\begin{align}
    \vertiii{\mathcal{N}(u)}_{1,I} &\lesssim_{k,\tilde k, \delta, \tilde\delta, p} \max_{j=2,5}\vertiii{u}_I^j, \label{eq_bound_N(u)}\\
    \vertiii{\mathcal{N}(u^{(1)})-\mathcal{N}(u^{(2)})}_{1,I} &\lesssim_{k,\tilde k, \delta, \tilde\delta, p} \max_{j=1,4}(\vertiii{u^{(1)}}_I+\vertiii{u^{(2)}}_I)^j\vertiii{u^{(1)}-u^{(2)}}_I.\label{eq_bound_N(u1)-N(u2)}
\end{align}
\end{subequations}
\end{lemma}

\begin{proof}
\textbf{Proof of estimate~\eqref{eq_bound_N(u)}.} We start with deriving \eqref{eq_bound_N(u)}. For both regimes $1 < n < \frac 32$ and $\frac 32 < n < 3$, $\mathcal{N}(u)$ has the form
\begin{equation}\label{nonlinearity-structure-1}
\NN(u) \stackrel{\eqref{eq_def_nonlin}}{=} -x^{-1}\MMM(u+1,\dots,u+1) + x^{-1}\mathfrak p(D)u,
\end{equation}
where $\mathfrak p(D)$ and $\MMM$ are defined as in \eqref{eq_def_p(D)} and \eqref{eq_def_M}, respectively. Hence, $\NN(u)$ consists of the super-linear terms of $x^{-1} \MMM(u+1,\dots,u+1)$. Denote by $\MMM_\mathrm{sym}$ the symmetrization of $\MMM$. Note that $\MMM_\mathrm{sym}$ is multi-linear in all of its arguments. Because of this, $\NN(u)$ can be written as a linear combination with constant coefficients of terms of the form
\begin{equation}\label{eq_form_N}
x^{-1}\MMM_\mathrm{sym}(u,u,w^{(3)},w^{(4)},w^{(5)}),\quad\text{with } w^{(3)},w^{(4)},w^{(5)} \in\{u, 1\}.
\end{equation}
Using the multi-linearity of $\MMM_\mathrm{sym}$ and noting that evaluating $\MMM_\mathrm{sym}$ for constants in all arguments yields zero because of \eqref{nonlinearity-structure-1} and
\begin{equation}\label{nonlinearity-structure-2}
\MMM_\mathrm{sym}(1,1,1,1,1) \stackrel{\eqref{eq_def_p(D)_1<n<3/2}, \eqref{eq_def_p(D)_3/2<n<3}}{=} \tfrac{1}{5} \mathfrak p(D)1 = 0,
\end{equation}
we see that only terms that consist of at least one in $x$ non-constant argument of $\MMM_\mathrm{sym}$ remain and therefore, $\NN(u)$ can be written as a linear combination with constant coefficients of
\begin{equation}\label{eq_rewriting_M_step2}
x^{-1} \MMM_\mathrm{sym}(u-u_0,w^{(2)},w^{(3)},w^{(4)},w^{(5)}), \quad w^{(2)}\in\{u-u_0,u_0\}, \ w^{(3)},w^{(4)},w^{(5)} \in \{u-u_0,u_0,1\}.
\end{equation}
We consider $\vertiii{\NN(u)}_{1,I}$ (cf.~\eqref{f_norm_i}) consisting of norms of $\NN(u)$ with weights $-1+\frac{1}{p}\pm\delta, -1+\beta \pm \delta$ and $\tilde k+4$, respectively $k+4$ derivatives. First consider the case which is subcritical with respect to the term $x^{-1+\beta}$, which consists of the parts of the norm with weights $-1+\tfrac{1}{p}\pm \delta$ and $-1+\beta-\delta$. We explain of which terms $D^l (x^{-1} \NN(u))$, with $l \le k+4$ or $l \leq \tilde k +4$, consists. Note that $\MMM_\mathrm{sym}$ distributes four derivatives onto its arguments. Hence, $D^l (x^{-1} \NN(u))$ can be written as a product of terms $x^{-1} v^{(1)} \times  v^{(2)} \times v^{(3)} \times v^{(4)} \times v^{(5)}$, where
\begin{subequations}\label{v1-v5-sub}
\begin{alignat}{2}
    &v^{(1)} = D^{l_1}(u-u_0),\quad &&l_1\leq l+4 \\
    &v^{(2)} \in \{D^{l_2}(u-u_0),u_0\},\quad &&l_2\leq \tfrac{l+4}{2},\\
    &v^{(3)}, v^{(4)}, v^{(5)} \in \{D^{l_3}(u-u_0),u_0,1\},\quad &&l_3\leq \tfrac{l+4}{2}.
\end{alignat}
\end{subequations}
From \eqref{eq_ineq_coeff} of Lemma~\ref{lem_coeff_est}, we note that the supremum over $(t,x) \in I \times (0,\infty)$ of $v^{(2)}$, $v^{(3)}$, $v^{(4)}$, and $v^{(5)}$ can be bounded by $\vertiii{u}$ provided $\frac{\max\{k,\tilde k\}+8}{2} < \tilde k + \frac{15}{2} - \frac 4 p$, that is,
\[
\tilde k > \tfrac k 2 - \tfrac 7 2 + \tfrac 4 p \qquad \text{and} \qquad \tilde k > \tfrac 8 p - 7,
\]
which is implied by \eqref{constraint-k}. Hence, for the part of $\vertiii{\NN(u)}_{1,I}$ consisting of the norms with weights $-1+\frac 1 p \pm\delta$ and $-1+\beta-\delta$ the following estimate holds
\begin{align}
    &\int_I \Big(\verti{\NN(u)}^p_{\tilde k+4,-1+\frac{1}{p}-\tilde\delta} + \verti{\NN(u)}^p_{\tilde k+4,-1+\frac{1}{p}+\tilde\delta}+t^{p\beta-1}\verti{\NN(u)}^p_{k+4,-1+\beta-\delta}\Big) \d t \nonumber\\
    &\lesssim_{k,\tilde k, \delta, \tilde\delta, p} \int_I \Big(\verti{u-u_0}^p_{\tilde k+8,\tfrac{1}{p}-\tilde\delta}+\verti{u-u_0}^p_{\tilde k+8,\tfrac{1}{p}+\tilde\delta}+t^{p\beta-1}\verti{u-u_0}^p_{k+8,\beta-\delta}\Big) \d t\nonumber \\
    &\phantom{\lesssim_{k,\tilde k, \delta, \tilde\delta, p}} \times \big(\sup_{t\geq0} \big(\vertii{u-u_0}_{l_2,I}^p +|u_0|^p\big)\big) \times \big(1+\big(\sup_{t\geq 0}(\vertii{u-u_0}_{l_2,I}^p+|u_0|^p)\big)^3\big) \nonumber\\
    &\stackrel{\eqref{sol_norm_i},\eqref{eq_ineq_coeff}}{\lesssim_{\tilde k, \delta ,p}}\vertiii{u}_I^{2p}(1+\vertiii{u}_I^{3p}).\label{eq_bound_N}
\end{align}

Now, consider the supercritical case, consisting of the part $\int_I t^{p\beta-1}\verti{\NN(u)}^p_{k+4,-1+\beta+\delta} \, \d t$ of the norm $\vertiii{\NN(u)}_{1,I}$ with weight $\beta+\delta$. In this case, \eqref{eq_rewriting_M_step2} turns into one of the following
\begin{align}
x^{-1}\MMM_\mathrm{sym}(u-u_0-u_{\beta}x^{\beta},w^{(2)},w^{(3)},w^{(4)},w^{(5)}),\quad &w^{(2)}\in\{u-u_0,u_0\}, \nonumber \\
& w^{(3)},w^{(4)},w^{(5)}\in\{u-u_0,u_0,1\}\label{eq_rewriting_M_>beta_1},\\
x^{-1}\MMM_\mathrm{sym}(u_\beta x^{\beta},w^{(2)},w^{(3)},w^{(4)},w^{(5)}),\quad &w^{(2)}\in\{u-u_0,u_0\}, \nonumber \\
& w^{(3)},w^{(4)},w^{(5)} \in\{u-u_0,u_0,1\}.\label{eq_rewriting_M>beta_2}
\end{align}
Noting that
\begin{equation}\label{eq_reason_nonconstant}
\MMM_\mathrm{sym}(x^{\beta},1,1,1,1)=\tfrac{1}{5} \mathfrak p(D)x^{\beta}=0,
\end{equation}
since $\beta$ is a root of $\mathfrak p(D)$, \eqref{eq_rewriting_M>beta_2} can be simplified to
\begin{equation}\label{eq_rewriting_m>beta_3}
x^{-1} \MMM_\mathrm{sym}(u_{\beta}x^{\beta},u-u_0,w^{(3)},w^{(4)},w^{(5)}),\quad w^{(3)},w^{(4)},w^{(5)}\in\{u-u_0,u_0,1\}.
\end{equation}
We will argue similarly to before for the term $\int_I t^{p\beta-1}\verti{\NN(u)}^p_{k+4,-1+\beta+\delta} \, \d t$. In both cases we use that again $\MMM_\mathrm{sym}$ distributes four derivatives onto its arguments. First, considering the form as given in \eqref{eq_rewriting_M_>beta_1}. These terms in $D^l (x^{-1} \NN(u))$, where $l\leq k+4$, are a linear combination with constant coefficients of terms of the form $x^{-1} v^{(1)}\times\dots\times v^{(5)}$ given by
\begin{subequations}\label{v1-v5-beta-1}
\begin{alignat}{2}
    &v^{(1)} = D^{l_1}(u-u_0-u_{\beta}x^{\beta}),\quad &&l_1\leq l+4\leq k+8,\\
    &v^{(2)} \in \{D^{l_2}(u-u_0),u_0\},\quad &&l_2\leq l+4\leq k+8,\\
    &v^{(3)}, v^{(4)}, v^{(5)} \in \{D^{l_3}(u-u_0),u_0,1\}, \quad &&l_3\leq \tfrac{l+4}{2} \leq k+8.
\end{alignat}
\end{subequations}
This gives
\begin{align}\nonumber
    \verti{x^{-1} v^{(1)} \times \dots \times v^{(5)}}^p_{-1+\beta+\delta}\lesssim& \verti{u-u_0-u_{\beta}x^{\beta}}^p_{k+8,\beta+\delta} \big(\vertii{u-u_0}_{k+8,I}^p+|u_0|^p\big) \\
    &\times \big(1+(\vertii{u-u_0}_{k+8,I}^p+|u_0|^p)^3\big). \label{eq_bound_v_1...v_5_1}
\end{align}
\medskip

Secondly, considering \eqref{eq_rewriting_m>beta_3} we see that terms in $D^l (x^{-1} \NN(u))$, where $l\leq k+4$, are a linear combination with constant coefficients of terms of the form $x^{-1} v^{(1)} \times \dots \times v^{(5)}$ given by
\begin{alignat}{2}
    &v^{(1)} =u_{\beta},\nonumber\\
    &v^{(2)} = x^{\beta}D^{l_1}(u-u_0),\quad &&l_1\leq l+4 \le k+8,\nonumber\\
    &v^{(3)}, v^{(4)}, v^{(5)} \in \{D^{l_2}(u-u_0),u_0,1\},\quad &&l_2\leq \tfrac{l+4}{2} \le k+8. \nonumber
\end{alignat}
This gives, using that $\verti{x^{\beta}\cdot}_{\beta+\delta}=\verti{\cdot}_{\delta}$,
\begin{align}\label{eq_bound_v_1...v_5_2}
\verti{x^{-1} v^{(1)} \times\dots\times v^{(5)}}^p_{\beta+\delta}\lesssim |u_{\beta}|^p \verti{u-u_0}^p_{k+8,\delta} \big(1+\big(\vertii{u}_{k+8,I}^p+|u_0|^p\big)^3\big).
\end{align}
Here is where we deviate significantly from \cite[Lemma~8.1]{GGKO} and the choice $\tilde\delta > \delta$ becomes crucial at least for $p > 2$. We notice that
\[
\verti{u-u_0}_{k+8,\delta}^2 = \sum_{j = 0}^{k+8} \int_0^1 x^{-2\delta} (D^j(u-u_0))^2 \, \tfrac{\d x}{x} + \sum_{j = 0}^{k+8} \int_1^\infty x^{-2\delta} (D^j(u-u_0))^2 \, \tfrac{\d x}{x}.
\]
Choosing $\kappa := \tilde\delta-\delta$ in \eqref{lower_bound_l2eta} of Lemma~\ref{lemma_char_interpol}, we get
\[
\sum_{j = 0}^{k+8} \int_0^1 x^{-2\delta} (D^j(u-u_0))^2 \, \tfrac{\d x}{x} \lesssim_{k,\tilde k, \delta, \tilde \delta, p} \verti{u-u_0}_{\tilde k + 8 - \frac 4 p, \tilde\delta,p}^2
\]
provided $\tilde k \ge k+\tfrac 1 2+ \tfrac 3 p$, which is implied by \eqref{constraint-k}. Secondly, we get under the same constraint,
\begin{align*}
\sum_{j = 0}^{k+8} \int_1^\infty x^{-2\delta} (D^j(u-u_0))^2 \, \tfrac{\d x}{x} &\lesssim_\delta \verti{u_0}^2 + \sum_{j = 0}^{k+8} \int_1^\infty x^{-2\delta} (D^ju)^2 \, \tfrac{\d x}{x} \\
&\lesssim_{k,\tilde k, \delta, \tilde \delta, p} \verti{u_0}^2 + \verti{u}_{\tilde k + 8 - \frac 4 p, - \tilde\delta,p}^2,
\end{align*}
where $\kappa := - \delta - \tilde\delta$ was chosen in \eqref{lower_bound_l2eta} of Lemma~\ref{lemma_char_interpol}. Hence,
\begin{equation}\label{est-k8delta}
\verti{u-u_0}_{k+8,\delta} \lesssim_{k,\tilde k, \delta, \tilde \delta, p} \verti{u_0} + \verti{u}_{\tilde k + 8 - \frac 4 p, - \tilde\delta,p} + \verti{u-u_0}_{\tilde k + 8 - \frac 4 p, \tilde\delta,p}.
\end{equation}
Combining \eqref{eq_bound_v_1...v_5_1} and \eqref{eq_bound_v_1...v_5_2} gives, using $L^p$ bounds in time for the first term and supremum bounds in time for the rest of the terms,
\begin{align}
&\int_I t^{p\beta-1}\verti{\NN(u)}^p_{k+4,-1+\beta+\delta} \, \d t \nonumber \\
&\quad\stackrel{\eqref{est-k8delta}}{\lesssim_{k,\tilde k, \delta, \tilde \delta, p}} \int_I t^{p\beta-1} \big(\verti{u-u_0-u_{\beta}x^{\beta}}^p_{k+8,\beta+\delta} +|u_\beta|^p\big) \d t \nonumber \\
    &\quad\phantom{\stackrel{\eqref{est-k8delta}}{\lesssim_{k,\tilde k, \delta, \tilde \delta, p}}} \times \Big[\vertii{u-u_0}_{k+8,I}^p + \sup_{t\in I}\big(\verti{u_0}^p + \verti{u}_{\tilde k + 8 - \frac 4 p, - \tilde\delta,p}^p + \verti{u-u_0}_{\tilde k + 8 - \frac 4 p, \tilde\delta,p}^p\big)\Big] \nonumber\\
    &\quad\phantom{\stackrel{\eqref{est-k8delta}}{\lesssim_{k,\tilde k, \delta, \tilde \delta, p}}} \times \Big[1+\big(\vertii{u-u_0}_{k+8,I}^p+\sup_{t \in I} |u_0|^p\big)^3\Big]\nonumber\\
    &\quad \stackrel{\eqref{eq_ineq_coeff}}{\lesssim_{k,\tilde k, \delta, \tilde \delta, p}} \vertiii{u}_I^{2p}(1+\vertiii{u}_I^{3p}).\label{eq_bound_N>beta}
\end{align}
where Lemma~\ref{eq_ineq_coeff} and the constraint \eqref{constraint-k} were used.

\medskip

Equation \eqref{eq_bound_N(u)} now follows by adding  \eqref{eq_bound_N} and \eqref{eq_bound_N>beta} and taking the $p^{\text{th}}$ root.

\medskip

\textbf{Proof of estimate~\eqref{eq_bound_N(u1)-N(u2)}.} We only sketch the differences compared to the proof of estimate~\eqref{eq_bound_N(u)}. Because $\MMM_\mathrm{sym}$ is multi-linear, it holds that $\NN(u^{(1)})-\NN(u^{(2)})$ is a linear combination with constant coefficients of terms of the form
\begin{equation*}
x^{-1} \MMM_\mathrm{sym}(u^{(1)}-u^{(2)},w^{(2)},w^{(3)},w^{(4)},w^{(5)}),\quad w^{(2)} \in \{u^{(1)},u^{(2)}\}, \ \ w^{(3)}, w^{(4)}, w^{(5)} \in\{u^{(1)},u^{(2)},1\}.
\end{equation*}
For the subcritical part of the norm $\vertiii{\mathcal N(u^{(1)})-\mathcal N(u^{(2)})}_1$, which consists of norms with weight $\tfrac{1}{p}\pm\delta$ and $\beta-\delta$, the terms from above are rewritten as
\begin{align*}
    x^{-1} \MMM_\mathrm{sym}(w^{(1)},w^{(2)},w^{(3)},w^{(4)},w^{(5)}), \qquad &w^{(1)} \in\{u^{(1)}-u^{(1)}_0-(u^{(2)}-u^{(2)}_0),u^{(1)}_0-u^{(2)}_0\},\\
    &w^{(2)} \in \{u^{(1)}-u^{(1)}_0,u^{(2)}-u^{(2)}_0,u^{(1)}_0,u^{(2)}_0\},\\
    &w^{(3)}, w^{(4)}, w^{(5)} \in \{u^{(1)}-u^{(1)}_0,u^{(2)}-u^{(2)}_0,u^{(1)}_0,u^{(2)}_0,1\}.
\end{align*}
Note that before we could argue that the $w^{(1)}$ had to be non-constant in $x$, but this reasoning does not hold here. Hence, we have to work with both the non-constant and constant choices for $w^{(1)}$. First, consider the case that $w^{(1)} = u^{(1)} - u^{(1)}_0 - (u^{(2)} - u^{(2)}_0)$. In this case, $D^l(x^{-1} \NN(u^{(1)})-x^{-1} \NN(u^{(2)}))$ is, for $l\leq \tilde{k}+4$ or $l \le k+4$, a linear combination of terms of the form  $x^{-1} v^{(1)} \times\dots\times v^{(5)}$ with
\begin{alignat*}{2}
    v^{(1)} &= D^{l_1}\big(u^{(1)} -u^{(1)}_0-(u^{(2)}-u^{(2)}_0)\big), \quad &&l_1\leq l+4,\\
    v^{(2)} &\in \big\{D^{l_2}(u^{(1)} -u^{(1)}_0),D^{l_2}(u^{(2)}-u^{(2)}_0),u^{(1)}_0,u^{(2)}_0\big\},\quad &&l_2\leq l+4,\\
    v^{(3)}, v^{(4)}, v^{(5)} &\in \big\{D^{l_3}(u^{(1)} -u^{(1)}_0),D^{l_3}(u^{(2)}-u^{(2)}_0),u^{(1)}_0,u^{(2)}_0,1\big\},\quad &&l_3\leq\tfrac{l+4}{2}.
\end{alignat*}
Note that we can in fact assume either $l_1 \le \frac{l+4}{2}$ or $l_2 \le \frac{l+4}{2}$. Hence, up to relabelling $v^{(1)}$ and $v^{(2)}$, the situation is analogous to \eqref{v1-v5-sub}.

\medskip

In the case that $w^{(1)} = u^{(1)}_0-u^{(2)}_0$, $D^l(x^{-1}\NN(u^{(1)})-x^{-1}\NN(u^{(2)}))$ is for $l\leq k+4$ or $l \leq \tilde{k}+4$ a linear combination with constant coefficients of terms of the form $x^{-1} v^{(1)}\times\dots\times v_5$ with
\begin{alignat*}{2}
    v^{(1)} &=u^{(1)}_0-u^{(2)}_0,\\
     v^{(2)} &\in \big\{D^{l_2}(u^{(1)} -u^{(1)}_0),D^{l_2}(u^{(2)}-u^{(2)}_0)\big\},\quad &&l_2\leq l+4,\\
    v^{(3)}, v^{(4)}, v^{(5)} &\in \big\{D^{l_3}(u^{(1)} -u^{(1)}_0),D^{l_3}(u^{(2)} - u^{(2)}_0),u^{(1)}_0,u^{(2)}_0,1\big\},\quad &&l_3 \leq \tfrac{l+4}{2}.
\end{alignat*}
Note that in $x$ constant terms can be excluded because of \eqref{nonlinearity-structure-1} and \eqref{nonlinearity-structure-2}. Up to relabeling (interchanging $v^{(1)}$ and $v^{(2)}$), this term can be treated as those appearing in the context of \eqref{v1-v5-sub}. In summary, the same reasoning as in \eqref{eq_bound_N} applies and yields
\begin{align}
    &\int_I \Big(\verti{\NN(u^{(1)})-\NN(u^{(2)})}^p_{\tilde k+4,-1+\frac{1}{p}-\tilde\delta} + \verti{\NN(u^{(1)})-\NN(u^{(2)})}^p_{\tilde k+4,-1+\frac{1}{p}+\tilde\delta}\Big) \d t \nonumber \\
    &+ \int_I t^{p\beta-1}\verti{\NN(u^{(1)})-\NN(u^{(2)})}^p_{k+4,-1+\beta-\delta} \d t \nonumber\\
    &\lesssim_{k,\tilde k, \delta, \tilde\delta, p} \vertiii{u^{(1)}-u^{(2)}}^p \big(\vertiii{u^{(1)}}_I^{p}+\vertiii{u^{(2)}}_I^{p}\big) \big(1+\vertiii{u^{(1)}}_I^{3p}+\vertiii{u^{(2)}}_I^{3p}\big).\label{eq_bounds_nonlinearity}
\end{align}
\medskip

For the super-critical part of the norm, the part with weight $\beta+\delta$, we decompose
\begin{align*}
   x^{-1} \MMM_\mathrm{sym}(u^{(1)}_0-u^{(2)}_0,w^{(2)},w^{(3)},w^{(4)},w^{(5)}),\qquad &w^{(2)} \in \big\{u^{(1)} -u^{(1)}_0,u^{(2)} -u^{(2)}_0\big\},\\
    &w^{(3)},w^{(4)},w^{(5)} \in \big\{u^{(1)} -u^{(1)}_0,u^{(2)} -u^{(2)}_0,u^{(1)}_0, u^{(2)}_0,1\big\}
\end{align*}
as
\begin{align}
x^{-1} \MMM_\mathrm{sym}(u^{(1)}_0-u^{(2)}_0,w^{(2)},w^{(3)},&w^{(4)},w^{(5)}), \nonumber \\
 &w^{(2)}\in \big\{u^{(1)} -u^{(1)}_0-u^{(1)}_\beta x^\beta,u^{(2)} -u^{(2)}_0-u^{(2)}_\beta x^{\beta}\big\},\nonumber\\
&w^{(3)},w^{(4)},w^{(5)} \in \big\{u^{(1)} -u^{(1)}_0,u^{(2)} -u^{(2)}_0,u^{(1)}_0,u^{(2)}_0,1\big\},\label{eq_decomposition_nonlin_1}
\end{align}
and
\begin{align}
x^{-1} \MMM_\mathrm{sym}(u^{(1)}_0-u^{(2)}_0,w^{(2)},w^{(3)},w^{(4)},w^{(5)}), \quad &w^{(2)} \in \big\{u^{(1)}_\beta x^\beta,u^{(2)}_\beta x^\beta\big\},\nonumber\\
&w^{(3)} \in \big\{u^{(1)}-u^{(1)}_0,u^{(2)}-u^{(2)}_0\big\},\nonumber\\
&w^{(4)},w^{(5)} \in \big\{u^{(1)} -u^{(1)}_0,u^{(2)} -u^{(2)}_0,u^{(1)}_0,u^{(2)}_0,1\big\},\label{eq_decomposition_nonlin_2}
\end{align}
where in the last case $w^{(3)}$ is non-constant in $x$ because of \eqref{eq_reason_nonconstant}. For $D^l(x^{-1} \NN(u^{(1)})-x^{-1}\NN(u^{(2)}))$, $l\leq k+4$, the terms from \eqref{eq_decomposition_nonlin_1} lead to a linear combination with constant coefficients of terms of the form $x^{-1} v^{(1)}\times\dots\times v^{(5)}$ with
\begin{alignat}{2}
    &v^{(1)} = u^{(1)}_0 - u^{(2)}_0,\nonumber\\
    &v^{(2)} \in \big\{D^{l_1}(u^{(1)} -u^{(1)}_0-u^{(1)}_\beta x^\beta),D^{l_1}(u^{(2)} -u^{(2)}_0-u^{(1)}_\beta x^\beta)\big\}, \quad &&l_1\leq l+4,\nonumber\\
    &v^{(3)},v^{(4)},v^{(5)} \in \big\{D^{l_2}(u^{(1)}-u^{(1)}_0),D^{l_2}(u^{(2)}-u^{(2)}_0),u^{(1)}_0,u^{(2)}_0,1\big\}, \quad &&l_2\leq l+4.\nonumber
\end{alignat}
This is analogous to the decomposition in \eqref{v1-v5-beta-1} on interchanging $v^{(1)}$ and $v^{(2)}$.

\medskip

For $D^l(x^{-1}\NN(u^{(1)})-x^{-1}\NN(u^{(2)}))$, $l\leq k+4$, the terms from \eqref{eq_decomposition_nonlin_2} lead to a linear combination with constant coefficients of terms of the form $x^{-1} v^{(1)} \times\dots\times v^{(5)}$ with
\begin{alignat*}{2}
    &v^{(1)} = u^{(1)}_0-u^{(2)}_0,\\
    &v^{(2)} \in \big\{u^{(1)}_\beta,u^{(2)}_\beta\big\},\\
    &v^{(3)} \in \big\{x^{\beta}D^{l_1}(u^{(1)}-u^{(1)}_0),x^{\beta}D^{l_1}(u^{(2)}-u^{(2)}_0)\big\},\quad&&l_1 \leq l+4 \le k+8, \\
    &v^{(4)},v^{(5)} \in \big\{D^{l_2}(u^{(1)}-u^{(1)}_0),D^{l_2}(u^{(2)}-u^{(2)}_0),u^{(1)}_0,u^{(2)}_0,1\big\}, \quad &&l_2\leq \tfrac{l+4}{2} \le k+8.
\end{alignat*}
In this case, we can bound
\begin{align}\label{eq_bound_v_1...v_5_beta}
\verti{x^{-1} v^{(1)} \times\dots\times v^{(5)}}^p_{\beta+\delta}\lesssim& \verti{u_0^{(1)}-u_0^{(2)}}^p \big(|u^{(1)}_{\beta}|^p+|u^{(2)}_{\beta}|^p \big) \big(\verti{u^{(1)}-u^{(1)}_0}^p_{k+8,\delta}+\verti{u^{(2)}-u^{(2)}_0}^p_{k+8,\delta}\big) \nonumber \\
&\times \big(1+\big(\vertii{u^{(1)}}_{k+8,I}^p+|u^{(1)}_0|^p+\vertii{u^{(2)}}_{k+8,I}^p+|u^{(2)}_0|^p\big)^3\big).
\end{align}
Using \eqref{est-k8delta}, we then obtain
\begin{align*}
\int_I t^{p\beta-1}\verti{\NN(u^{(1)})-\NN(u^{(2)})}^p_{k+4,-1+\beta+\delta} \, \d t \lesssim_{k,\tilde k, \delta, \tilde \delta, p}&\vertiii{u^{(1)}-u^{(2)}}_I^p \big(\vertiii{u^{(1)}}_I^{p}+\vertiii{u^{(2)}}_I^{p}\big) \\
&\times \big(1+\vertiii{u^{(1)}}_I^{3p}+\vertiii{u^{(2)}}_I^{3p}\big),
\end{align*}
which in conjunction with \eqref{eq_bounds_nonlinearity} finishes the proof of \eqref{eq_bound_N(u1)-N(u2)}.
\end{proof}
%

\subsection{Proof of the Main Result}\label{sec_proof_main_result}
Here, we will prove Theorem~\ref{Main_thm}, where we will make use of the results from Proposition~\ref{prop_sol_Au=v} and Lemma~\ref{lemma_estimates_for_fixed_point}. Note that the proof is entirely standard but for the sake of completeness, we nonetheless give all details.

\begin{proof}[Proof of Theorem~\ref{Main_thm}]
In the proof, all estimates depend on $k$, $\tilde k$, $\delta$, $\tilde \delta$, and $p$.

\medskip

\textbf{Existence.}
For $\eps>0$ to be determined below let $u^{(0)} \in U^{(0)}(\tilde k, \tilde\delta, p)$ with $\vertiii{u^{(0)}}_0<\eps$ and define
\[
S := \big\{u \in U(k,\tilde k, \delta, \tilde\delta, p) \colon \vertiii{u} \le\eta, u|_{t=0}=u^{(0)}\big\}
\]
for $\eta>0$ to be determined in what follows. Let $\SSS$ be the solution operator constructed in Proposition~\ref{prop_sol_Au=v}. Then \eqref{eq_nonlin_cauchy} can be recast in the fixed-point problem
\begin{equation}\label{fixed}
u=\TT(u):=\SSS \NN(u).
\end{equation}
Our aim is to apply Banach's contraction-mapping principle to \eqref{fixed}. Therefore, for $u \in S$ by the maximal-regularity estimate \eqref{a-priori-perturbed} and the nonlinear estimate \eqref{eq_bound_N(u)} it follows that
\[
\vertiii{\TT(u)} =\vertiii{\SSS\NN(u)} \stackrel{\eqref{a-priori-perturbed}}{\lesssim}\vertiii{u^{(0)}}_0 +\vertiii{\NN(u)}_1 \stackrel{\eqref{eq_bound_N(u)}}{\lesssim} \vertiii{u^{(0)}}_0+ \max_{j=2,5}\vertiii{u}^j,
\]
so that $\vertiii{\TT(u)} =\vertiii{\SSS\NN(u)} \le C(\eps+\eta^2)$ for $\eta \le 1$ and $C < \infty$ only depending on $k$, $\tilde k$, $\delta$, $\tilde \delta$, and $p$. If we take $\eps := \eta^2$ and $\eta \le \frac{1}{2C}$, we infer that $\TT$ maps $S$ into itself. For showing the contraction property, suppose that $u^{(1)}, u^{(2)} \in S$. By \eqref{a-priori-perturbed} and \eqref{eq_bound_N(u1)-N(u2)} it follows that
\begin{align*}
    \vertiii{\TT(u^{(1)}|_I)-\TT(u^{(2)}|_I)}_I &= \vertiii{\SSS(\NN(u^{(1)}|_I)-\NN(u^{(2)}|_I))}_I \stackrel{\eqref{a-priori-perturbed}}{\lesssim} \vertiii{\NN(u^{(1)}|_I)-\NN(u^{(2)}|_I)}_{1,I},
\end{align*}
for $I = [0,T]$ with $T \in (0,\infty)$ or $I = [0,\infty)$,
so that by \eqref{eq_bound_N(u1)-N(u2)},
\begin{equation}\label{eq_bound_T_2}
\vertiii{\TT(u^{(1)}|_I)-\TT(u^{(2)}|_I)}_I \le \tilde C \max_{j=1,4}(\vertiii{u^{(1)}|_I}_I+\vertiii{u^{(2)}|_I}_I)^j \vertiii{u^{(1)}-u^{(2)}}_I
\end{equation}
for $\tilde C < \infty$ only depending on $k$, $\tilde k$, $\delta$, $\tilde \delta$, and $p$. Hence, $\TT \colon S \to S$ is a contraction for $\eta < \min\{\frac 1 2,\frac{1}{2\tilde C}\}$. Now applying the contraction-mapping principle gives existence and uniqueness of a solution $u$ to \eqref{fixed} and thus to \eqref{eq_nonlin_cauchy}.

\medskip

\textbf{Uniqueness.} Let $w \in U(k,\tilde k, \delta, \tilde\delta, p)$ be another solution to \eqref{eq_nonlin_cauchy}. Let $t^* \in (0,\infty)$ be the maximal time such that $w|_{[0,t^*)} = u|_{[0,t^*)}$. By continuity in time in $U^{(0)}(\tilde k, \tilde\delta, p)$ we must have $w(t^*) = u(t^*)$. By a time shift we may assume without loss of generality $t^* = 0$. By dominated convergence and in view of the definitions \eqref{eq_def_norm_u^0} and \eqref{sol_norm_i} of $\vertiii{\cdot}_0$ and $\vertiii{\cdot}_I$, respectively, we have $\vertiii{u|_I}_I \to \vertiii{u^{(0)}}_0\le \eta^2$ and $\vertiii{w|_I}_I \to \vertiii{u^{(0)}}_0\le \eta^2$ as $\verti{I} \searrow 0$. Since $\eta < \min\{\frac 1 2,\frac{1}{2\tilde C}\}$ and because $w = \TT(w)$ as well as $u = \TT(u)$, \eqref{eq_bound_T_2} entails $u|_{[0,T]} = w|_{[0,T]}$ for $T > 0$ sufficiently small, so that $t^*$ was not maximal.

\medskip

\textbf{Continuity and convergence.} The fact that $[0,\infty) \owns t \mapsto \vertiii{u(t)}_0$ is continuous and $\vertiii{u(t)}_0 \to 0$ as $t \to \infty$ follows from Lemma \ref{lem_coeff_est}.
\end{proof}
%

\section{Concluding remarks}\label{sec-conclusion}
We end with some concluding remarks on coercivity. In \cite[Lemma~5.2~(b)]{GGKO} it was found that a fourth-order polynomial operator $\mathfrak p(D) = \prod_{j = 1}^4 (D - \gamma_j)$ is coercive in the sense of \eqref{coercivity-pd} if and only if there exists a constant $L_\alpha > 0$ such that $\Re \mathfrak p(i\xi+\alpha) \ge L_\alpha$ for all $\xi \in \R$. This follows on commuting $x^{-\alp}$ with $D$, passing to $s := \ln x$, and using the Fourier transform and Plancherel's identity. As in the proof of \cite[Proposition~5.3]{GGKO} we find
\begin{align*}
\Re \mathfrak p(i\xi+\alpha) &= \kappa^2 - 2 a \kappa + b = (\kappa - a)^2 + b - a^2, & \text{where } \kappa &:= \xi^2 \ge 0, \\
a &:= \frac 1 2 \sum_{1 \le j < \ell \le 4} (\gamma_j-\alpha) (\gamma_\ell-\alpha), & b &:= \prod_{j = 1}^4 (\gamma_j-\alpha)
\end{align*}
Because of $\kappa \ge 0$ we infer that indeed $\Re \mathfrak p(i\xi+\alpha) \ge L_\alpha$ for all $\xi \in \R$ and an $L_\alpha > 0$ if and only if one of the following two conditions holds true:
\begin{enumerate}[label={(\Roman*)}]
\item\label{it:cond_i} $a \le 0$ and $b > 0$,
\item\label{it:cond_ii} $b > a^2$.
\end{enumerate}
The condition $b > 0$ in \ref{it:cond_i} can be reformulated to \eqref{eq_condition_coerc_1}, while the condition $a \le 0$ in \ref{it:cond_i} is equivalent to \eqref{eq_condition_coerc_2}. For solving $b > a^2$ in \ref{it:cond_ii}, we need to find the roots of the fourth-order polynomial $b - a^2$ in $\alpha$. Though explicit characterizations in terms of radicals of $\gamma_j$ can be found, these turn out to be quite involved. For the range $0 < n < 3/2$ (\S\ref{sssec:0n32} also applies in this range) with the polynomial $\mathfrak p(D)$ given by \eqref{eq_zeros_1}, we find with the software \emph{Mathematica} the algebraic expressions
\begin{equation}\label{alp_cond2}
\alpha = \frac{4n-2 + \sigma_1\sqrt{4 n^2-12 n+13 + \sigma_2\sqrt{16 n^4-96 n^3+120 n^2+72 n-119}}}{8 n-16},
\end{equation}
where $\sigma_1, \sigma_2 \in \{-1,+1\}$. The expressions for $3/2 < n < 3$ (cf.~\S\ref{sssec:32n3} and \eqref{eq_zeros_2} for the choice of $\gamma_j$) are much more involved, so that we omit them here. We find the factorization
\[
16 n^4-96 n^3+120 n^2+72 n-119 = 16 \big(n-(\tfrac 1 2 - \sqrt 2)\big) \big(n-(\tfrac 5 2 - \sqrt 2)\big) \big(n-(\tfrac 1 2 + \sqrt 2)\big) \big(n-(\tfrac 5 2 + \sqrt 2)\big).
\]
We have
\[
\tfrac 1 2 - \sqrt 2 < 0 < 1 < \tfrac 5 2 - \sqrt 2 < \tfrac 3 2 < \tfrac 1 2 + \sqrt 2 < \tfrac 5 2 + \sqrt 2.
\]
This entails that $\alpha$ given by \eqref{alp_cond2} is not real for $n < \tfrac 5 2 - \sqrt 2 > 1$ for any choice of $\sigma_1, \sigma_2 \in \{-1,+1\}$ since then $16 n^4-96 n^3+120 n^2+72 n-119 < 0$.

\medskip

The entire coercivity range is displayed in Figure~\ref{fig:coercivity-full}.
\begin{figure}[htp]
    \centering
    \includegraphics[width=\textwidth]{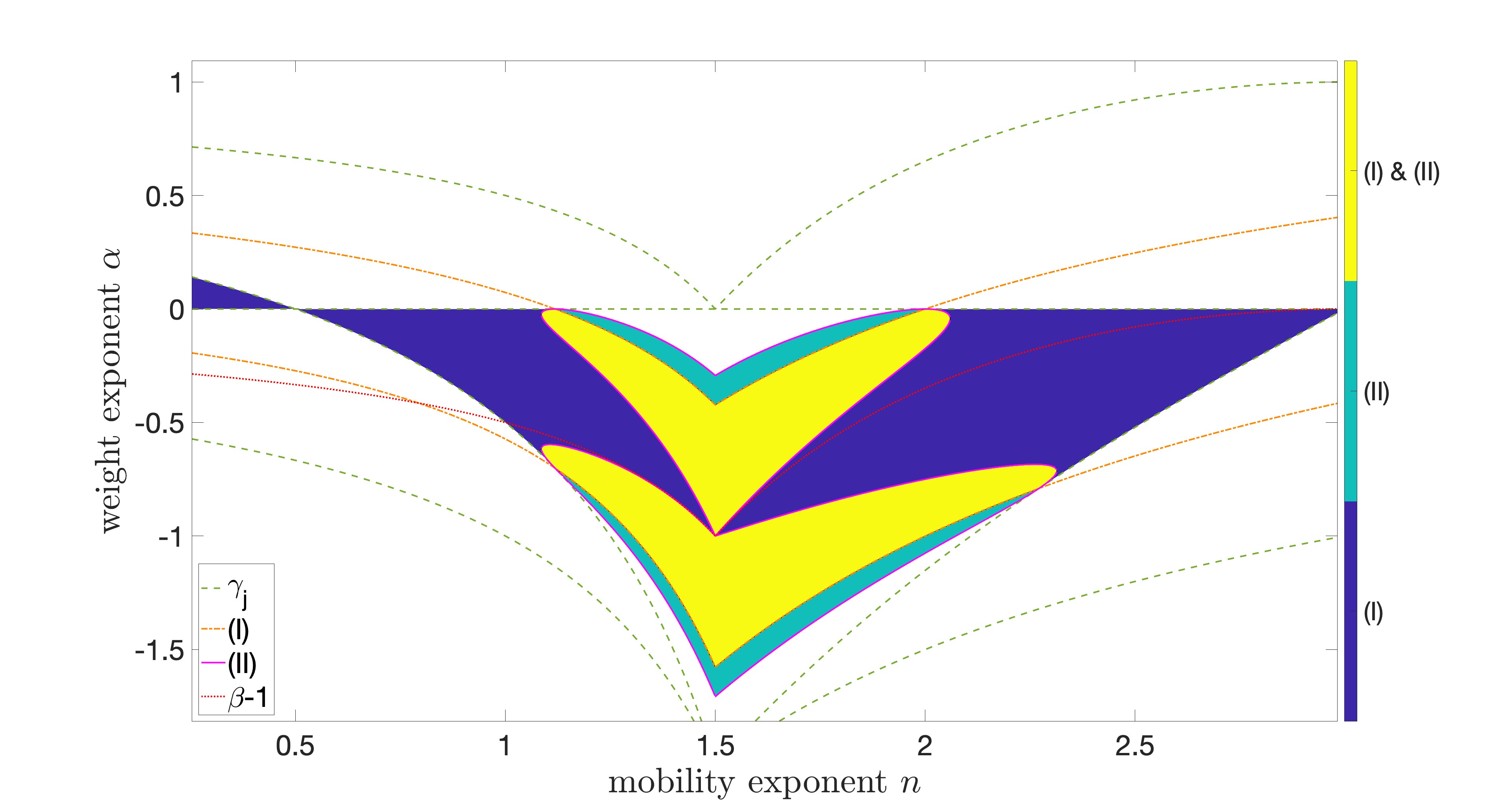}
    \caption{full coercivity range (color online). The areas and boundaries in the $n$-$\alpha$ plane are shown for which \ref{it:cond_i} or \ref{it:cond_ii} are valid and thus $\mathfrak p(D) = \mathfrak p_n(D)$ (given by \eqref{eq_def_p(D)} and \eqref{eq_zeros_12}) is coercive.}
    \label{fig:coercivity-full}
\end{figure}
We used the software \emph{Matlab} to generate the plot and evaluated the conditions \ref{it:cond_i} and \ref{it:cond_ii}. We have used the exact algebraic expressions \eqref{coerc_range} and \eqref{alp_cond2} to plot the contours and for $\frac 3 2 < n < 3$ we have solved for the roots $\alpha$ of $b = a^2$ numerically, cf.~\ref{it:cond_ii}. This verifies the enlarged coercivity range at least for $n < \frac 3 2$. We can deduce the following:
\begin{enumerate}
\item Since the criterion~\ref{it:cond_ii} cannot be satisfied for $n < \frac 5 2 - \sqrt 2 > 1$, this criterion does not lead to an improved range of validity of Theorem~\ref{Main_thm} to a larger interval for $n$ since the condition $\beta - 1$ being in the coercivity range of $\mathfrak p(D)$ requires $n > 1$ by \eqref{eq_zeros_1} and \eqref{coerc_range_n132}. However, again by \eqref{coerc_range_n132} we anticipate that we can prove well-posedness of \eqref{eq_nonlin_cauchy} for $\frac 1 2  < n \le 1$ by omitting all terms with time weights in \eqref{triple_bar_norms}, i.e., well-posedness without a regularization proof at the free boundary.
\item In Figure~\ref{fig:coercivity-full}, we recognize that we have coercivity for $\alpha$ in an interval containing $(-1,0)$ (the exact interval is $[-1,0)$ for $n = 2$, see \eqref{coerc_range_n323}) also for a range of values $n \in [2,2+\nu)$ with $\nu > 0$. Indeed, for $n = 2$ and $\alpha = -1$ we have $a = 0$ while $b = \frac 1 8 > 0$, so that indeed \ref{it:cond_ii} is valid in a neighborhood around $(n,\alpha) = (2,-1)$. By \eqref{coerc_range_n323} and \ref{it:cond_ii} we further infer that the upper boundary for coercivity is $\alpha = 0$ for $n \in [2,3)$. This indicates that the regularity result of \cite{G2016} presumably extends to $n \in [2,2+\nu)$.
\end{enumerate}
We also remark that conditions analogous to those provided in \ref{it:cond_i} and \ref{it:cond_ii} may be deduced at least for sixth- and eighth-order equations. However, explicit characterizations in terms of radicals such as in Lemma~\ref{lem_coercive} will for higher-order equations be more complicated and might in general not be feasible as the Galois groups of the corresponding polynomial equations might not be solvable for $m > 4$.

\bibliography{gnann_wisse_general_mobility_v3}

\begin{thebibliography}{10}

\bibitem{Alt}
H.W. Alt.
\newblock {\em Linear Functional Analysis}.
\newblock Springer, 2012.

\bibitem{AnsiniGiacomelli2004}
L.~Ansini and L.~Giacomelli.
\newblock Doubly nonlinear thin-film equations in one space dimension.
\newblock {\em Arch. Ration. Mech. Anal.}, 173(1):89--131, 2004.

\bibitem{BelgacemGnannKuehn2016}
F.~Ben Belgacem, M.~V. Gnann, and C.~Kuehn.
\newblock A dynamical systems approach for the contact-line singularity in
  thin-film flows.
\newblock {\em Nonlinear Anal.}, 144:204--235, 2016.

\bibitem{BerettaBertschDalPasso1995}
E.~Beretta, M.~Bertsch, and R.~Dal~Passo.
\newblock Nonnegative solutions of a fourth-order nonlinear degenerate
  parabolic equation.
\newblock {\em Arch. Rational Mech. Anal.}, 129(2):175--200, 1995.

\bibitem{Bergh_Loefstroem}
J.~Bergh and J.~L{\"o}fstr{\"o}m.
\newblock {\em Interpolation spaces}.
\newblock Springer-Verlag, 1976.

\bibitem{Bernis1996a}
F.~Bernis.
\newblock Finite speed of propagation and continuity of the interface for thin
  viscous flows.
\newblock {\em Adv. Differential Equations}, 1(3):337--368, 1996.

\bibitem{Bernis1996b}
F.~Bernis.
\newblock Finite speed of propagation for thin viscous flows when {$2\leq
  n<3$}.
\newblock {\em C. R. Acad. Sci. Paris S\'{e}r. I Math.}, 322(12):1169--1174,
  1996.

\bibitem{BernisFriedman1990}
F.~Bernis and A.~Friedman.
\newblock Higher order nonlinear degenerate parabolic equations.
\newblock {\em J. Differential Equations}, 83(1):179--206, 1990.

\bibitem{Bernis1992}
F.~Bernis, L.A. Peletier, and S.M. Williams.
\newblock {Source type solutions of a fourth order nonlinear degenerate
  parabolic equation}.
\newblock {\em Nonlinear Analysis, Theory, Methods \& Applications}, 1992.

\bibitem{BertozziPugh1996}
A.~L. Bertozzi and M.~Pugh.
\newblock The lubrication approximation for thin viscous films: regularity and
  long-time behavior of weak solutions.
\newblock {\em Comm. Pure Appl. Math.}, 49(2):85--123, 1996.

\bibitem{BertschDalPassoGarckeGruen1998}
M.~Bertsch, R.~Dal~Passo, H.~Garcke, and G.~Gr\"{u}n.
\newblock The thin viscous flow equation in higher space dimensions.
\newblock {\em Adv. Differential Equations}, 3(3):417--440, 1998.

\bibitem{BertschGiacomelliKarali2005}
M.~Bertsch, L.~Giacomelli, and G.~Karali.
\newblock Thin-film equations with ``partial wetting'' energy: existence of
  weak solutions.
\newblock {\em Phys. D}, 209(1-4):17--27, 2005.

\bibitem{Bonn2009}
D.~Bonn, J.~Eggers, J.~Indekeu, J.~Meunier, and E.~Rolley.
\newblock Wetting and spreading.
\newblock {\em Reviews of Modern Physics}, 81:739--805, 2009.

\bibitem{BW}
M.~Bowen and T.P. Witelski.
\newblock Pressure-dipole solutions of the thin-film equation.
\newblock {\em European Journal of Applied Mathematics}, 2019.

\bibitem{BringmannGiacomelliKnuepferOtto2016}
B.~Bringmann, L.~Giacomelli, H.~Kn{\"u}pfer, and F.~Otto.
\newblock Corrigendum to "smooth zero-contact-angle solutions to a thin-film
  equation around the steady state" [{J}. {D}ifferential {E}quations 245 (2008)
  1454-1506 doi: 10.1016/j.jde.2008.06.005].
\newblock {\em Journal of Differential Equations}, 261(2):1622--1635, 2016.

\bibitem{CarlenUlusoy2007}
E.~A. Carlen and S.~Ulusoy.
\newblock Asymptotic equipartition and long time behavior of solutions of a
  thin-film equation.
\newblock {\em J. Diff. Equ.}, 241(2):279--292, 2007.
\newblock All Open Access, Bronze Open Access.

\bibitem{CarlenUlusoy2014}
E.~A. Carlen and S.~Ulusoy.
\newblock Localization, smoothness, and convergence to equilibrium for a thin
  film equation.
\newblock {\em Discrete Contin. Dyn. Syst. - A}, 34(11):4537--4553, 2014.
\newblock All Open Access, Bronze Open Access, Green Open Access.

\bibitem{CarrilloToscani2002}
J.~A. Carrillo and G.~Toscani.
\newblock Long-time asymptotics for strong solutions of the thin film equation.
\newblock {\em Comm. Math. Phys.}, 225(3):551--571, 2002.

\bibitem{ConstantinDupontGoldsteinKadanoffShelleyZhou1993}
P.~Constantin, T.~F. Dupont, R.~E. Goldstein, L.~P. Kadanoff, M.~J. Shelley,
  and S.-M. Zhou.
\newblock Droplet breakup in a model of the {H}ele-{S}haw cell.
\newblock {\em Phys. Rev. E (3)}, 47(6):4169--4181, 1993.

\bibitem{ConstantinElgindiNguyenVicol2018}
P.~Constantin, T.~Elgindi, H.~Nguyen, and V.~Vicol.
\newblock On singularity formation in a {H}ele-{S}haw model.
\newblock {\em Comm. Math. Phys.}, 363(1):139--171, 2018.

\bibitem{DalPassoGiacomelliGruen2001}
R.~Dal~Passo, L.~Giacomelli, and G.~Gr\"{u}n.
\newblock A waiting time phenomenon for thin film equations.
\newblock {\em Ann. Scuola Norm. Sup. Pisa Cl. Sci. (4)}, 30(2):437--463, 2001.

\bibitem{DalPassoGiacomelliShishkov2001}
R.~Dal~Passo, L.~Giacomelli, and A.~Shishkov.
\newblock The thin film equation with nonlinear diffusion.
\newblock {\em Comm. Partial Differential Equations}, 26(9-10):1509--1557,
  2001.

\bibitem{DeNittiFischer2022}
N.~De~Nitti and J.~Fischer.
\newblock Sharp criteria for the waiting time phenomenon in solutions to the
  thin-film equation.
\newblock {\em Comm. Partial Differential Equations}, 47(7):1394--1434, 2022.

\bibitem{DeSimon1964}
L.~De~Simon.
\newblock Un'applicazione della teoria degli integralisingolari allo studio
  delle equazioni differenzialilineari astratte del primo ordine.
\newblock {\em Rendiconti del Seminario Matematico della Universit{\`a} Di
  Padova}, 34:205--223, 1964.

\bibitem{Degtyarev2017}
S.~Degtyarev.
\newblock Classical solvability of the multidimensional free boundary problem
  for the thin film equation with quadratic mobility in the case of partial
  wetting.
\newblock {\em Discrete and Continuous Dynamical Systems- Series A}, 37(7):3625
  -- 3699, 2017.

\bibitem{Esselborn2016}
E.~Esselborn.
\newblock Relaxation rates for a perturbation of a stationary solution to the
  thin-film equation.
\newblock {\em SIAM J. Math. Anal.}, 48(1):349--396, 2016.

\bibitem{Fischer2013}
J.~Fischer.
\newblock Optimal lower bounds on asymptotic support propagation rates for the
  thin-film equation.
\newblock {\em J. Differential Equations}, 255(10):3127--3149, 2013.

\bibitem{Fischer2014}
J.~Fischer.
\newblock Upper bounds on waiting times for the thin-film equation: the case of
  weak slippage.
\newblock {\em Arch. Ration. Mech. Anal.}, 211(3):771--818, 2014.

\bibitem{Fischer2016}
J.~Fischer.
\newblock Behaviour of free boundaries in thin-film flow: the regime of strong
  slippage and the regime of very weak slippage.
\newblock {\em Ann. Inst. H. Poincar\'{e} C Anal. Non Lin\'{e}aire},
  33(5):1301--1327, 2016.

\bibitem{Gennes1985}
P.-G.~de Gennes.
\newblock Wetting: statics and dynamics.
\newblock {\em Reviews of Modern Physics}, 57:827--863, 1985.

\bibitem{GGKO}
L.~Giacomelli, M.V. Gnann, H.~Kn{\"u}pfer, and F.~Otto.
\newblock {Well-posedness for the Navier-slip thin-film equation in the case of
  complete wetting}.
\newblock {\em Journal of Differential Equations}, 257(1):15--81, 2014.

\bibitem{GiacomelliGnannOtto2013}
L.~Giacomelli, M.V. Gnann, and F.~Otto.
\newblock {Rigorous asymptotics of traveling-wave solutions to the thin-film
  equation and Tanner's law}.
\newblock {\em Nonlinearity}, 29(9):2497--2536, 2016.

\bibitem{GiacomelliGruen2006}
L.~Giacomelli and G.~Gr\"{u}n.
\newblock Lower bounds on waiting times for degenerate parabolic equations and
  systems.
\newblock {\em Interfaces Free Bound.}, 8(1):111--129, 2006.

\bibitem{GiacomelliKnuepfer2010}
L.~Giacomelli and H.~Kn{\"u}pfer.
\newblock A free boundary problem of fourth order: Classical solutions in
  weighted h{\"o}lder spaces.
\newblock {\em Communications in Partial Differential Equations},
  35(11):2059--2091, 2010.

\bibitem{GiacomelliKnuepferOtto2008}
L.~Giacomelli, H.~Kn{\"u}pfer, and F.~Otto.
\newblock Smooth zero-contact-angle solutions to a thin-film equation around
  the steady state.
\newblock {\em Journal of Differential Equations}, 245(6):1454--1506, 2008.

\bibitem{GiacomelliOtto2003}
L.~Giacomelli and F.~Otto.
\newblock Rigorous lubrication approximation.
\newblock {\em Interfaces Free Bound.}, 5(4):483--529, 2003.

\bibitem{GiacomelliShishkov2005}
L.~Giacomelli and A.~Shishkov.
\newblock Propagation of support in one-dimensional convected thin-film flow.
\newblock {\em Indiana Univ. Math. J.}, 54(4):1181--1215, 2005.

\bibitem{Gilbert1972}
J.~E. Gilbert.
\newblock {"Interpolation between weighted Lp-spaces"}.
\newblock {\em Arkiv f{\"o}r Matematik}, 10(1):235 -- 249, 1972.

\bibitem{Gnann2015}
M.~V. Gnann.
\newblock Well-posedness and self-similar asymptotics for a thin-film equation.
\newblock {\em SIAM Journal on Mathematical Analysis}, 47(4):2868 -- 2902,
  2015.

\bibitem{G2016}
M.V. Gnann.
\newblock {On the Regularity for the Navier-Slip Thin-Film Equation in the
  Perfect Wetting Regime}.
\newblock {\em Archive for Rational Mechanics and Analysis}, 222:1285--1337,
  2016.

\bibitem{GnannPetrache}
M.V. Gnann and M.~Petrache.
\newblock {The Navier-slip thin-film equation for 3D fluid films: existence and
  uniqueness}.
\newblock {\em Journal of Differential Equations}, 265(11):5832--5958, 2018.

\bibitem{Greenspan1978}
H.P. Greenspan.
\newblock On the motion of a small viscous droplet that wets a surface.
\newblock {\em Journal of Fluid Mechanics}, 84(1):125 -- 143, 1978.

\bibitem{Gruen2003}
G.~Gr\"{u}n.
\newblock Droplet spreading under weak slippage: a basic result on finite speed
  of propagation.
\newblock {\em SIAM J. Math. Anal.}, 34(4):992--1006, 2003.

\bibitem{Gruen2004a}
G.~Gr\"{u}n.
\newblock Droplet spreading under weak slippage---existence for the {C}auchy
  problem.
\newblock {\em Comm. Partial Differential Equations}, 29(11-12):1697--1744,
  2004.

\bibitem{Gruen2004b}
G.~Gr\"{u}n.
\newblock Droplet spreading under weak slippage: the waiting time phenomenon.
\newblock {\em Ann. Inst. H. Poincar\'{e} C Anal. Non Lin\'{e}aire},
  21(2):255--269, 2004.

\bibitem{GuentherProkert2008}
M.~G{\"u}nther and G.~Prokert.
\newblock A justification for the thin film approximation of {S}tokes flow with
  surface tension.
\newblock {\em Journal of Differential Equations}, 245(10):2802 -- 2845, 2008.

\bibitem{Haasse}
M.~Haase.
\newblock {Identification of Some Real Interpolation Spaces}.
\newblock {\em Proceedings of the American Mathematical Society},
  134(8):2349--2358, 2006.

\bibitem{Huh1971}
C.~Huh and S.C. Scriven.
\newblock {Hydrodynamic model of steady movement of a solid/liquid/fluid
  contact line}.
\newblock {\em Journal of Colloid Interface Science}, 35(1):85--101, 1971.

\bibitem{HulshofShishkov1998}
J.~Hulshof and A.~E. Shishkov.
\newblock The thin film equation with {$2\leq n<3$}: finite speed of
  propagation in terms of the {$L^1$}-norm.
\newblock {\em Adv. Differential Equations}, 3(5):625--642, 1998.

\bibitem{John2015}
D.~John.
\newblock {On uniqueness of weak solutions for the thin-film equation}.
\newblock {\em Journal of Differential Equations}, 259(8):4122--4171, 2015.

\bibitem{Kienzler2016}
C.~Kienzler.
\newblock Flat fronts and stability for the porous medium equation.
\newblock {\em Comm. Partial Differential Equations}, 41(12):1793--1838, 2016.

\bibitem{Knuepfer2011}
H.~Kn{\"u}pfer.
\newblock Well-posedness for the navier slip thin-film equation in the case of
  partial wetting.
\newblock {\em Communications on Pure and Applied Mathematics}, 64(9):1263 --
  1296, 2011.

\bibitem{Knuepfer2015}
H.~Kn{\"u}pfer.
\newblock Well-posedness for a class of thin-film equations with general
  mobility in the regime of partial wetting.
\newblock {\em Archive for Rational Mechanics and Analysis}, 218(2):1083 --
  1130, 2015.

\bibitem{Knuepfer2023}
H.~Kn{\"u}pfer.
\newblock Erratum to: Well-posedness for a class of thin-film equations with
  general mobility in the regime of partial wetting.
\newblock {\em Archive for Rational Mechanics and Analysis}, in preparation,
  2023.

\bibitem{KnuepferMasmoudi2013}
H.~Kn{\"u}pfer and N.~Masmoudi.
\newblock Well-posedness and uniform bounds for a nonlocal third order
  evolution operator on an infinite wedge.
\newblock {\em Communications in Mathematical Physics}, 320(2):395--424, 2013.

\bibitem{KnuepferMasmoudi2015}
H.~Kn{\"u}pfer and N.~Masmoudi.
\newblock Darcy's flow with prescribed contact angle: Well-posedness and
  lubrication approximation.
\newblock {\em Archive for Rational Mechanics and Analysis}, 218(2):589--646,
  2015.

\bibitem{Koch1999}
H.~Koch.
\newblock {\em Non-Euclidean Singular Intergrals and the Porous Medium
  Equation}.
\newblock {H}abilitation thesis, Ruprecht-Karls-Universit\"at Heidelberg, 1999.

\bibitem{Lunardi}
A.~Lunardi.
\newblock {\em Analytic Semigroups and Optimal Regularity in Parabolic
  Problems}.
\newblock Birkh{\"a}user, 1995.

\bibitem{MajdoubMasmoudiTayachi2021}
M.~Majdoub, N.~Masmoudi, and S.~Tayachi.
\newblock Relaxation to equilibrium in the one-dimensional thin-film equation
  with partial wetting and linear mobility.
\newblock {\em Comm. Math. Phys.}, 385(2):837--857, 2021.

\bibitem{Masmoudi2011}
N.~Masmoudi.
\newblock About the {Hardy} inequality.
\newblock In {\em An invitation to mathematics. From competitions to research.
  With a preface by G\"unter M. Ziegler}, pages 165--180. Berlin: Springer,
  2011.

\bibitem{MatiocProkert2012}
B.-V. Matioc and G.~Prokert.
\newblock Hele-{S}haw flow in thin threads: A rigorous limit result.
\newblock {\em Interfaces and Free Boundaries}, 14(2):205 -- 230, 2012.

\bibitem{MatthesMcCannSavare2009}
D.~Matthes, R.~J. McCann, and G.~Savaré.
\newblock A family of nonlinear fourth order equations of gradient flow type.
\newblock {\em Commun. Partial Differ. Equ.}, 34(11):1352--1397, 2009.
\newblock All Open Access, Green Open Access.

\bibitem{Mellet2015}
A.~Mellet.
\newblock The thin film equation with non-zero contact angle: a singular
  perturbation approach.
\newblock {\em Comm. Partial Differential Equations}, 40(1):1--39, 2015.

\bibitem{Navier1823}
C.L.M.H. Navier.
\newblock M{\'e}moire sur les lois du mouvement des fluides.
\newblock {\em M{\'e}moires de l'Acad{\'e}mie Royale des Sciences de l'Institut
  de France}, 6:389--440, 1823.

\bibitem{Oron1997}
A.~Oron, S.H. Davis, and S.G. Bankoff.
\newblock Long-scale evolution of thin liquid films.
\newblock {\em Reviews of Modern Physics}, 69(3):931--980, 1997.

\bibitem{Otto1998}
F.~Otto.
\newblock Lubrication approximation with prescribed nonzero contact angle.
\newblock {\em Comm. Partial Differential Equations}, 23(11-12):2077--2164,
  1998.

\bibitem{Seis2018}
C.~Seis.
\newblock The thin-film equation close to self-similarity.
\newblock {\em Analysis and PDE}, 11(5):1303 -- 1342, 2018.

\bibitem{SeisWinkler2024}
C.~Seis and D.~Winkler.
\newblock Invariant manifolds for the thin film equation.
\newblock {\em Arch. Ration. Mech. Anal.}, 248(2):Paper No. 27, 64, 2024.

\bibitem{Young1805}
T.~Young.
\newblock An essay on the cohesion of fluids.
\newblock {\em Philosophical Transactions of the Royal Society of London},
  95:65--87, 1805.

\end{thebibliography}
\bibliographystyle{plain}

\end{document}